\documentclass[11pt]{article}

\usepackage[T1]{fontenc}      
\usepackage[utf8]{inputenc}   

\newcommand{\newOld}[1]{#1}

\usepackage{hyperref}
\usepackage{amsmath,amsthm} 
\usepackage{amssymb,mathrsfs} 
\usepackage{amssymb}
\usepackage{amsfonts}
\usepackage{upgreek}
\usepackage{nicefrac}
\usepackage{geometry}
\usepackage{graphicx}
\usepackage{color}
\usepackage{algorithm2e}
\usepackage{stmaryrd}
\DeclareMathAlphabet{\mathpzc}{OT1}{pzc}{m}{it}
\usepackage[inline]{enumitem}
\usepackage{url}
\usepackage{graphicx} 
\usepackage{lmodern} 
\usepackage{xcolor}
\usepackage{bbm,bm}
\usepackage{comment}

\usepackage{ifthen}
\usepackage{xargs}
\usepackage{caption}
\usepackage{subcaption}

\usepackage{todonotes}

\hypersetup{
  colorlinks = true,
  citecolor = blue,
}

\usepackage{aliascnt}
\usepackage{cleveref}
\usepackage{autonum}
\makeatletter
\newtheorem{theorem}{Theorem}
\crefname{theorem}{theorem}{Theorems}
\Crefname{Theorem}{Theorem}{Theorems}

\newaliascnt{lemma}{theorem}
\newtheorem{lemma}[lemma]{Lemma}
\aliascntresetthe{lemma}
\crefname{lemma}{lemma}{lemmas}
\Crefname{Lemma}{Lemma}{Lemmas}

\newaliascnt{corollary}{theorem}
\newtheorem{corollary}[corollary]{Corollary}
\aliascntresetthe{corollary}
\crefname{corollary}{corollary}{corollaries}
\Crefname{Corollary}{Corollary}{Corollaries}

\newaliascnt{proposition}{theorem}
\newtheorem{proposition}[proposition]{Proposition}
\aliascntresetthe{proposition}
\crefname{proposition}{proposition}{propositions}
\Crefname{Proposition}{Proposition}{Propositions}

\newaliascnt{definition}{theorem}

\aliascntresetthe{definition}
\crefname{definition}{definition}{definitions}
\Crefname{Definition}{Definition}{Definitions}

\newaliascnt{definitionProposition}{theorem}

\aliascntresetthe{definitionProposition}
\crefname{Proposition and Definition}{Proposition and Definition}{Proposition and Definition}
\Crefname{Proposition and Definition}{Proposition and Definition}{Proposition and Definition}

\newaliascnt{remark}{theorem}
\newtheorem{remark}[remark]{Remark}
\aliascntresetthe{remark}
\crefname{remark}{remark}{remarks}
\Crefname{Remark}{Remark}{Remarks}

\crefname{example}{example}{examples}
\Crefname{Example}{Example}{Examples}

\crefname{figure}{figure}{figures}
\Crefname{Figure}{Figure}{Figures}

\newtheorem{assumption}{\textbf{H}\hspace{-3pt}}
\Crefname{assumption}{\textbf{H}\hspace{-3pt}}{\textbf{H}\hspace{-3pt}}
\crefname{assumption}{\textbf{H}}{\textbf{H}}

\Crefname{assumptionAO}{\textbf{AO}\hspace{-3pt}}{\textbf{AO}\hspace{-3pt}}
\crefname{assumptionAO}{\textbf{AO}}{\textbf{AO}}

\Crefname{assumptionL}{\textbf{L}\hspace{-3pt}}{\textbf{L}\hspace{-3pt}}
\crefname{assumptionL}{\textbf{L}}{\textbf{L}}

\Crefname{assumptionA}{\textbf{A}\hspace{-3pt}}{\textbf{A}\hspace{-3pt}}
\crefname{assumptionA}{\textbf{A}}{\textbf{A}}

\Crefname{assumptionG}{\textbf{G}\hspace{-3pt}}{\textbf{G}\hspace{-3pt}}
\crefname{assumptionG}{\textbf{G}}{\textbf{G}}

\allowdisplaybreaks

\def\bfM{\mathbf{M}}

\def\mch{\mathcal{H}}
\def\mci{\mathcal{I}}

\def\rset{\mathbb{R}}
\def\rsetd{\mathbb{R}^d}
\def\rsetdd{\mathbb{R}^{2d}}

\def\nset{\mathbb{N}}
\def\nsets{\mathbb{N}^*}

\def\rml{\mathrm{L}}
\def\rmd{\mathrm{d}}

\def\rme{\mathrm{e}}

\def\rmN{\mathrm{N}}

\def\rmC{\mathrm{C}}

\newcommand{\cco}{\llbracket}
\newcommand{\ccf}{\rrbracket}
\newcommand{\po}{\left(}
\newcommand{\pf}{\right)}
\newcommand{\co}{\left[}
\newcommand{\cf}{\right]}
\newcommand{\R}{\mathbb R}

\newcommand{\dd}{\mathrm{d}}

\newcommand{\na}{\nabla}

\newcommand{\gabriel}[1]{\textcolor{brown}{#1}}

\newcommandx{\functionspace}[2][1=+]{\mathbb{F}_{#1}(#2)}

\newcommandx{\VarDeux}[3][3=]{\operatorname{Var}^{#3}_{#1}\left\{#2 \right\}}

\newcommand{\LeftEqNo}{\let\veqno\@@leqno}

\newcommand{\floor}[1]{\left\lfloor #1 \right\rfloor}
\newcommand{\ceil}[1]{\left\lceil #1 \right\rceil}

\newcommand{\N}{\ensuremath{\mathbb{N}}}

\newcommand{\PE}{\mathbb{E}}

\newcommand{\abs}[1]{\left\vert #1 \right\vert}
\newcommand{\absLigne}[1]{\vert #1 \vert}

\newcommandx{\Vnorm}[2][1=V]{\| #2 \|_{#1}}
\newcommandx{\VnormEq}[2][1=V]{\left\| #2 \right\|_{#1}}
\newcommandx{\norm}[2][1=]{\ifthenelse{\equal{#1}{}}{\left\vert #2 \right\vert}{\left\vert #2 \right\vert^{#1}}}
\newcommandx{\normLigne}[2][1=]{\ifthenelse{\equal{#1}{}}{\Vert #2 \Vert}{\Vert #2\Vert^{#1}}}

\newcommand{\parenthese}[1]{\left(#1 \right)}

\newcommand{\parentheseDeux}[1]{\left[ #1 \right]}
\newcommand{\defEns}[1]{\left\lbrace #1 \right\rbrace }
\newcommand{\defEnsLigne}[1]{\lbrace #1 \rbrace }

\newcommand{\ps}[2]{\left\langle#1,#2 \right\rangle}
\newcommand{\psLigne}[2]{\langle#1,#2 \rangle}

\newcommandx\probaMarkovTilde[2][2=]
{\ifthenelse{\equal{#2}{}}{{\widetilde{\mathbb{P}}_{#1}}}{\widetilde{\mathbb{P}}_{#1}\left[ #2\right]}}

\newcommand{\bigO}{\ensuremath{\mathcal O}}

\newcommand{\plusinfty}{+\infty}

\def\ie{\textit{i.e.}}

\def\eqsp{\;}
\newcommand{\coint}[1]{\left[#1\right)}

\newcommand{\ooint}[1]{\left(#1\right)}

\newcommand{\ocintLigne}[1]{(#1]}

\newcommand{\tint}{\mathtt{T}}
\newcommand{\gammaint}{\gamma}

\newcommandx{\weight}[2][2=n]{\omega_{#1,#2}^N}

\newcommand{\dps}{\displaystyle}

\newcommandx\sequence[3][2=,3=]
{\ifthenelse{\equal{#3}{}}{\ensuremath{\{ #1_{#2}\}}}{\ensuremath{\{ #1_{#2}, \eqsp #2 \in #3 \}}}}
\newcommandx\sequenceD[3][2=,3=]
{\ifthenelse{\equal{#3}{}}{\ensuremath{\{ #1_{#2}\}}}{\ensuremath{( #1)_{ #2 \in #3} }}}

\newcommandx{\sequencen}[2][2=n\in\N]{\ensuremath{\{ #1_n, \eqsp #2 \}}}
\newcommandx\sequenceDouble[4][3=,4=]
{\ifthenelse{\equal{#3}{}}{\ensuremath{\{ (#1_{#3},#2_{#3}) \}}}{\ensuremath{\{  (#1_{#3},#2_{#3}), \eqsp #3 \in #4 \}}}}
\newcommandx{\sequencenDouble}[3][3=n\in\N]{\ensuremath{\{ (#1_{n},#2_{n}), \eqsp #3 \}}}

\newcommand{\opnorm}[1]{{\left\vert\kern-0.25ex\left\vert\kern-0.25ex\left\vert #1 
    \right\vert\kern-0.25ex\right\vert\kern-0.25ex\right\vert}}

\def\bfe{\mathbf{e}}

\newcommandx{\CPE}[3][1=]{{\mathbb E}_{#1}\left[\left. #2 \, \middle \vert \, #3 \right. \right]} 

\newcommandx{\CPELigne}[3][1=]{{\mathbb E}_{#1}[\left. #2 \,  \vert \, #3 \right. ]} 

\newcommandx{\CPVar}[3][1=]{\mathrm{Var}^{#3}_{#1}\left\{ #2 \right\}}
\newcommand{\CPP}[3][]
{\ifthenelse{\equal{#1}{}}{{\mathbb P}\left(\left. #2 \, \right| #3 \right)}{{\mathbb P}_{#1}\left(\left. #2 \, \right | #3 \right)}}

\newcommandx{\osc}[2][1=]{\mathrm{osc}_{#1}(#2)}

\def\IdM{\operatorname{I}_d}

\def\Lone{\mathrm{L}^1}

\def\a{a}
\def\b{b}
\def\c{c}
\def\e{e}

\def\y{y}

\def\bE{\bar{E}}

\def\bT{\bar{T}}

\def\Mt{\tilde{M}}
\def\tM{\Mt}

\def\tT{\tilde{T}}

\def\bE{\bar{E}}
\def\bT{\bar{T}}

\def\rmD{\mathrm{D}}

\newcommand\coupling[2]{\Gamma(\mu,\nu)}

\newcommand{\fracmm}[2]{ #1  / #2 }

\renewcommand{\geq}{\geqslant}
\renewcommand{\leq}{\leqslant}

\def\Leb{\mathrm{Leb}}

\def\vareps{\varepsilon}

\def\bdelta{\bar{\delta}}

\def\varphibf{\boldsymbol{\varphi}}

\def\bfm{\mathbf{m}}

\def\mtt{\mathtt{m}}
\def\Mtt{\mathtt{M}}

\def\a{\mathbf{a}}

\newcommandx{\wasserstein}[3][1=\distance,3=]{\mathscr{W}_{#1}^{#3}\left(#2\right)}
\newcommandx{\wassersteinLigne}[3][1=\distance,3=]{\mathscr{W}_{#1}^{#3}(#2)}
\newcommandx{\wassersteinD}[1][1=\distance]{\mathscr{W}_{#1}}
\newcommandx{\wassersteinDLigne}[1][1=\distance]{\mathscr{W}_{#1}}

\newcommand{\txts}{\textstyle}

\def\rmN{\mathrm{N}}

\def\gauss{\mathrm{N}}

\def\Pens{\mathcal{P}}
\def\bfX{\mathbf{X}}

\def\bfV{\mathbf{V}}
\def\bfP{\mathbf{P}}
\def\rmK{\mathrm{K}}
\def\rmP{\mathrm{P}}
\def\rmV{\mathrm{V}}
\def\CLS{C_{\mathrm{LS}}}

\def\refl{\mathrm{R}}

\def\nq{d_0}
\def\Uq{U^{(q)}}
\def\Wq{W^{(q)}}
\def\bfo{\mathbf{o}}
\def\tbfo{\tilde{\mathbf{o}}}
\def\Vo{\mathfrak{V}_{\bfo}}
\newcommandx{\Voi}[1][1=i]{\mathfrak{V}_{\bfo,#1}}
\newcommandx{\Vlyapc}[2][1=\bfo,2=i]{\mathfrak{V}_{#1,#2}}

\def\Wlyap{\mathfrak{W}}
\def\bfomega{\boldsymbol{\omega}}

\def\Wol{\Wlyap_{\bfomega}^{(\ell)}}
\def\Vlyol{\Vo^{(\ell)}}
\def\WoU{\Wlyap_{\bfomega}^{(1)}}
\newcommandx{\Woi}[1][1=i]{\Wlyap_{\bfomega,#1}}
\newcommandx{\Wlyapc}[2][1=\bfomega,2=i]{\Wlyap_{#1,#2}}

\renewcommand{\iint}[2]{\{#1,\ldots,#2\}}

\def\a{a}
\def\b{b}
\def\c{c}

\def\e{e}
\def\f{f}

\def\teta{\tilde{\eta}}

\def\rmPo{\rmP_{\bfomega}}

\def\tbfm{\tilde{\bfm}}

\def\tC{\tilde{C}}
\def\bC{\bar{C}}
\def\tdelta{\tilde{\delta}}

\def\nC{C}
\def\espilon{\epsilon}
\def\bareta{\bar{\eta}}
\def\bD{\bar{D}}

\def\bfA{\mathbf{A}}
\def\bfB{\mathbf{B}}
\def\bfS{\mathbf{S}}
\def\ta{\tilde{a}}
\def\tkappa{\tilde{\kappa}}
\def\rmQ{\mathrm{Q}}
\def\Deltabf{\boldsymbol{\Delta}}

\usepackage{xcolor}

\title{Second order quantitative bounds for unadjusted generalized Hamiltonian Monte Carlo}
\author{Evan Camrud, Alain Durmus, Pierre Monmarché, Gabriel Stoltz}

\begin{document}

\maketitle

\begin{abstract}~ This paper provides a convergence analysis for
  generalized Hamiltonian Monte Carlo samplers, a family of Markov
  Chain Monte Carlo methods based on leapfrog integration of
  Hamiltonian dynamics and kinetic Langevin diffusion, that
  encompasses the unadjusted Hamiltonian Monte Carlo method.  Assuming
  that the target distribution $\pi$ satisfies a log-Sobolev inequality and
  mild conditions on the corresponding potential function, we
  establish quantitative bounds on the relative entropy of the
  iterates defined by the algorithm, with respect to $\pi$.
  Our
  approach is based on a perturbative and discrete version of the 
  modified entropy method developed to establish hypocoercivity for the continuous-time kinetic Langevin process.   As a corollary of our main result, we are able to derive complexity bounds for the class of algorithms at hand. In particular, we show that the total number of iterations to achieve a target accuracy $\varepsilon >0$ is of order $d/\varepsilon^{1/4}$, where $d$ is the dimension of the problem. This result can be further improved in the case of weakly interacting mean field potentials, for which we find a total number of iterations of order $(d/\varepsilon)^{1/4}$.
 \end{abstract}

\section{Introduction}\label{sec:intro}

We consider in this paper the problem of sampling from a target distribution $\pi$. This problem is ubiquitous in various fields such as statistical physics \cite{lelievre2016partial}, statistics \cite{gelman2013bayesian}, and machine learning \cite{andrieu2003introduction}. However, in most applications, the distribution $\pi$ has a density with respect to a dominating measure known up to an intractable multiplicative constant. Markov chain Monte Carlo methods are now a family of popular algorithms for solving this problem. They consist in designing a Markov chain associated with a Markov kernel for which $\pi$ is an invariant distribution. One of the best known MCMC instances is the family of Metropolis--Hastings algorithms \cite{metropolis1953equation,hastings1970monte}. In the case where the target distribution has a smooth and positive density with respect to the Lebesgue measure on $\rset^d$, denoted by $\pi$ still, another class of MCMC algorithms is based on discretizations of continuous-time stochastic dynamics \cite{grenander:miller:1994,grenander:1983,rossky1978brownian}. Famous examples of such MCMC methods are the Unadjusted Langevin Algorithm (ULA)~\cite{roberts:tweedie:1996} and Stochastic Gradient Langevin Dynamics~\cite{welling:the:2011}, which are based on the overdamped Langevin diffusion.

Here we consider numerical schemes based on Hamiltonian-type dynamics, \ie, ideal Hamiltonian
Monte Carlo and kinetic (or underdamped) Langevin diffusion. Although
these two dynamics are different, they have many features in common.
For example, both define an extended process on $\rset^{2d}$ that has
the product of $\pi$ and the standard Gaussian distribution as
invariant measure, denoted hereafter by $\mu$. Also, the infinitesimal
generators of these extended processes differ only in their symmetric
parts, while their antisymmetric parts are the same and correspond to
the Hamiltonian dynamics associated with the potential~$U$ associated
with~$\pi$, \ie, $\pi \propto \exp(-U)$. From this observation it
follows that common discretization strategies have been employed for
both dynamics. Among these methods, those based on splitting
techniques \cite{bussi2007accurate,Leimkuhler} are particularly attractive since they come with valuable
properties and important convergence guarantees. 

In this paper, we are
particularly interested in the family of splitting methods known as
generalized Hamiltonian Monte Carlo  (gHMC). This method has been
shown to yield weak second order errors. More precisely, denoting by
$\rmP_{\bfomega}$ the Markov kernel associated with gHMC, where~$\bfomega$ are the hyperparameters of the algorithm including the step size~$\delta  > 0$, the refreshment rate~$\eta$ and the integration time~$\tint>0$, 
the following convergence bound holds under suitable conditions (see e.g. \cite{abdulle2015long} for the Langevin case~$\tint=\delta$): for any sufficiently regular function $f : \rset^{2d} \to \rset$ and  initial distribution $\nu_0$, there exist $C_f \geq 0$ and $c_1 > 0$ depending on the parameters~$\bfomega$ only through the friction~$\gamma=(1-\eta)/\tint$ ($c_1$ being independent of~$f,\nu_0$) such that for any  number of iterations $k \in \nset$, $\vert \nu_0 \rmP_{\bfomega}^k f - \pi(f)\vert \leq C_f(\rme^{-c_1 \tint k} +
\delta^2)$.
This result illustrates the advantages of using well-chosen
splitting strategies compared to traditional Euler schemes, for which
similar conclusions can be drawn but with a larger second term on the right hand side of the previous inequality, typically of order~$\delta$ instead of~$\delta^2$. 
While weak error bounds already
provide significant convergence guarantees, another line of research
is concerned with establishing quantitative bounds for MCMC
algorithms, paying particular attention to the dimension dependence \cite{dalalyan2017theoretical,durmus2017nonasymptotic}.

Regarding the kinetic Langevin algorithm, existing works
\cite{Chatterji,cheng2018underdamped,zhang2023improved,Dalalyan2018OnSF}
analyze for most of them a modification of the Euler scheme. For this
particular algorithm, \cite{Chatterji} shows that, when $\pi$ satisfies
a log-Sobolev inequality and under additional conditions on $U$,
and denoting by $\nu_k$ the distribution of the $k$-th iterate of the
algorithm starting from $\nu_0$ with step size $\delta > 0$, 
there are constants $C_{\nu_0} \geq 0$ and~$c_1 > 0$, independent of $\delta$, such that for any $k \in\nset$, $\mch(\nu_k | \pi) \leq C_{\nu_0} (\rme^{-c_1 \delta k } + \delta^2)$.
where $\mch$ denotes the relative entropy or Kullback--Leibler divergence. Note that the Pinsker inequality then implies the same
type of bounds for the total variation distance, but with a second
term of order $\delta$. 

Quantitative bounds for splitting schemes of the kinetic Langevin diffusion are scarcer   up to our knowledge. Higher order quantitative bounds are established in \cite{MonmarcheSplitting}  in this context for the case when the potential $U$ is convex. On the other hand, regarding unadjusted HMC, combining results from \cite{BouRabeeEberleZimmer} and \cite{durmus2021asymptotic} implies that, for a fixed integration time $\tint>0$, there exists~$c_1 > 0$ (depending on the parameters of the dynamics only through~$\tint$) such that, for any initial distribution~$\nu_0$, 
 $ k \in\nset$, $\mathbf{W}_1(\tilde{\nu}_k, \pi) \leq C_{\nu_0}(\rme^{-c_1   k } +
\delta^2)$, 
where~$C_{\nu_0} \geq 0$ is a constant, $\tilde {\nu}_k$ is the distribution of the $k$-th iterates of unadjusted HMC starting from $\nu_0$, and $\mathbf{W}_1$ denotes the Wasserstein distance of order~$1$. This result once again highlights the improved accuracy of the leapfrog integrator. Finally, the analysis of gHMC in the Wasserstein distance has been conducted in
\cite{monmarche2022hmc}, but only for the strongly convex scenario.

The main contribution of this paper is to extend and generalize to the non-convex scenario the results we just mentioned and to analyze gHMC. In particular, we show that gHMC achieves a higher order accuracy than traditional Euler schemes in relative entropy, under the condition that~$\pi$ satisfies a log-Sobolev inequality, and additional relatively mild assumptions on the potential energy function~$U$. Roughly speaking, we establish a bound of the form
$\mch(\nu_0 \rmP^k | \mu) \leq C_{\nu_0}\rme^{-c_1 k } +
C_{2}\delta^4$ for some explicit constants~$C_{\nu_0},C_2 \geq 0 $ and
$c_1 > 0$. Our approach is based on a perturbative argument of the modified entropy approach initiated in
\cite{Villani2009}, and more precisely to its recent discrete-time variation in \cite{MonmarcheIdealized} for idealized gHMC. From our main result, we derive bounds on the total number of iterations $k$ and a step size $\delta$ to achieve
$\mch(\nu_0 \rmP^k | \mu) \leq \vareps$ for some $\vareps > 0$,
where two cases can be distinguished: in the general case, we obtain a bound on the number of gradient computations of order $\bigO(d/\vareps^{1/4})$; while in a weakly interacting mean-field regime, an improved bound of order $\bigO((d/\vareps)^{1/4})$ can be achieved.

The paper is organized as follows. The family of (unadjusted) generalized Hamiltonian Monte Carlo samplers is introduced in Section~\ref{subsec:definitions}. Our  main assumptions and results are stated in Section~\ref{subsec:MainResult}. They are discussed and compared to previous works in  Section~\ref{sec:discussion}. Section~\ref{sec:proof-theor-refthm:m}  is devoted to the proof of  Theorem~\ref{thm:main_KL}. Other more technical proofs are postponed to Appendices.  

\paragraph{Notation.}
We denote by $|x|$ the Euclidean norm of $x \in\rset^d$.
We denote by $\rmC^{k}(\R^d)$ the set of functions from $\R^d$ to $\R$ with continuous derivatives up to order $k$. For $f\in \rmC^k(\R^d)$, $\na f$ stands for the gradient of $f$ and $\rmD^k f$ for the $k$-th derivative of $f$. In addition, $|\rmD^k f(x)|$ is the multilinear operator norm of $\rmD^k f(x)$ with respect to the Euclidean norm on $\R^d$ and $\|\rmD^k f\|_\infty = \sup_{x\in\R^d}|\rmD^k f(x)| $. For a $d\times d$ matrix $\bfA$, $|\bfA|$ stands for the operator norm on $\R^d$ with respect to the Euclidean norm. For a differentiable map $\Phi:\R^d \rightarrow \R^p$ and $x\in\R^d$ we write $\na \Phi(x) = (\partial_{x_i} \Phi_j(x))_{i\in\cco 1,d\ccf,j\in\cco 1,p\ccf}$ (where $i$ stands for the line and $j$ for the column). This notation is such that $\na\po \Phi \circ \Psi\pf  = \na \Psi \na \Phi\circ \Psi$ and is consistent in the case $p=1$ (since $\na\Phi$ is in that case the gradient of $\Phi$). Notice that $\na \Phi$ is the \emph{transpose} of what is most commonly named the Jacobian matrix of $\Phi$.

 $\Leb$  stands for the Lebesgue measure, $\mathrm{N}(0,\IdM)$ for the $d$-dimensional standard Gaussian distribution and
$\mathcal{P}(\rset^d)$ for the set of distribution of $\rset^d$ endowed with its Borel $\sigma$-field denoted by $\mathcal{B}(\rset^d)$. We write $\nu\ll\mu$ when $\nu$ is absolutely continuous with respect to $\mu$.  We define the relative entropy (or Kullback-Leibler [KL-] divergence) and the Fisher information of $\nu_1 \in \Pens(\rsetd)$ with respect to  $\nu_2 \in \Pens(\rsetd)$ respectively as
 \begin{equation}
 \mathcal  H(\nu_1|\nu_2)  = \begin{cases}
\int \ln \po \frac{\rmd \nu_1}{\rmd \nu_2}\pf  \rmd \nu_2  & \text{if } \nu_1 \ll \nu_2\\
+ \infty & \text{otherwise } 
\end{cases}
\eqsp, \quad 
\mathcal  I(\nu_1|\nu_2)   =
\begin{cases}
\int \left| \nabla\ln \po \frac{\rmd \nu_1}{\rmd \nu_2}\pf\right|^2 \rmd \nu_1  & \text{if } \nu_1 \ll \nu_2\\
+ \infty & \text{otherwise }
\end{cases}
\eqsp,
 \end{equation}
where for a measurable function $f: \rsetd \to \rset$, $|\na f|$ is defined  for any $z \in\rset^d$ as
\[|\na f|(z) \ = \ \lim_{r\downarrow 0}\ \sup\left\{ \frac{|f(z)-f(y)|}{|z-y|} \, :\,\ y\in\rsetd,\ 0<|y-z|\leqslant r\right\}\,.\]
If $f$ is continuously differentiable, note that $\abs{\nabla f}(z)$ is simply the norm of $\nabla f(z)$ and therefore our notation is consistent.

We write $\cco i,j\ccf$ the integer interval $\{i,\dots,j\}$ for $i,j\in\mathbb Z$.


\section{Non-asymptotic bounds for splitting schemes for Hamiltonian type dynamics}\label{sec:main}

\subsection{Splitting schemes for Hamiltonian type dynamics -- gHMC}\label{subsec:definitions}

Recall that we assume that the target distribution $\pi$ admits a positive density with respect to the Lebesgue measure $\pi \propto \exp(-U)$. In addition, we suppose the following condition on $U$.

\begin{assumption}\label{hyp:target}
  The potential $U \in \rmC^{4}(\rset^d)$ is such that~$\rme^{-U} \in \Lone(\rset^d)$, and there exists~$L>0$ such that~$|\na^2 U(x)| \leqslant L$ for any $x\in\R^d$.
\end{assumption}

We now introduce more formally the two continuous dynamics that are considered in this work: Hamiltonian dynamics and kinetic (also referred to as underdamped) Langevin dynamics.
As previously mentioned, these two dynamics leave the extended target distribution $\mu = \pi \otimes \gauss(0,\IdM)$ invariant.

\paragraph{Ideal Hamiltonian Monte Carlo.}
The Hamiltonian dynamics associated with the potential~$U$ defines the differential flow $(\psi_t)_{t \geq 0}$ for any $t \in \rset_+$ and $x,v \in\rset^d$ as $\psi_t(x,v)=(x_t,v_t)$, where $(x_t,v_t)_{t \geq 0}$ is the solution of the Hamiltonian differential equation:
\begin{equation}
  \label{eq:def_hamil_ode}
  \partial_t x_t  = v_t \eqsp,\quad \partial_t v_t = -\nabla U(x_t) \eqsp, \quad \text{ with } (x_0,v_0)=(x,v) \eqsp.
\end{equation}
As discussed in~\cite{neal:2010} for instance, the flow~$(\psi_t)_{t \geq 0}$ preserves $\mu$, \ie, if $(X_0,V_0)$ has distribution $\mu$ then so does $\psi_t(X_0,V_0)$ for all $t \geq 0$. However, the trajectory $(\psi_t(X_0,V_0))_{t \geq 0}$ is of course not ergodic for~$\mu$ as, starting from any fixed initial conditions $x_0,v_0\in\rsetd$, it remains in the corresponding level set  $\{(x,v) \in \rsetdd \, : \, H(x,v) = H(x_0,v_0)\}$ of the Hamiltonian function~$H$  defined by
\begin{equation}
  \label{eq:def_hamil}
  H(x,v) = U(x) + \frac12 \abs{v}^2 \eqsp.
\end{equation}
To address this issue, one can add a velocity randomization (or refreshment) at deterministic times multiple of~$\tint  >0$. To formalize this procedure, consider the Markov operator $\rmD_{\eta}$ given, for $\eta \in \coint{0,1}$ and any measurable and bounded function $f :\rsetdd \to \rset$, by 
\begin{equation}
  \label{eq:2}
  \rmD_{\eta}f(x,v)  = \int_{\R^d} f\parenthese{x, \eta v + \sqrt{1-\eta^2}  g} \varphibf(g) \rmd g \eqsp, 
\end{equation}
where $\varphibf$ stands for the density with respect to $\Leb$ of $\mathrm{N}(0,\IdM)$. Then, under mild assumptions, the resulting kernel
\begin{equation}
  \label{eq:def_ham_ideal}
  \rmK_{\eta,\tint}f(x,v) = \rmD_\eta[f\circ \psi_{\tint}](x,v) \eqsp,
\end{equation}
is ergodic with respect to $\mu$. Recently, \cite{MonmarcheIdealized} used  hypocoercivity techniques from \cite{Villani2009} to show exponential convergence of $  \rmK_{\eta,\tint}$ in a modified entropy with respect to $\mu$. 

\paragraph{Underdamped Langevin dynamics.}
The second process which shares some important features with Hamiltonian dynamics is the underdamped Langevin diffusion:
\begin{equation}\label{eq:EDScontinuousLangevin2}
\dd \bfX_s  =  \bfV_s \, \dd s \eqsp, \qquad 
\dd \bfV_s  =  - \na U(\bfX_s) \, \dd s - \gamma \bfV_s \, \dd s + \sqrt{2\gamma} \, \dd B_s\,,
\end{equation}
where $(B_s)_{s \geq 0}$ is a $d$-dimensional Brownian motion and $\gamma>0$ is a damping parameter.
It can  easily be  shown that under mild assumptions on $U$, the Hamiltonian dynamics with a suitably scaled refreshment  weakly converges to  \eqref{eq:EDScontinuousLangevin2}: for any $s \geq 0 $, $x,v\in\rset^d$ and smooth and bounded function $f : \rsetdd \to \rset$,  it holds $\displaystyle \lim_{\delta \to 0} \rmK_{\eta_{\delta},\delta}^{\ceil{s/\delta}}f(x,v)  = \bfP_s f(x,v)$ where
\begin{equation}
  \label{eq:def_eta_delta}
  \eta_{\delta} = \rme^{-\delta \gamma} \eqsp,
\end{equation}
 and $(\bfP_s)_{s\geq 0}$ is the Markov semigroup associated with \eqref{eq:EDScontinuousLangevin2}. 
 
 \paragraph{Discretizations of underdamped Langevin and ideal HMC.}
 We now present a family of splitting schemes which encompasses
 discretizations for both Hamiltonian and underdamped
 Langevin dynamics. This family of algorithms will be referred to as generalized Hamiltonian Monte Carlo (gHMC). It is based on the Verlet discretization of \eqref{eq:def_hamil_ode}
 defined at time $k \delta$, for $k \in \nset$ and a given stepsize~$\delta >0$, by $(x_{k},v_{k}) = \Phi^k_{\delta}(x_0,v_0)$ where $\Phi^k_{\delta} = \Phi^{k-1}_\delta \circ \Phi_{\delta}$, with~$\Phi^0_{\delta}$ is the identity function and 
 \begin{equation}
   \label{eq:verlet_decompo}
      \Phi_{\delta} = \Phi^{(v)}_{\delta/2} \circ \Phi^{(x)}_{\delta} \circ \Phi^{(v)}_{\delta/2}\eqsp, \quad \Phi^{(v)}_{\delta/2}(x,v) = \po x,v -(\delta/2) \nabla U(x)\pf  \eqsp,  \quad \Phi^{(x)}_{\delta}(x,v) = \po x + \delta v ,v\pf  \eqsp.
 \end{equation}
 It is well-known that the integrator $\Phi_{\delta}$ is symplectic and reversible (see~\cite{HairerLubichWanner06}). Its inverse reads
  \begin{equation}
   \label{eq:verlet_inverse_and_refl}
   \Phi_{\delta}^{-1} = \refl \circ \Phi_{\delta} \circ \refl \eqsp, \qquad \refl(x,v) = (x,-v) \eqsp.
 \end{equation}

The gHMC algorithm then consists in the composition of a (possibly partial) velocity refreshment with a $K$-step Verlet scheme. Setting $\bfomega= (K,\delta,\eta)$, this corresponds to the inexact Markov chain Monte Carlo (MCMC) method with Markov kernel specified  by
\begin{equation}
  \label{eq:def_SHD}
  \rmPo =  \rmD_{\eta}\rmV_{\delta}^K  \eqsp, 
\end{equation}
where $\rmV_{\delta}$ corresponds to the deterministic kernel
\begin{equation}
  \label{eq:def_Verlet_kernel}
  \rmV_{\delta}((x,v),\cdot) = \updelta_{\Phi_{\delta}(x,v)}(\cdot) \eqsp, \quad \text{ for $x,v \in \rset^d$} \eqsp. 
\end{equation}

For $\eta=0$ and $K \delta = \tint$ for an integration time $\tint >0$, the kernel~$\rmP_{\bfomega}$ corresponds to the usual unadjusted Hamiltonian Monte Carlo algorithm, which is a discretized version of the ideal Hamiltonian dynamics $\rmK_{\tint ,\eta}$ defined by \eqref{eq:def_ham_ideal}. On the other hand, with~$K=1$ and the damping parameter~$\eta_{\delta}$ defined in~\eqref{eq:def_eta_delta}, the kernel~$\rmP_{\bfomega}$ corresponds to a splitting scheme whose invariant probability measure is correct at second order in~$\delta$ (this corresponds in fact to the so-called GLA scheme~\cite{BO10}; see also~\cite{Leimkuhler,lelievre2016partial} for other splitting schemes). Alternatively, taking $\eta=0$ and $K=1$ leads to the Euler--Maruyama scheme 
with step size $\delta^2/2$ for the overdamped Langevin process.

\paragraph{Normalization and rescaling.}
When \Cref{hyp:target} is satisfied, we introduce
\begin{equation}
  \label{eq:def_tint_gammaint}
  \tint = \delta K\sqrt{L}\eqsp,\qquad \gammaint = (1-\eta)/\tint\,,
\end{equation}
which are respectively, up to a suitable rescaling, the physical time of integration of the Hamiltonian dynamics by the Verlet integrator, and the strength of the damping. As discussed in \cite{monmarche2022hmc}, upon rescaling the process so that $L=1$ (see the discussion in Section~\ref{subsubsec:scaling}), the time needed for the process to forget its initial velocity is of order~$1/\gammaint$, so that this quantity is also the typical distance covered by one of the coordinates of the process in a single ballistic run in flat parts of the space.

\subsection{Main result}\label{subsec:MainResult}

We present in this section our main result, under additional assumptions regarding~$\pi$ and~$U$ (see Section~\ref{sec:actual_main_result}). We next provide in Section~\ref{sec:sufficient_conditions_for_main_result} sufficient conditions under which these additional assumptions are satisfied.

\subsubsection{Uniform convergence in time}
\label{sec:actual_main_result}
\paragraph{Additional assumptions.}
Our first assumption is that~$\pi$ satisfies a log-Sobolev inequality with constant $\CLS \geq 0$,  \ie 
\begin{equation}\label{eq:log-Sob}
\forall \nu \in \Pens(\rset^d)\,,\qquad \mathcal H(\nu|\pi)  \leqslant \CLS \, \mathcal I(\nu|\pi)\eqsp.
\end{equation}
Various sufficient conditions for this inequality to hold have been derived, see~\cite{BakryGentilLedoux}. It holds for instance when~$U$ is uniformly convex outside some compact set. Various estimates are available for $\CLS$, e.g. in mean-field, convex or low-temperature cases, which can capture more specific information on the target~$\pi$ than uniform bounds on the curvature $\ps{x-y}{ \na U(x)-\na U(y)} /|x-y|^2$ that are used in direct coupling methods. Notably, in non-convex cases, for a fixed dimension $d$, up to polynomial terms in the other parameters, the log-Sobolev constant is of order~$\rme^{c_*}$ where~$c_*$ is the so-called critical height of~$U$ (see~\cite[Section 5]{MonmarcheIdealized} and references within).


Part of our analysis consists in controlling some numerical errors of the leapfrog integrator, \ie, for example differences of the Hamiltonian function evaluated at the dynamics at times~$0$ and $s >0$ starting from the $k$-th iterates of the gHMC chain. To this end, we need additional regularity conditions on~$U$. More specifically, we require uniform bounds on the third and fourth derivatives of $U$, as made precise in the following condition.
 

\begin{assumption}\label{assu:derivativesUbounded}
There exist two norms $\rmN_3$ and $\rmN_4$ on $\R^d$ such that, for all $x,y,z\in\R^d$,
\begin{align}
\abs{ \po  \na^2 U(x+y) - \na^2 U(x)\pf z} & \leqslant      \rmN_3(y)\rmN_3(z) \,, \label{eqdef:assuU3bounded}\\
\abs{\na U(x+y) - \na U(x) - \frac12\po \na^2 U(x)+\na^2 U(x+y)\pf y}  &\leqslant    \rmN_4^3(y)
\,. \label{eqdef:assuU4bounded}
\end{align}
\end{assumption}

Using a third order Taylor expansion of $\na U$, it can be shown that~\eqref{eqdef:assuU4bounded} follows from a uniform bound on the third and fourth derivatives of~$U$ (see Proposition~\ref{prop:regularity} below).
\Cref{{assu:derivativesUbounded}} may be relaxed, i.e, if one assumes only bounded third derivatives for~$U$, our main result still holds but with worse dependencies with respect to the step size, see Remark~\ref{rem:order1} below.

Bounding the left-hand sides in~\eqref{eqdef:assuU3bounded} and~\eqref{eqdef:assuU4bounded} with respect to the Euclidean distance on~$\rset^d$ may lead to sub-optimal dependencies with respect to the dimension~$d$ in~\Cref{thm:main_KL} below (for instance when~$U$ is separable, see Theorem~\ref{Thm:drift} below), which is why we allow in \Cref{assu:derivativesUbounded} some flexibility in the choice of the norms. A specific choice of norms is given in \Cref{prop:regularity}, and the relevance of this choice for the dimensionality dependence is discussed in  \Cref{rem:scaling_dimension}.

As made precise in \Cref{sec:numerique}, the numerical errors that need to be controlled at a point~$(x,v)\in\R^{2d}$ can be bounded by the function 
\begin{equation}\label{eq:defVmoments}
   \mathbf{M} (x,v) = L^2 |v|^2 + L^2 |\na U(x)|^2 + \rmN_3^4(v)  + \rmN_3^4(\na U(x) ) + \rmN_4^6(v)  + \rmN_4^6(\na U(x) )\,.
\end{equation}
In order to control the expectation of these errors, one needs uniform moment bounds for the  Markov kernel~$\rmP_{\bfomega}$. These estimates typically follow from a Lyapunov condition, but for clarity we postpone this analysis to Theorem~\ref{Thm:drift} below and, for the time being, state it as the next assumption.


\begin{assumption}\label{assu:moments}
There exist $\rho,C_1,C_2> 0$ and $\Wlyap : \rsetdd \to \rset_+$ such that, for all~$(x,v) \in \rset^{2d}$, $k\in \{0,\ldots,K\}$ and $n\in\N$,
\begin{equation}
  \label{eq:1_bis}
  \rmP_{\bfomega}^n \rmD_{\eta}\rmV_{\delta}^k \bfM (x,v) \leq C_1 \rme^{-\rho n \tint} \Wlyap(x,v) + C_2 \,,
\end{equation}
where $\tint$ is defined in \eqref{eq:def_tint_gammaint}.
\end{assumption}

In practice, this can be established by designing a suitable Lyapunov function $\Wlyap$ such that $\mathbf{M} \leqslant C_1 \Wlyap$ and
\begin{equation}
  \label{eq:drift_v0}
  \rmP_{\bfomega} \Wlyap \leqslant \rme^{-\rho \tint} \Wlyap + C_2' \tint \eqsp,
\end{equation}
for some constants~$C_1,C_2',\rho>0$. As we will see in Theorem~\ref{Thm:drift}, the constants~$C_1,C_2,\rho$ typically only depend on the parameters~$\bfomega$ through~$\gammaint$ defined in~\eqref{eq:def_tint_gammaint}. In particular, they are uniform over all sufficiently small step-sizes~$\delta$. For the continuous-time kinetic Langevin diffusion and its discretizations, such Lyapunov functions have been designed in  \cite{MATTINGLY2002185,Talay,DEMS} under coercitivity conditions on~$U$. We show in Appendix~\ref{sec:Lyapunov} that appropriate modifications of these Lyapunov functions can be used to establish a drift condition~\eqref{eq:drift_v0} for $\rmP_{\bfomega}$.

\paragraph{Main convergence result.}
We are now in position to state our main result. For the sake of clarity, we only consider the case where $L=1$ in \Cref{hyp:target}, the general case being obtained by rescaling, as discussed at the end of this section. 

\begin{theorem}\label{thm:main_KL}
Assume that \Cref{hyp:target}-\Cref{assu:derivativesUbounded} hold with $L=1$, and consider~$\bfomega =(K,\delta,\eta)$ with~$K \in\nsets$, $\delta >0$ and $\eta \in \coint{0,1}$ such that \Cref{assu:moments} holds and~$\tint=K\delta \leqslant 1/10$.   Furthermore,  assume that $\pi$ satisfies the log-Sobolev inequality \eqref{eq:log-Sob}. Consider an initial condition $\nu_0$ such that $\mathcal I( \nu_0|\mu) < +\infty$ and $\int_{\R^{2d}} \abs{z}^p \rmd \nu_0(z) <\plusinfty$ for any $p >0$.
Then, for any $n\in\N$, 
\begin{equation}
\label{eq:bound_theo}
\begin{aligned}
  \mathcal H \po \nu_0 \rmPo^n | \mu\pf & \leqslant (1+ \kappa  \tint)^{-n}  \left[ \mathcal H(\nu_0|\mu) + 2a \, \mathcal I(\nu_0|\mu)\right]  \\
 & \qquad + \delta^4 M   \left[ n\tint  \theta^{n\tint} C_1\nu_0( \Wlyap )+  C_2 ( \kappa^{-1} +\tint) \right], 
\end{aligned}
\end{equation}
where $\theta= \max(\rme^{-\rho},(1+\kappa \tint)^{-1/\tint})$ and 
\begin{equation}\label{eq:a_M_simpler}
      a = \frac{\gammaint}{7+3(\gammaint+3)^2}\,,\qquad  \kappa   = \frac{ a}{3\max( \CLS,1)+6 a}\,,\qquad M = \newOld{9} + \frac{1}{a}\,.  
    \end{equation}
\end{theorem}

The first term on the right hand side of~\eqref{eq:bound_theo} encodes the long-time convergence of the idealized HMC chain~\cite{MonmarcheIdealized} (and goes to $0$ as $n\rightarrow \infty$). The Fisher term~$\mathcal I(\nu_0|\mu)$ appearing there is due to the fact Theorem~\ref{thm:main_KL} is proven using  a modified entropy (defined in \eqref{eq:modifiedEntropy}) instead of $\mathcal H$. The second line of~\eqref{eq:bound_theo}, of order $\delta^4$, corresponds to the numerical error due to the Verlet integration, which is decomposed into two terms. The first one comes from the fact that the initial condition is not invariant under~$\rmP_{\bfomega}$ and highlights the exponential forgetting of this initial condition since $\theta<1$. The second term involving~$C_2$ accounts for the numerical error at stationarity, \ie~the bias between the invariant probability measure of the discrete process and the reference probability measure~$\mu$.

Since, in the case $L=1$, the physical integration time of the Hamiltonian dynamics is~$\tint$, the total physical time of the simulation after $n$ iterations is~$n\tint$. The constant $\kappa$ is therefore the convergence rate per unit physical time. In view of~\eqref{eq:a_M_simpler}, $\kappa$ depends on the parameters~$\bfomega$ only through the friction~$\gammaint$, with a dependency of order~$\min(\gammaint,\gammaint^{-1})$ similarly to what is obtained for the continuous-time Langevin diffusion (see for instance~\cite{DKMS13,GS16,IOS19}). From this observation, it follows that  the convergence per gradient computation is $\kappa \tint/K = \kappa \delta$.

At the expense of more involved expressions, we provide an improved version of \Cref{thm:main_KL} in \Cref{Prop:dissipation} below. In particular, we provide a slightly sharper value of $\kappa $, which is optimal in the overdamped case since, in the continuous time limit~$\delta \to 0$, one obtains (see Remark~\ref{rmk:convergence_implies_LSI})
\begin{equation}
\label{eq:convergence_implies_LSI}    
\mathcal H(\hat \nu_0 \mathrm Q_s|\pi) \leqslant \rme^{-s/\CLS} \mathcal H(\hat \nu_0|\pi) \,,
\end{equation}
where $(\mathrm{Q}_t)_{t\geqslant 0}$ is the semi-group associated to the overdamped Langevin diffusion.
This shows that our result is sharp since  it \emph{implies} that $\pi$ satistifes a log-Sobolev inequality with constant $\CLS$ (see e.g., \cite[Theorem 5.2.1]{BakryGentilLedoux}).


\paragraph{Rescaling.}\label{subsubsec:scaling} 
As mentioned above, following \cite[Section 1.3]{MonmarcheIdealized}, it is not restrictive to assume that~$L=1$ in \Cref{hyp:target}. Indeed, if $Z_n=(X_n,V_n)$ is an unadjusted gHMC chain with potential energy function~$U$ and parameters~$K,\delta,\eta$, then $(\sqrt{L} X_n,V_n)$ is an unadjusted gHMC chain associated with the potential energy function~$\tilde U(x)=U(x/\sqrt{L})$ (so that $\na \tilde U$ is $1$-Lipschitz under \Cref{hyp:target}) and parameters $K,\delta \sqrt{L},\eta$. Moreover, if $\pi$ satisfies a log-Sobolev inequality with constant $\CLS$, then the change of variable $x \mapsto x/\sqrt{L}$ implies that the rescaled target measure $\tilde \pi \propto \rme^{-\tilde U}$ satisfies a log-Sobolev inequality with constant $L\CLS$. Similarly, using the same change of variable, the relative entropy is invariant by scaling, namely, if $(X,V)$ has distribution $\nu \in \mathcal{P}(\rsetdd)$ and $\tilde \nu$ stands for the law of $(\sqrt{L}X,V)$, then $\mathcal H\po \nu | \mu \pf = \mathcal H\po \tilde \nu | \tilde \mu \pf $,
where $\tilde \mu = \tilde \pi \otimes \mathrm N(0,\IdM)$; while the Fisher information satisfies
\[\mathcal I\po \tilde \nu| \tilde \mu \pf = \int_{\R^{2d}} \frac{L^{-1} |\na_x h|^2 + |\na_v h|^2}{h} \dd \mu\,, \text{ where $h = \dd \nu/\dd \mu$} \eqsp.\]

%
%

\subsubsection{Sufficient conditions for \Cref{assu:derivativesUbounded} and \Cref{assu:moments}}\label{sec:sufficient_conditions_for_main_result} 

We now provide practical conditions on~$U$ to establish \Cref{assu:derivativesUbounded} and \Cref{assu:moments}. We assume that~$d=\nq q$ for some $\nq ,q\geqslant 1$ and, decomposing $x=(x_1,\dots,x_{\nq}) \in \R^d = (\R^q)^{\nq}$, that~$U$ is of the following form:
\begin{equation}\label{eq:Umean_field}
U(x) = \sum_{i=1}^{\nq} \Uq (x_i) + \frac{\epsilon}{\nq } \sum_{i,j=1}^{\nq } \Wq(x_i-x_j) \eqsp,
\end{equation}
where $\Uq$ and $\Wq$ are potentials defined on $\R^q$,  and $\epsilon>0$. Such mean field potentials are commonly encountered in various applications (see e.g., \cite{ChaintronDiez}). Furthermore, in line with the approach taken in \cite{BouRabeeSchuh}, we consider this formulation as a representative example to demonstrate the dimension dependence in weakly correlated cases. Let us however emphasize that~\eqref{eq:Umean_field} is not a restrictive setting since it always holds with the choice $\nq =1$, $\Uq=U$, and $\Wq=0$. However, as discussed below, this mean-field formulation allows us to precisely determine the dependence of $C_1$ and $C_2$ in \Cref{assu:moments} on the dimension $d$, in particular by distinguishing the roles of $\nq$ and $q$. For conditions under which $\pi$ satisfies a log-Sobolev inequality with a constant $\CLS$ independent of $\nq$ for $U$ given by \eqref{eq:Umean_field} (provided $\epsilon$ is sufficiently small), see \cite[Theorem 8]{GuillinWuZhang}.

First, as proved in Appendix~\ref{sec:regularity}, we can readily obtain the following.

\begin{proposition}\label{prop:regularity}
Assume that  $d=\nq  q$ for some $\nq ,q\geqslant 1$ and that $U$ is of the form \eqref{eq:Umean_field} where $U^{(q)},W^{(q)}\in\mathrm{C}^4(\R^q)$ have uniformly bounded third and fourth derivatives. Then  \Cref{assu:derivativesUbounded}  holds with 
\[
N_3(x) = L_3 \left( \sum_{i=1}^{\nq} |x_i|^4 \right)^{1/4} \, ,
\qquad 
N_4(x) = L_4 \left( \sum_{i=1}^{\nq} |x_i|^6 \right)^{1/6} \, ,
\]
where  
  \begin{align}
     L_3^4  &= (1+\epsilon) \left\|\rmD^3 U^{(q)}\right\|_\infty ^2 + 16 \epsilon(1+\epsilon) \left\|\rmD^3 W^{(q)}\right\|_\infty^2 \eqsp, \\
     L_4^6 &= \po \frac{1+\epsilon}{144} \left\| \rmD^4 U^{(q)} \right\|_\infty^2 + \frac{4\epsilon(1 +\epsilon) }{9 } \left\|\rmD^4 W^{(q)} \right\|_\infty^2 \pf q \,.
  \end{align}
  As a consequence, assuming furthermore that $\na U^{(q)}(0) = 0 = \na W^{(q)}(0)$, the function~$\mathbf{M} (x,v)$ defined in~\eqref{eq:defVmoments} is bounded as 
  \[
  \mathbf{M} (x,v) \leqslant   \sum_{i=1}^{\nq}  \sum_{\ell\in\{2,4,6\}} L_\ell^\ell ( r_\ell |x_i|^\ell + |v_i|^\ell) \, ,
  \]
  with $r_2 = 1$, $r_\ell = 2^{\ell-1}\|\na^2 U^{(q)}\|_\infty^{\ell}  + 2^{3\ell}\epsilon^{\ell}  \|\na^2 W^{(q)}\|_\infty^\ell $ for $\ell\in\{4,6\}$ and $L_2 = L$. 
\end{proposition}

To establish moment bounds, we consider the following assumption.

\begin{assumption}\label{ass:for_drift} The dimension is $d=\nq  q$ for some $\nq ,q\geqslant 1$ and $U$ is of the form~\eqref{eq:Umean_field} where $U^{(q)},W^{(q)}\in\mathrm{C}^2(\R^q)$ have uniformly bounded second derivatives and are such that $\nabla U^{(q)}(0) = \nabla W^{(q)}(0) = 0$. In addition,  there exist $\mtt >0$ and $\Mtt \geq 0$ such that 
  for any $x_1 \in \rset^q$,
  \begin{equation}
    \label{eq:ass_drift}
 \ps{x_1 }{ \nabla U^{(q)}(x_1)}  \geq \mtt \abs{x_1}^{2} -\Mtt \eqsp.
  \end{equation}
\end{assumption}

Note that this assumption is more general than assuming that~$U^{(q)}$ is strongly convex outside a compact set. For instance, a potential such as $\abs{x_1}^2 + C \cos \abs{x_1}$ satisfies~\eqref{eq:ass_drift} for any~$C \in \mathbb{R}$, but is not strongly convex when~$|C| \geq 2$.

\begin{theorem}\label{Thm:drift}
Suppose that \Cref{ass:for_drift} is satisfied, and define 
\begin{equation}\label{eq:Wlyap_simple}
\Wlyap(x,v) =\sum_{\ell\in\{2,4,6\}} \sum_{i=1}^{\nq} \po  |x_i|^\ell + |v_i|^\ell\pf \,.
\end{equation}
For any $\gammaint_0 >0$, there exist~$\overline{\tint},\overline{\delta},\overline{C}_1,\overline{C}_2,\overline{\epsilon},\rho >0$, which depend on~$\|\na^2 U^{(q)}\|_\infty$, $\|\na^2 W^{(q)}\|_\infty$, $\mtt$, $\Mtt$, $\gammaint_0$ (but not on~$q$,$\nq$), such that the following holds: for any~$\epsilon \in (0,\overline{\epsilon}]$ and~$\bfomega =(K,\delta,\eta)$ with~$\delta \in (0,\overline{\delta}]$, $\eta \in [0,1)$ such that~$(1-\eta)/(K\delta) = \gammaint_0$ and~$\delta K \leq \overline{\tint}$, it holds, for any~$(x,v) \in \rset^{2d}$, $k\in \{0,\ldots,K\}$ and $n\in\N$,
\begin{equation}
  \label{eq:1_bis}
  \rmP_{\bfomega}^n \rmD_{\eta}\rmV_{\delta}^k \Wlyap (x,v) \leq \overline{C}_1 \rme^{-\rho n K\delta } \Wlyap(x,v) + \overline{C}_2 \nq q^3 \,.
\end{equation}
\end{theorem}

Explicit expressions for $\overline{C}_1,\overline{C}_2,\overline{\epsilon},\rho$ are given in Appendix~\ref{sec:Lyapunov}, where a more detailed result, Theorem~\ref{theo:drift}, is stated and proven. As explained in Appendix~\ref{sec:Lyapunov}, Theorem~\ref{Thm:drift} is a straightforward corollary of Theorem~\ref{theo:drift}.

\begin{remark}\label{rem:scaling_dimension}
In order to determine more precisely the dimension dependence of~$C_1,C_2$ in~\Cref{assu:moments}, obtained for~$\Wlyap$ given by \eqref{eq:Wlyap_simple} by combining Proposition~\ref{prop:regularity} with Theorem~\ref{Thm:drift}, we treat~$\|\rmD^k U^{(q)}\|_\infty$ and $\|\rmD^k W^{(q)}\|_\infty$ for $k\in\{2,3,4\}$ as constants independent from~$q$. This motivates defining 
$R = q^{-1} \max \{L_\ell^\ell \max(r_\ell,1) : \ell \in \{2,4,6\} \} $,
which can be considered to be uniform in~$\nq$ and~$q$ in this context, so that the expressions  $C_1 = R \overline{C}_1 q$, $C_2 = R \overline{C}_2 \nq q^4$
capture the dimension dependence of~$C_1,C_2$ in terms of~$q,\nq$.
\end{remark}

\subsection{Discussion and related works}\label{sec:discussion}

\paragraph{Complexity bounds.}
Regarding the dependence in the dimension of the constants appearing  in \Cref{assu:moments}, we distinguish two cases in view of Remark~\ref{rem:scaling_dimension}. In the first one, which we call the weakly interacting case, we consider that $C_2$ and $\nu_0(\Wlyap )$ in \eqref{eq:bound_theo} are $\mathcal O(d)$. This corresponds in Remark~\ref{rem:scaling_dimension} to the case where $q$ is fixed, so that $d$ is proportional to $d_0$. In the second case, which we call the general case, we consider that $C_2$ and $\nu_0(\Wlyap )$  are~$\mathcal O(d^4)$.  This corresponds in Remark~\ref{rem:scaling_dimension} to the case~$d_0=1$ and~$d=q$. In both settings, we consider that $L$, $C_1$ and $\gammaint$ are constants independent of the dimension, and that~$\mathcal H(\nu_0|\mu)$ and~$\mathcal I(\nu_0|\mu)$ are~$\mathcal O(d)$ (see e.g. \cite[Lemma~1]{VempalaWibisono}). Moreover, the convergence rate~$\kappa$ is of order~$1/\CLS$, with $\CLS$ large, so that~$(1+\kappa \tint)^{-n}$ is of order~$\rme^{-\alpha n K \delta/\CLS}$ for some constant~$\alpha > 0$. In this context, Theorem~\ref{thm:main_KL} gives the following complexity bounds for a given error tolerance~$\varepsilon>0$ (in relative entropy):
\begin{itemize}
    \item in the general case, upon choosing
$\delta = \mathcal O ( \varepsilon^{1/4} d^{-1})$ (so that the second term on the right hand side of~\eqref{eq:bound_theo} is of order~$\varepsilon$), we get that $\mathcal H(\nu_0 \rmPo^n|\mu) \leqslant \newOld{\varepsilon}$  after a number $n K = \mathcal O( \CLS d \vareps^{-1/4} \ln (d\vareps^{-1}) ) $  of computations of the gradient $\na U$.
\item in the weakly interacting case, a similar reasoning suggests choosing~$\delta = \mathcal O ( \varepsilon^{1/4}d^{-1/4})$, in which cas~$\mathcal H(\nu_0 \rmPo^n|\mu) \leqslant \varepsilon$ after a number~$n K = \mathcal O( \CLS (d/ \vareps)^{1/4} \ln  (d\vareps^{-1}) )$
of computations of the gradient $\na U$. 
\end{itemize}
Note that, in both cases, 
the Pinsker and Talagrand inequalities respectively imply that~$ \| \mu-\nu_0 \rmPo^n\|_{\mathrm{TV}}$ and~$\mathbf W_2\po \mu,\nu_0 \rmPo^n\pf$ are of order~$\mathcal O(\varepsilon^{1/2})$, where~$\| \cdot\|_{\mathrm{TV}}$ and~$\mathbf W_2$ respectively denote the total variation distance and Wasserstein distance of order~$2$. In view of this, the dependence of the number of gradient evaluation on~$\varepsilon^{1/2}$, namely~$nK = \mathcal O (\varepsilon^{-1/4})$ (omitting the  logarithmic term), highlights that gMHC is a second-order splitting scheme. This is consistent with existing weak error estimates; see  \cite{abdulle2015long} and references therein.

\paragraph{Error estimates on the invariant probability measure.}
For a fixed set of parameters~$\bfomega$, if $\nu_0 \rmPo^n$ converges weakly to a stationary distribution~$\mu_{\bfomega}$, \Cref{thm:main_KL} implies that~$ \| \mu-\mu_{\bfomega}\|_{\mathrm{TV}}$ and~$\mathbf W_2\po \mu,\mu_{\bfomega}\pf$ are both of order~$\mathcal O(\delta^2)$ (using the lower semicontinuity of the entropy function, see~\cite[Theorem~29.1]{Villani_OT09}). This is indeed the expected scaling of the bias on the invariant probability measure of the Markov chain, since the latter one is a composition of two evolution operators, namely~$\rmD_{\eta}$ which exactly preserves~$\mu$, and~$\rmV_{\delta}$ which is a second order approximation of the Hamiltonian evolution operator exactly preserving~$\mu$.

\paragraph{Comparison with previous works.}
There is an abundant literature on the convergence of Markov chain Monte Carlo algorithms based on discretization of continuous-time dynamics, see for instance~\cite{MATTINGLY2002185,Talay,Dalalyan2018OnSF,Dwivedi2,MonmarcheSplitting,ADNR,MangoubiSmith}.  Making a complete survey and a detailed description of these works is beyond the scope of the paper. Here, we focus on works establishing quantitative bounds for gHMC or some of its particular instances, and convergence in the entropic sense for discretizations of the kinetic Langevin diffusion.

\begin{itemize}
    \item \textbf{Entropy methods for discretization of the kinetic Langevin diffusion.} An adaptation of the modified entropy strategy from~\cite{Villani2009} for continuous-time processes has been conducted by~\cite{Chatterji} for a slight modification of the Euler--Maruyama scheme for the kinetic Langevin process. More precisely, the scheme considered in \cite{Chatterji} (introduced in \cite{cheng2018underdamped} in the machine learning community, and referred to as the Stochastic exponential Euler scheme in~\cite{DEMS}, where earlier apparitions in the physics literature are mentioned~\cite{chandrasekhar1943stochastic,ERMAK1980169,Skeel}) is given by the solution of
    \begin{equation}\label{eq:schemeChenetal}
\dd \bfX_t^\delta  =  \bfV_t^\delta \dd t \eqsp, \qquad 
\dd \bfV_t^\delta  =  - \na U(\bfX_{\delta \lceil t/\delta\rceil}^\delta)\dd t - \gamma \bfV_t^\delta \dd t + \sqrt{2\gamma}\dd B_t\,,
\end{equation}
    which is similar to~\eqref{eq:EDScontinuousLangevin2}, except that the force $-\na U$ is constant within time intervals~$[k\delta,(k+1)\delta)$ for~$k\in\N$. 
Hence, in \cite{Chatterji}, the approximate entropy dissipation is established through a continuous-time derivation. In contrast, we follow the discrete time computations of~\cite{MonmarcheIdealized}, which focuses on idealized gHMC, where the exact Hamiltonian dynamics is followed instead of the Verlet integration. In particular, we are not considering in this work the underdamped Langevin diffusion process as a continuous-time reference, since we also cover unadjusted HMC. Moreover, in terms of dependence with respect to the stepsize~$\delta$, the schemes of~\cite{Chatterji} are only of first order 
(see the discussion in~\cite[Section~3.3]{Chatterji}, in particular Equation~(14)), and we obtain an improvement from~$\sqrt{\varepsilon}$ to~$\varepsilon^{1/4}$ in the complexity. Concerning the dimension~$d$, the stepsize in~\cite{Chatterji} scales as $d^{-1/2}L_H^{-1}$, where $L_H$ is the Lipschitz constant of $\na^2 U$ in the Frobenius norm, \ie, $U$ is supposed to satisfy $\|\na^2 U(x)-\na^2 U(y)\|_F \leqslant L_H|x-y|$ for all $x,y\in\R^d$. Assuming that $L_H$ is independent of~$d$ thus corresponds to the weakly interacting case, while the assumption that the derivatives of $U$ are bounded uniformy in~$d$ (in terms of the Euclidean norm) leads to $L_H$ of order $\sqrt{d}$. We get an improvement from $d^{-1/2}$ to $d^{-1/4}$ in the first case, and  the same dependence~$d^{-1}$ in the second. Notice that, in order to take advantage of an higher order  numerical scheme, we have to assume that~$U$ has  bounded first four derivatives, in contrast to~\cite{Chatterji} which only assumes that~$U$ has  bounded first three derivatives. As discussed in Remark~\ref{rem:order1}, if we consider this second condition on~$U$, a result similar to~\cite{Chatterji} and with the same dependence on~$\varepsilon$ and~$d$ can be derived. The main takeaway of this observation is that gHMC can achieve better accuracy when~$U$ is sufficiently smooth, and does not yield a worst complexity than the discretization~\eqref{eq:schemeChenetal} considered in~\cite{Chatterji} under the same conditions.  Finally, up to our knowledge, our result is the first one based on entropy methods for unadjusted HMC, and for splitting discretization schemes of the underdamped Langevin diffusion.
\item \textbf{HMC.}
  Analyses of HMC for strongly convex potentials $U$ have been conducted in~\cite{MangoubiSmith,BouRabeeEberleZimmer}. For non-convex potentials,
  the works~\cite{BouRabeeEberleZimmer,durmus2021asymptotic,BouRabeeEberle,BouRabeeSchuh} have established several non-asymptotic convergence bounds for position HMC (\ie, $\eta=0$), both in Wasserstein $1$ distance and total variation norm. In the mean-field case in particular, \cite[Theorem 10]{BouRabeeSchuh} shows that, after~$\mathcal O(\sqrt{\nq/\varepsilon'})$ gradient evaluations, the law of the chain is at distance at most~$\varepsilon'$ to~$\pi$ in terms of the $\mathbf W_{\ell_1}$ Wasserstein distance associated to the $\ell_1$ norm $\|x-y\|_{\ell_1}  = \sum_{i=1}^{\nq}|x_i-y_i|$.
    To compare this result with ours, we can bound this distance by~$\sqrt{\nq}$ times the standard~$\mathbf W_2$ Wasserstein distance (associated with the Euclidean distance),  whose square is bounded by the relative entropy times the log-Sobolev constant of $\pi$ (which, from the results of~\cite{GuillinWuZhang}, is independent of~$\nq$ in the weakly interacting case where~$\epsilon $ is small enough, which is the regime where this complexity is obtained in~\cite{BouRabeeSchuh}). In other words, we should take $\varepsilon' = \sqrt{\nq \varepsilon}$ in the results of \cite{BouRabeeSchuh} for a fair comparison. In that case, the number of gradient computations required in~\cite[Theorem 10]{BouRabeeSchuh} is   $\mathcal O((\nq/\varepsilon)^{1/4})$ in the weakly interacting case, which is the same as our results.
    The approach in~\cite{BouRabeeEberle,BouRabeeEberleZimmer,BouRabeeSchuh} is based on reflection couplings and concave modifications of the distance (building upon the method developed for continuous-time diffusion processes, in particular~\cite{EBERLE20111101} for the overdamped Langevin process and~\cite{EberleGuillinZimmer} for the underdamped one). In particular, it does not provide a result for the relative entropy or the~$\mathbf W_2$ Wasserstein distance, and the bound on the convergence rate is expressed in terms of global bounds on~$\ps{x-y}{ \na U(x)-\na U(y)} /|x-y|^2$.
    \item \textbf{Splitting schemes for kinetic Langevin diffusion for strongly convex potential $U$.} Splitting schemes for Langevin diffusion with  strongly convex potential energy functions have been investigated in~\cite{MonmarcheSplitting,leimkuhler2023contraction}. Their conclusions on the complexity for the corresponding gHMC method are similar to ours. Finally, the analysis of the general gHMC methodology have been conducted in \cite{monmarche2022hmc} under the same strong convexity condition. Once again, the same conclusions can be drawn.
\end{itemize}
As a conclusion, up to our knowledge, none of the existing works cover our results, which give simple second order explicit estimates in relative entropy uniformly over a class of unadjusted gHMC samplers and for non-convex potential energy functions~$U$. 


\section{Proof of Theorem~\ref{thm:main_KL}}\label{sec:proof-theor-refthm:m}

The proof of Theorem~\ref{thm:main_KL} is an easy consequence of a more detailed result given in Section~\ref{Sec:dissipation}, stated for a modified entropy function mixing a relative entropy and some Fisher information similar to the one considered for the time continuous Langevin dynamics. Instead of time derivatives, we have to estimate the discrete-time evolution of the two parts of the modified entropy along the two steps~$\rmV_\delta$ and~$\rmD_\eta$ of the chain.  The technical details of this analysis are postponed to~\Cref{sec:intermediaryLemmas}, where this evolution is shown to involve, on the one hand, the Jacobian matrix of the Verlet map and, on the other hand, some numerical error terms (due to the fact~$\mu$ is not invariant by~$\rmV_\delta$). The study of these terms is performed respectively in~\Cref{subsec:Psi,sec:numerique}. 

\subsection{A more detailed approximate decay result for a modified entropy}
\label{Sec:dissipation}


\paragraph{Relative density and modified entropy.} 
In the remainder of the present Section~\ref{sec:proof-theor-refthm:m}, for an initial distribution $\nu_0 \in \mathcal{P}(\rsetdd)$ which satisfies the conditions in \Cref{thm:main_KL}, parameters $\bfomega = (K,\delta,\eta)$ and $n \in\nset$, 
we denote by
\begin{equation}
  \label{eq:4}
  h_n = \frac{\rmd\left(\nu_0 \rmP_{\bfomega}^n\right)}{\rmd \mu} \eqsp,
  \qquad 
  \mch(h_n) = \mch(\nu_0 \rmP_{\bfomega}^n | \mu ) \eqsp, 
  \qquad 
  \mci(h_n) = \mci(\nu_0 \rmP_{\bfomega}^n | \mu) \eqsp, 
\end{equation}
where $\mch$ and $\mci$ denote (with some abuse of notation) the relative entropy and Fisher information with respect to~$\mu$ respectively. More generally, we denote by $\mch(h) = \mch(\nu | \mu)$ and $\mci(h) = \mci(\nu |\mu)$ for any probability density $h : \rsetdd \to \rset$ with respect to~$\mu$, so that~$\nu = h \mu$ is a probability measure. In addition,  for $\tilde{\bfA} \in \rset^{2d\times 2d}$, or possibly a matrix field $z \mapsto {\tilde{\bfA}}(z)$ on $\R^{2d}$,  define 
\begin{equation}
  \mathcal I_{\tilde{\bfA}}(\nu | \mu)   =
\begin{cases}
\displaystyle \int_{\R^{2d}} \left| {\tilde{\bfA}} \nabla\ln \po \frac{\rmd \nu}{\rmd \mu }\pf\right|^2 \rmd \nu  & \text{if } \nu \ll \mu, \\
+ \infty & \text{otherwise } \eqsp.
\end{cases}
\end{equation}
Note that if $\nu \ll \mu$ and $h= \rmd \nu / \rmd \mu$, 
\begin{equation}
  \label{eq:FisherA}
  \mathcal I_{\tilde{\bfA}}(\nu | \mu)  =  \int_{\R^{2d}} \frac{|{\tilde{\bfA}} \nabla h|^2}{h} \, \dd \mu\eqsp. 
\end{equation}
Following the convention above, we write $ \mathcal I_{\tilde{\bfA}}(h) = \mathcal I_{\tilde{\bfA}}( \nu  |\mu)$, for any probability density $h : \rsetdd \to \rset$ with respect to $\mu$ with $\nu = h \mu$. 

The main result of this section, from which we deduce \Cref{thm:main_KL}, establishes the dissipation in time of a suitable modified entropy up to some numerical error. More precisely, 
we consider, for $\nu \in\mathcal{P}(\rsetdd)$ and $\tilde{a}>0$,
\begin{equation}
\label{eq:A}
\mathcal{L}_{\tilde{a}}(\nu |\mu) = \mch(h) + \tilde{a} \mathcal I_\bfA(h) \eqsp,
\qquad
\bfA = \frac1{\sqrt2} \begin{pmatrix}
1 & 1 \\ 1 & 1
\end{pmatrix} \,.
\end{equation}
When $\nu \ll \mu$ with $h = \rmd \nu/\rmd \mu$, 
\begin{equation}\label{eq:modifiedEntropy}
    \mathcal L_{\tilde{a}}(\nu |\mu) = \int_{\R^{2d}} h \ln (h) \, \dd \mu+ \tilde{a}\int_{\R^{2d}} \frac{\left|\na_x h  + \na_v h \right|^2}{h} \, \dd \mu\,,
  \end{equation}
  and, following the same convention as previously, we write $\mathcal{L}_{\tilde{a}}(h) = \mathcal{L}_{\tilde{a}}(\nu | \mu)$, for any probability density $h : \rsetdd \to \rset$ with respect to $\mu$ with $\nu = h\mu$. With this choice, $\mathcal L_{\tilde{a}}$ involves a mixed derivative term $\ps{\na_x h}{ \na_v h}$ in the Fisher part, which has been crucially used in hypocoercive studies of the continuous-time kinetic Langevin process \cite{Talay,Herau2007,Villani2009}. More specifically, \eqref{eq:modifiedEntropy} is exactly the modified entropy introduced in continuous-time settings in~\cite[Section~6]{Villani2009}.






\paragraph{Precise statement of the convergence result.}
We are now ready to state the main result of this section. 

\begin{theorem}\label{Prop:dissipation}
Assume \Cref{hyp:target}-\Cref{assu:derivativesUbounded} with $L=1$, and consider~$\bfomega =(K,\delta,\eta)$ with~$K \in\nsets$, $\delta >0$ and $\eta \in \ooint{0,1}$ such that \Cref{assu:moments} holds and $\tint=K\delta \leqslant 1/10$. Furthermore, assume that $\pi$ satisfies the log-Sobolev inequality~\eqref{eq:log-Sob}. Consider an initial condition~$\nu_0$ such that 
$\mathcal I( \nu_0|\mu) < +\infty$ and for all $p >0$, $ \int_{\R^{2d}} \abs{z}^p \rmd \nu_0(z) <\plusinfty$.
Then, for any $n\in\N$,
\begin{equation}\label{eq:contract_n1}
    \po 1+ \kappa  \tint\pf \mathcal L_a(h_{n+1}) \leqslant \mathcal L_a(h_n) + \tint   \delta^4 M \po C_1 \rme^{-\rho nt} \nu_0\po \Wlyap\pf  + C_2\pf\,,
  \end{equation}
  where $a,\kappa,M$ are given in \eqref{eq:a_M_simpler}.
  
Moreover, a refined estimate can be obtained as follows: define
   \[
   m_1 =  2  - 3 \tint\,, 
   \qquad  
   m_2 = \frac{\gammaint}{\eta} + 2  + 4 \tint\,, 
   \qquad 
   m_3 =  \frac{\gammaint(1+\eta)}{\eta^2} \po \frac{1}{2 \ta}+1\pf-2-4\tint\,,
   \]
   and
   \begin{equation}\label{eq:lambda_etaneq0}
     \lambda  =
\frac{m_1+m_3}{2 } - \sqrt{\po \frac{m_1+m_3}{2 }\pf^2 - m_1m_3 + m_2^2}
     \end{equation}
     Then, for any $\tilde{a} >0$ such that $\lambda >0$,  for any $\varepsilon\in(0,1)$ and $n \in\nset$,
\begin{equation}\label{eq:contract_n1_2}
    \po 1+ \tkappa  \tint\pf \mathcal L_{\ta}(h_{n+1}) \leqslant \mathcal L_{\ta}(h_n) + \tint   \delta^4 \tM \po C_1 \rme^{-\rho nt} \nu_0\po \Wlyap\pf  + C_2\pf\,,
  \end{equation}
  where
  \[ 
  \tkappa   = \frac{ \ta \lambda (1-\varepsilon)(1-3\tint)}{\max(L \CLS,1)+2 \ta}\,,
  \qquad 
  \tM = \frac{1}{ \varepsilon \lambda} \po \frac4{15\ta }   +    24 \ta \pf +    \tilde a  \po (33  \lambda +4) \tint  +\frac{\lambda\vareps}{240} \pf   
  \,. \]
\end{theorem}

Note that this result is only stated for $\eta>0$. As explained in the proof of \Cref{thm:main_KL}, the case $\eta=0$ is obtained by a limiting procedure. The expression for~$\lambda$ in~\eqref{eq:lambda_etaneq0} admits a limit when~$\eta\to 0$; more precisely, this limit is 
$2 - 3\tint -   2\ta/[\tint(1+2\ta)]$ which can be obtained from
\[
\lambda = \frac{m_1m_3-m_2^2}{\displaystyle \frac{m_1+m_3}{2 } + \sqrt{\po \frac{m_1+m_3}{2 }\pf^2 - m_1m_3 + m_2^2}},
\]
the numerator and denominator being respectively $\eta^{-2} \gammaint(2-3\tint)(1+1/(2\tilde{a}))$ and~$\eta^{-2} \gammaint [(1+1/(2\tilde{a}))-\gammaint]$ at dominant order in~$\eta$, with~$\gammaint = \tint^{-1}$ when~$\eta=0$ in view of~\eqref{eq:def_tint_gammaint}. 

\begin{remark}
\label{rmk:convergence_implies_LSI}
We discuss in this remark how to obtain~\eqref{eq:convergence_implies_LSI} from the refined convergence result~\eqref{eq:contract_n1_2}. First, since~$L$ is any upper bound of~$\|\na^2 U\|_\infty$, we can assume without loss of generality that~$L \CLS \geq 1$. Upon rescaling, it can therefore be assumed that~$L=1$ and~$\CLS \geq 1$ (see the discussion at the end of \Cref{sec:actual_main_result}). We consider an initial condition of the form $\nu_0 = \hat \nu_0 \otimes \mathrm{N}(0,\IdM)$. Having in mind that the Euler--Maruyama discretization of overdamped Langevin dynamics corresponds to one-step HMC with full resampling of the momenta at each step with effective step-size~$\delta^2/2$, we set~$K=1$ and~$\eta=0$, and therefore consider~$n=2s/\delta^2$ for some time~$s>0$. It follows from~\eqref{eq:def_tint_gammaint} that $\tint = \delta$. Choosing for instance~$\varepsilon = \delta$ (or any other choice ensuring that~$\varepsilon$ vanishes as~$\delta \to 0$) and~$\ta = \alpha \delta$, it is then easily seen that~$\lambda$ is at dominant order~$2(1-\alpha)$, so that the convergence rate~$\tkappa$ is of order~$2\delta \alpha(1-\alpha)/\CLS$. This suggests to choose~$\alpha=1/2$ to maximize the convergence rate. Note that the factor~$(1+\tkappa\tint)^{-n}$ then converges to~$s / \CLS$ as~$\delta \to 0$ with these choices, which leads to~\eqref{eq:convergence_implies_LSI} as claimed.
\end{remark}

\Cref{Prop:dissipation} is proved below. First, let us deduce Theorem~\ref{thm:main_KL} from Theorem~\ref{Prop:dissipation}:

\begin{proof}[Proof of Theorem~\ref{thm:main_KL}] 
 In the case $\eta>0$, an easy induction based on~\eqref{eq:contract_n1} in \Cref{Prop:dissipation} implies that, for any $n\geqslant 1$,
\[
\mathcal L_a(h_{n})  \leqslant \po 1+ \kappa  \tint\pf^{-n} \mathcal L_a(h_0) + \tint  \delta^4 M \ \sum_{k=1}^n \po 1+ \kappa  \tint\pf^{-k}\co \newOld{C_1}\rme^{-\rho (n-k) \tint} \nu_0\po \Wlyap\pf  + C_2\cf \eqsp.
\]
The proof then follows from the bounds 
\[
\sum_{k=1}^n (1+\kappa \tint)^{-k} \leqslant \frac{1+\kappa \tint}{\kappa \tint}\,, 
\qquad 
\sum_{k=1}^n (1+\kappa \tint)^{-k}\rme^{-\rho(n-k) \tint}  \leqslant n \co \max\po (1+\kappa \tint)^{-1/\tint},\rme^{-\rho} \pf \cf^{n\tint} \,, 
\]
as well as~$\mathcal H(h_n) \leqslant \mathcal L_a(h_n)$ and~$\mathcal L_a(h_0) \leqslant \mathcal H(h_0) + 2a \mathcal I(h_0)$. 

The case $\eta=0$ is obtained by letting $\eta$ go to~0 in the previous case.\footnote{We mention that the argument in the case $\eta=0$ in \cite{MonmarcheIdealized} is not correct, an inequality is used in the wrong sense.} Indeed, it is clear that the right hand side of~\eqref{eq:bound_theo} converges. For the left hand side, we use the lower semi-continuity of the relative entropy (see~\cite[Theorem 29.1]{Villani_OT09}) with respect to the weak convergence of probability measures. The convergence of $\nu_0 \rmPo^n$ in the Wasserstein sense (hence in the weak sense) as $\eta\rightarrow 0$ follows from the moment bounds of \Cref{assu:moments}  together with a synchronous coupling argument, i.e. using the same Gaussian variables in the randomization steps of two chains, one with $\eta=0$ and one with $\eta'>0$ (so that the expected distance between the chains goes to zero as $\eta'$ goes to zero).
\end{proof}

\paragraph{Proof of \Cref{Prop:dissipation}.}
Before providing the proof of \Cref{Prop:dissipation}, let us recall the key steps of the proof of the similar result established in \cite{Villani2009} for the continuous-time kinetic Langevin process, as the structure of the proof is similar. Denoting by $(\hat h_t)_{t\geqslant 0}$ the  analogue of $(h_n)_{n\in\N}$ in this continuous-time context, the first observation is that
\begin{equation}\label{eq:Villani1}
    \frac{\rmd}{\rmd t} \mathcal H(\hat h_t) = - \int_{\R^{2d}} \frac{|\na_v \hat h_t|^2}{\hat h_t} \, \dd \mu\,.
\end{equation}
This entropy dissipation is not sufficient to conclude to an exponential decay, as it can vanish for non-constant densities~$\hat h_t$ (and thus in particular it cannot be upper bounded by $-\kappa \mathcal H(\hat h_t)$ for some $\kappa>0$). A key observation is then that, when $\na^2 U$ is bounded,
\begin{equation}\label{eq:Villani2}
\frac{\rmd}{\rmd t}\left( \int_{\R^{2d}} \frac{|\na_x \hat h_t + \na_v \hat h_t |^2}{\hat h_t} \, \dd \mu \right) \leqslant - c \int_{\R^{2d}}  \frac{|\na_x \hat h_t|^2}{\hat h_t} \, \dd \mu + C \int_{\R^{2d}}  \frac{|\na_v \hat h_t|^2}{\hat h_t} \, \dd \mu     
\end{equation}
for some constants $c,C>0$. The last term can be controlled thanks to the entropy dissipation \eqref{eq:Villani1}. As a consequence, for $a$ small enough, considering $\mathcal L_a$ given by \eqref{eq:modifiedEntropy}, 
\begin{equation}\label{eq:Villani3}
\frac{\rmd}{\rmd t} \mathcal L_a(\hat h_t) \leqslant - a c \mathcal I(\hat h_t)\,,    
\end{equation}
and the conclusion follows from
\[\mathcal L_a(\hat h_t) \leqslant (\max(C_{\rm LS},1)+2a)  \mathcal I(\hat h_t)\,, \]
where we used that $\mu$ satisfies a log-Sobolev inequality with constant $\max(C_{LS},1)$ as the tensor product of $\pi$ and $\mathrm{N}(0,\IdM)$. 

The proof of Theorem~\ref{Prop:dissipation} is fundamentally based on the same ingredients, but with discrete-time evolutions along the two steps $\rmV_\delta$ and $\rmD_\eta$ of the chain, instead of the time derivatives of \eqref{eq:Villani1} and \eqref{eq:Villani2}. This requires various estimates given in \Cref{sec:intermediaryLemmas}.

\begin{proof}[Proof of \Cref{Prop:dissipation}]
We start by general estimates, and then specify at the very end particular choices leading respectively to~\eqref{eq:contract_n1} and~\eqref{eq:contract_n1_2}. Fix~$\tilde{a} >0$ and $n \in\nset$.
In this proof, we consider the matrix $\bfA$ given in \eqref{eq:A}. We write the proof for $\eta>0$, the case $\eta=0$ being obtained by passing to the limit~$\eta\to0$ in the final expressions, as discussed after \Cref{Prop:dissipation}.

The estimates on the evolution of~$h_n$ are based on a dual formulation where we consider the evolution of an arbitrary bounded measurable function~$g$:
\begin{equation}
\int_{\R^{2d}} h_{n+1} g \, \dd \mu = \int_{\R^{2d}} h_n  \rmPo g \, \dd \mu \,,
\end{equation}
in other words $h_{n+1}= \rmPo^* h_n$ where we denote by $\rmQ^*$ the adjoint of a bounded linear operator~$\rmQ$ on~$\rml^2(\mu)$. The first task is to identify the action of the adjoints of the operators constituting~$\rmPo$. 
A simple calculation shows that the  damping step is reversible in the probabilistic sense, namely $\rmD_\eta^* = \rmD_\eta $. 
Since the Verlet map $\Phi_\delta^K$ satisfies $|\mathrm{det} \nabla \Phi_{\delta}^K | \equiv 1$ (as a consequence of symplecticity~\cite{HairerLubichWanner06}),
\[
\begin{aligned}
\int_{\R^{2d}} h  \left(\rmV_\delta^K g \right) \dd \mu  
& = \int_{\R^{2d}} h(z) g\po \Phi^{K}_\delta(z)\pf \mu(z) \, \dd z
= \int_{\R^{2d}} h\po \Phi^{-K}_\delta(z)\pf g (z)   \mu\po \Phi^{-K}_\delta(z)\pf\dd z \\
& = \int_{\R^{2d}}  g \left[\po \rmV_\delta^K\pf ^*h \right] \dd  \mu \,, 
\end{aligned}
\]
which allows to identify the action of the adjoint operator as
\begin{equation}
    \label{eq:action_adjoint_Verlet}
\po \rmV_\delta^K\pf ^* h  =  \left(h \circ \Phi^{-K}_\delta\right) \rme^{\square H}\,,
\qquad 
\square H = H- H\circ \Phi^{-K}_\delta  \,.
\end{equation}
In addition, using the reversibility property~\eqref{eq:verlet_inverse_and_refl}, we obtain that $\Phi^{-K}_\delta = (\Phi_\delta^{-1})^{\circ K}$ with
\begin{align}
 \Phi_\delta^{-1}(x,v) &=     \po x- \delta v - \frac{\delta^2}{2} \na U(x)\ , \ v + \frac{\delta}{2}\co \na U(x) + \na U \po x- \delta v - \frac{\delta^2}{2} \na U(x)\pf  \cf \pf \label{eq:phi-1} \,. 
\end{align}

From these observations and since $\rmPo^* = (\rmV_\delta^K)^*\rmD_\eta^*$, we consider the decomposition
\[\mathcal{L}_{\ta}(h_{n+1}) - \mathcal{L}_{\ta}(h_n) = \mathcal{L}_{\ta} (\rmPo^* h_n)-\mathcal{L}_{\ta}(h_n) = \Deltabf_1+\Deltabf_2+\ta \Deltabf_3+\ta \Deltabf_4\,,\]
where 
\[
\begin{array}{lll}
   \Deltabf_1 & =  \mathcal H \po (\rmV_\delta^K)^*\rmD_\eta  h_n\pf -\mathcal H (\rmD_\eta  h_n) \eqsp,
\qquad  & \Deltabf_2 = \mathcal H (\rmD_\eta  h_n) - \mathcal H ( h_n)\\[5pt]
    \Deltabf_3 & = \mathcal I_\bfA \po (\rmV_\delta^K)^*\rmD_\eta h_n\pf \eqsp, 
\qquad 
   & \Deltabf_4=- \mathcal I_\bfA \newOld{(h_n)}\,.
\end{array}
\]
We bound these four terms as follows. By Corollary~\ref{C:improved_ent_verlet}, for any $\varepsilon_1>0$,
\[\Deltabf_1 \leqslant  4 \varepsilon_1 K \delta  \mathcal I  \po \rmD_\eta  h_n \pf + E_{\varepsilon_1} W_n \, ,
\]
where for conciseness we write $W_n = C_1 \rme^{-\rho n\tint} \nu_0(\Wlyap) + C_2$ (recall the notation from \Cref{assu:moments}), and, by  Corollaries~\ref{C:error1} and \ref{C:error2}, \begin{align}
E_{\varepsilon_1}& = \sum_{j=0}^{K-1} \co \newOld{6} \delta^7 j^2 \varepsilon_1     + \varepsilon_1 \int_0^\delta \newOld{2} s^6  \dd s + \frac1{30 \varepsilon_1} \delta^5  \cf \\
& \leqslant  \newOld{2 \vareps_1} \delta^7 K^3 + \frac{\newOld{2}}{7} \varepsilon_1 K \delta^7 + \frac{1}{30 \varepsilon_1} K \delta^5  \leqslant \newOld{\frac1{30}} \left( \varepsilon_1 + \frac{1}{\varepsilon_1}\right) K \delta^5 \eqsp,
\end{align}
where we used that $K\delta\leqslant 1/10$.
 
 
By Lemma~\ref{L:diss_damp},
\[\Deltabf_2\leq -\frac12 (\eta^{-2}-1)  \mathcal I_{\bfB_v}(\rmD_\eta  h_n) \eqsp, \qquad \Deltabf_4\leq \mathcal -\mathcal I_{\bfA \bfB_{\eta}^{-1}}(\rmD_\eta h_n) \eqsp,\]
where
\begin{equation}\label{eq:BvBeta}
\bfB_v = \begin{pmatrix}
0 & 0 \\ 0 & 1
\end{pmatrix}\,,\qquad \bfB_{\eta}^{-1} = \begin{pmatrix}
1 & 0 \\ 0 & 1/\eta
\end{pmatrix}\eqsp .    
\end{equation}

Finally, using Lemma~\ref{L:fisher_verlet} and  \newOld{Corollary~\ref{C:error1}} and setting
\begin{equation}
  \label{eq:proof_def_Psi}
  \Psi =(\na \Phi_\delta^{-K})\circ \Phi_\delta^{K} \eqsp,
\end{equation}
it holds, for any $\varepsilon_2>0$,
\begin{align}
   \Deltabf_3 & \leq   (1+\varepsilon_2)\mathcal I_{\bfA \Psi }(\rmD_\eta h_n) + \newOld{2} ( 1+\varepsilon_2^{-1}) |\bfA|^2  K^2 \delta^6 W_n  \\
   & \leqslant  \mathcal I_{\bfA \Psi }(\rmD_\eta h_n) + \varepsilon_2|\bfA|^2 \|\Psi\|_\infty ^2 \mathcal I(\rmD_\eta h_n) +\newOld{2} ( 1+\varepsilon_2^{-1}) |\bfA|^2  K^2 \delta^6 W_n  \\
     & \leqslant  \mathcal I_{\bfA \Psi }(\rmD_\eta h_n) +  3 \varepsilon_2  \mathcal I(\rmD_\eta h_n) + \newOld{4} ( 1+\varepsilon_2^{-1})   K^2 \delta^6 W_n\,, 
\end{align}
where we used $|\bfA|=\sqrt{2}$ and~$\|\Psi\|_\infty\|^2 \leq 3/2$, the latter bound following from Equation~\eqref{eq:normPsi} in \Cref{lem:VerletMatrice} and $\tint = K\delta  \leqslant 1/10$.

Therefore, by combining the inequalities above, we get 
\begin{align} 
\mathcal{L}_{\ta}(h_{n+1}) - \mathcal{L}_{\ta}(h_n) \ \leqslant \ - \frac{\eta^{-2}-1}{2} \mathcal I_{\bfB_v}(\rmD_\eta  h_n) 
+ \ta \mathcal I_{\bfA \Psi }(\rmD_\eta h_n)   - \ta\mathcal I_{\bfA \bfB_{\eta}^{-1}}(\rmD_\eta h_n)\\
+ \po 4 \varepsilon_1 K \delta + 3 \ta \varepsilon_2 \pf   \mathcal I  \po \rmD_\eta  h_n \pf + \co \frac{\varepsilon_1 + \varepsilon_1^{-1}}{30} K \delta^5 +  \newOld{4} \ta ( 1+\varepsilon_2^{-1})   K^2 \delta^6\cf  W_n\,.\label{eq:ent_diff1}
\end{align}
The sum of the terms in the first line of the right hand side is equal to 
\begin{equation}
  - \ta \mathcal I_{\bfS^{1/2}}(\rmD_\eta h_n)=- \ta \int_{\R^{2d}} \frac{\ps{\na \left(\rmD_\eta h_n\right)} { \bfS \na \left(\rmD_\eta h_n\right)}}{\rmD_\eta h_n} \, \dd \mu \eqsp,
  \label{eq:def_S_proof}  
\end{equation}
with, for $z\in\R^{2d}$,
\[\bfS(z) =   (2\ta)^{-1} (\eta^{-2}-1) \bfB_v^{\top} \bfB_v -  (\bfA\Psi(z))^{\top} \bfA\Psi(z) +  ( \bfA \bfB_{\eta}^{-1})^{\top} \bfA \bfB_{\eta}^{-1},
  \]
where $\Psi$ is defined in \eqref{eq:proof_def_Psi}.

A key step of the proof is now to establish a positive lower bound on $\bfS$ uniformly in $z$, for $\tilde{a}$ small enough, which is the analogue in our context of~\eqref{eq:Villani3}.  We introduce the $d\times d$ block decompositions
\[
\Psi(z)=\begin{pmatrix} \psi_{11}(z) & \psi_{12}(z) \\ \psi_{21}(z) & \psi_{22}(z) \end{pmatrix},
\qquad 
\bfS(z)=\begin{pmatrix} s_{11}(z) & s_{12}(z) \\ s_{21}(z) & s_{22}(z) \end{pmatrix}\,,
\]
 so that
\begin{align}
  s_{11}&=  \IdM -(\psi_{11}+\psi_{21})^{\top}(\psi_{11}+\psi_{21})\eqsp, \quad      s_{12}= s_{21}  =  \eta^{-1} \IdM  - (\psi_{11}+\psi_{21})^{\top}(\psi_{12}+\psi_{22})\eqsp, \\
   s_{22}  &=  \po \frac{\eta^{-2}-1}{2\ta} + \eta^{-2} \pf \IdM  -  (\psi_{22}+\psi_{12})^{\top}(\psi_{22}+\psi_{12})\,.
\end{align}
Using  Lemma~\ref{lem:VerletMatrice} and~$\tint \leq 1/10$, it holds, for any $x\in\R^d$,
\begin{align}
   | (\psi_{11}+\psi_{21})x| & \leqslant  \left|\po (1-\tint) \IdM  - \frac{\tint^2}{2} \alpha \pf x \right| + \frac{2 \tint^3}{3} |x|\leq \po 1 - \tint + \tint^2 \pf   |x| \\
    | (\psi_{22}+\psi_{12})x| & \leqslant  \left|\po  \IdM  + \tint\beta  - \frac{\tint^2}{2} \zeta  \pf x \right| + \frac{2 \tint^3}{3} |x|\leq \po 1 + \tint + \tint^2 \pf   |x|  \eqsp,
\end{align}
and similarly 
\begin{align}
     | (\psi_{11}+\psi_{21})x - x | & \leqslant  \po \tint + \tint^2 \pf   |x| \eqsp, \qquad 
     | (\psi_{22}+\psi_{12}) x - x |  \leqslant  \po \tint + \tint^2 \pf   |x| \,,
\end{align}
so that, for any $x,y\in \R^d$, 
\begin{align}
\ps{x}{   s_{11} x} & \geqslant  \co 1 - \po 1 -\tint + \tint^2 \pf^2 \cf    |x|^2
 \geqslant   \po  2 \tint - 3 \tint^2\pf   |x|^2 \eqsp, \\
\ps{y}{   s_{22} y} & \geqslant  \co \frac{\eta^{-2}-1}{2\ta} + \eta^{-2} - \po 1 + \tint + \tint^2 \pf^2 \cf    |y|^2
 \geqslant   \po \frac{\eta^{-2}-1}{2\ta} + \eta^{-2} - 1 -2\tint - 4 \tint^2\pf    |y|^2 \eqsp, \\  
\ps{   x}{   s_{12} y} & =  \po \eta^{-1}-1\pf   x \cdot  y + x \po \IdM - (\psi_{11}+\psi_{21})^{\top}(\psi_{12}+\psi_{22}) \pf  y \\
   & =  \po \eta^{-1}-1\pf   x \cdot  y +  \ps{ \po \IdM- \newOld{(\psi_{11}+\psi_{21})}\pf x }{ (\psi_{12}+\psi_{22}) y}  +\ps{ x}{  \po\IdM- \newOld{(\psi_{12}+\psi_{22})}\pf  y }\\
   & \geqslant  - \po \eta^{-1}-1 +  \po 2+ \tint+ \tint^2 \pf  \po \tint+ \tint^2 \pf\pf |x||y|\\
   & \geqslant  - \po \eta^{-1}-1 +  2 \tint + 4\tint^2 \pf |x||y|\,.
\end{align}
Note that the last right hand sides in the previous three inequalities are of order~$\tint$ at dominant order in~$\tint$ (using~\eqref{eq:def_tint_gammaint} to write~$\eta = 1-\gammaint\tint$). This motivates factoring out~$\tint$, and writing, for any $x,y\in \R^d$ and $z\in\R^{2d}$,
\begin{equation}
\label{eq:lower_bound_S_matrix}
\begin{pmatrix} x \\ y\end{pmatrix}^{\top} \bfS(z) \begin{pmatrix} x\\ y\end{pmatrix} \geqslant \tint \lambda \po |x|^2 +|y|^2\pf \eqsp, 
\end{equation}
where $\lambda$ is the smallest eigenvalue of the symmetric $2\times 2$ matrix
\begin{align}
   \begin{pmatrix}
m_1 & m_2 \\ m_2 & m_3 
\end{pmatrix} &:= \frac1{\tint} \begin{pmatrix}  2 \tint - 3 \tint^2 & - \po \eta^{-1}-1 +  2 \tint + 4\tint^2 \pf\\ - \po \eta^{-1}-1 +  2 \tint+ 4\tint^2 \pf& \displaystyle \frac{\eta^{-2}-1}{2\ta} + \eta^{-2} - 1 -2\tint - 4 \tint^2 \end{pmatrix} \\
&=       \bfB_{\eta}^{-1} 
\begin{pmatrix}  2  - 3 \tint & - \po  \gammaint   +  \eta(2  + 4\tint) \pf\\ - \po \gammaint   + \eta( 2 + 4\tint ) \pf& \displaystyle \gammaint (1+\eta) \po \frac{1}{2\ta}+1\pf    - \eta^2 (2  + 4\tint) \end{pmatrix}
\bfB_{\eta}^{-1}\,,
\end{align}
recalling the notation $\gammaint = (1-\eta)/\tint$. This provides the definitions given in~\eqref{eq:def_m_i_modified_entropy}. The expression~\eqref{eq:lambda_etaneq0} for~$\lambda$ then follows from
\begin{equation}
\lambda  =  \frac{m_1+m_3}{2 } - \sqrt{\po \frac{m_1+m_3}{2 }\pf^2 - m_1m_3 + m_2^2}\eqsp,
\end{equation}
which is positive provided~$\tilde{a}$ is small enough. 
Alternatively, to get the simpler expressions~\eqref{eq:a_M_simpler}, we proceed as follows:  using that $\eta\leqslant 1$ and  $\tint\leqslant 1/10$, we bound
\begin{equation}
    \label{eq:def_m_i_modified_entropy}
   m_1  \geqslant   \frac{17}{10} \,, \qquad 
    |m_2|   \leqslant   \frac{ 1}{\eta } \po \gammaint + \frac{12}5\pf  \,,  \qquad 
    m_3    \geqslant   \frac{ 1}{\eta^2} \po \frac{\gammaint}{2\ta}- \frac{12}5 \pf    \eqsp,  
\end{equation}
to get that
    \begin{align}
       \begin{pmatrix} x\\y\end{pmatrix}^{\top} \begin{pmatrix} m_1& m_2\\ m_2 & m_3\end{pmatrix}\begin{pmatrix} x\\ y\end{pmatrix} & =  m_1 |x|^2 + 2 m_2 |x||y| + m_3 |y|^2 \\
       & \geqslant \frac12 m_1 |x|^2   + \po  m_3 - \frac{2m_2^2}{m_1}\pf  |y|^2 \geqslant   \frac{17}{20}  \po |x|^2   + |y|^2  \pf \,,
    \end{align}
          \ie, $\lambda \geqslant 17  /20$ if we choose $\tilde{a}$ such that~$m_3 \geq m_1 /2 + 2m_2^2/m_1$. This condition is satisfied for any~$\eta\in(0,1]$ when it is satisfied for~$\eta=1$, in which case it reads
          \begin{equation}
         \frac{\gammaint}{2\ta} \geqslant \frac{17}{20} + \frac{12}5 + \frac{20}{17}\po \gammaint + \frac{12}{5}\pf^2 \,.\label{eq:13}
       \label{eq:def_lambda_moins_condtion}
     \end{equation}
 This holds in particular with the choice of $\ta=a$ in \eqref{eq:a_M_simpler}. 

In both cases, at this point, plugging the lower bound~\eqref{eq:lower_bound_S_matrix} in~\eqref{eq:def_S_proof}, we have determined $\lambda > 0$ such that
\[
\ta \mathcal I_{\bfS^{1/2}}(\rmD_\eta h_n) \leqslant - \ta\lambda \tint \mathcal I(\rmD_\eta h_n)\,. 
\]
Combining this result with  \eqref{eq:ent_diff1} yields 
\begin{equation}\label{eq:Lhn_intermediary}
      \mathcal{L}_{\ta}(h_{n+1}) - \mathcal{L}_{\ta}(h_n) \leqslant  -\rho '  \mathcal I  \po \rmD_\eta  h_n \pf  
    + \co \frac{\varepsilon_1 +\varepsilon_1^{-1}}{30} K \delta^5 +  \newOld{4} \ta ( 1+\varepsilon_2^{-1})   K^2 \delta^6\cf  W_n\,,
\end{equation}
for
\begin{equation}
  \label{eq:def_rho_prime_rppfo}
  \rho' =  \ta \lambda \tint -   4 \varepsilon_1 \tint - 3 \ta \varepsilon_2 \eqsp.
\end{equation}

 Now, applying Lemma~\ref{L:fisher_verlet} with $\tilde{\bfA} = \Psi^{-1}\circ \Phi_\delta^{-K}$ and~$h$ replaced by~$\rmD_\eta  h_n$ (see Remark~\ref{rem:Fisher_Verlet_MatrixField}), we obtain that, for any~$\varepsilon_3>0$
\[(1+\varepsilon_3) \mathcal I (\rmD_\eta h_n )  \geq \mathcal I_{\Psi^{-1}\circ \Phi_\delta^{-K}}  \po \rmPo^* h_n \pf - (1+\varepsilon_3^{-1})\left\|\Psi^{-1}\circ \Phi_\delta^{-K}\right\|_\infty^2 \mathrm{Er}_1^{\delta,K}(\rmD_\eta h_n)\,.\]
Thanks to Lemma~\ref{lem:VerletMatrice}, and more specifically to \eqref{eq:normPsi}, for any $z,u\in\R^{2d}$,
\[
\begin{aligned}
   |\Psi^{-1}(z) u|^2 & \geqslant \po 1+\tint +  \tint^2/2 +  \tint^3/3\pf^{-2} |u|^2 \\
   & \newOld{\geqslant} \newOld{ (1+5\tint/2)^{-1}|u|^2 }\\
   &\geqslant \po 1 - 5 \tint /2\pf |u|^2\,,
\end{aligned}
\]
so that~$\mathcal I_{\Psi^{-1}\circ \Phi_\delta^{-K}}  \po \rmPo^* h_n \pf \geqslant \po 1 - 5 \tint /2\pf  \mathcal I   \po \rmPo^* h_n \pf$. Using Lemma~\ref{lem:VerletMatrice} again, we have, for any $z,u\in\R^{2d}$,
\[|u| \leqslant |\Psi(z) u| + |\Psi(z) u - u| \leqslant  |\Psi(z) u| + \po \tint + \fracmm{\tint^2}{2}+ \fracmm{\tint^3}{3}\pf |u|\,, \]
and thus 
\[
\left|\Psi^{-1}(z) u\right| \leq |u| + \po \tint + \fracmm{\tint^2}{2}+ \fracmm{\tint^3}{3}\pf \left|\Psi^{-1}(z) u\right| \leq |u| + \frac12 \left|\Psi^{-1}(z) u\right|,
\]
so that~$\|\Psi^{-1}\circ \Phi_\delta^{-K}\|_\infty\leqslant 2$. Plugging these bounds in  \eqref{eq:Lhn_intermediary} yields
\begin{align}
   \mathcal{L}_{\ta}(h_{n+1}) - \mathcal{L}_{\ta}(h_n) &\leqslant   -\frac{\rho '}{1+\varepsilon_3}  \co \po 1 - \frac{5 \tint}{2}\pf  \mathcal I  \po h_{n+1}\pf - 4(1+\varepsilon_3^{-1})  \mathrm{Er}_1^{\delta,K}(\rmD_\eta h_n)\cf   \\
   &\qquad 
    + \co \frac{\varepsilon_1 + \varepsilon_1^{-1}}{30} K \delta^5 + \newOld{4} \ta ( 1+\varepsilon_2^{-1})   K^2 \delta^6\cf  W_n\,.
\end{align}
We next use that~$\mu$ satisfies a log-Sobolev inequality with constant~$\max(1,\CLS)$ (by tensorization from the log-Sobolev inequalities satisfied by~$\pi$ and~$\mathrm{N}(0,\IdM)$), and the equality~$|\bfA|^2= 2$, to write 
\[
\mathcal{L}_{\ta}(h_{n+1}) \leq \left( \max(1,\CLS)+2\ta \right)\mathcal{I}\po h_{n+1}\pf.
\]
Choosing $\varepsilon_3 = \tint/4$, using Corollary~\ref{C:error1} to bound the remaining numerical error term, setting 
\[
M =    \newOld{8} \newOld{\rho'}\po 1+ \frac{4}{\tint} \pf \tint  + \frac{\varepsilon_1 + \varepsilon_1^{-1}}{30}  +  \newOld{4} \ta ( 1+\varepsilon_2^{-1})   \tint  \,, 
\]
and noting that $(1-5\tint/2)/(1+\tint/4) \geq 1-3\tint$ and~$\rho'(1-3\tint)>0$, we can conclude that
\begin{align}
 \mathcal{L}_{\ta}(h_{n+1}) - \mathcal{L}_{\ta}(h_n)  & \leqslant   -\rho'(1-3\tint) \mathcal I  \po h_{n+1}\pf + \tint \delta^4 M  W_n  \\
 & \leqslant  -\frac{\rho'(1-3\tint)}{\max(1,\CLS)+2\ta} \mathcal{L}_{\ta}  \po h_{n+1}\pf + \tint \delta^4 M W_n\,.
\end{align}

Let us finally discuss how to obtain~\eqref{eq:contract_n1} and~\eqref{eq:contract_n1_2}:
\begin{itemize}
\item to obtain \eqref{eq:contract_n1} with the simpler expressions of $\kappa$ and $M$ in~\eqref{eq:a_M_simpler}, we first note that~$a$ in~\eqref{eq:a_M_simpler} is smaller than~$1/10$, as can be seen for instance by distinguishing the case~$\gamma \leq 3$, for which $a \leq \gamma/34$, and the case~$\gamma \geq 3$, for which $a \leq (\gamma+3)/[3(\gamma+3)^2]$). We next rely on the expression~\eqref{eq:def_rho_prime_rppfo} for~$\rho'$, with the choices~$\varepsilon_1 = a/20$ and~$\varepsilon_2 = \tint/20$, so that, replacing~$\lambda$ by~$17/20$,  
\[
\rho' = a\tint\left(\lambda-\frac{1}{5}-\frac{3}{20}\right) = \frac{a\tint}{2}\,.
\]
The value of~$\kappa$ is then obtained as
\[
\kappa = \frac{\rho'(1-3\tint)}{\max(1,\CLS)+2a} = \frac{a(1-3\tint)}{2\left[\max(1,\CLS)+2a\right]} \geq \frac{a}{3\left[\max(1,\CLS)+2a\right]}\,.
\]
In the expression of $M$, we simply bound $\rho' \leqslant 1/200$, $a\leqslant 1/10$, $\tint\leqslant 1/10$.
\item To get the sharper inequality \eqref{eq:contract_n1_2}, we keep a free parameter $\varepsilon\in(0,1)$ and choose   $\varepsilon_1 = \fracmm{\tilde a\varepsilon \lambda}{8}$, $\varepsilon_2 =  \fracmm{\tint \varepsilon \lambda}{6}$ which implies that $ \rho' = \tilde a \lambda \tint (1-\varepsilon)$ (which immediately leads to the claimed value for~$\kappa$) and
\[
M = \newOld{8}  \tilde a \lambda \tint  \po \tint + 4\pf  +   \frac{\ta\lambda\vareps}{240} + \frac4{15\ta\lambda\vareps}   +  \newOld{4} \ta \left( \tint +\frac6{ \varepsilon \lambda} \right), \]
and using that $\tint\leqslant 1/10$ gives the claimed result.
\end{itemize}
This allows to conclude the proof of \Cref{Prop:dissipation}.
  \end{proof}
  

  In the next sections, we derive the technical results that we used in the proof of \Cref{Prop:dissipation}. In Section~\ref{sec:intermediaryLemmas}, we bound $\mathcal{L}_{\ta}(h_{n+1})$ in terms of $h_n$, the Jacobian matrix of the (reverse) Verlet map $\Phi_\delta^{-K}$ and some  numerical error terms. The latters are studied respectively in Sections~\ref{subsec:Psi} and \ref{sec:numerique}.
  

\subsection{Evolution of entropies under elementary steps}\label{sec:intermediaryLemmas}

In this section, we compute the evolution of the relative entropy and Fisher-like terms \eqref{eq:FisherA} along the damping and Verlet steps. The velocity randomization step is the same as in  \cite[Section 4.1.1]{MonmarcheIdealized},  where the following is established.

\begin{lemma}[Dissipation from velocity randomization] \label{L:diss_damp}
Recalling from \eqref{eq:BvBeta} the notation
\[\bfB_v = \begin{pmatrix}
0 & 0 \\ 0 & 1
\end{pmatrix}\,,\qquad \bfB_{\eta} = \begin{pmatrix}
1 & 0 \\ 0 & \eta
\end{pmatrix}\,,\]
then, for all $\eta\in(0,1)$ and all measurable~$h$ with~$\mathcal H (h) < +\infty$ ,
\[\mathcal H (\rmD_\eta h )-\mathcal H (h )\leq -\po\frac{\eta^{-2}-1}{2}\pf \mathcal I_{\bfB_v} (\rmD_\eta  h).\]
and, for any matrix $\tilde{\bfA}$,
\[
\mathcal I_{\tilde{\bfA}\,\bfB_{\eta}^{-1}} (\rmD_\eta h )\leq \mathcal I_{\tilde{\bfA}} (h).
\]
\end{lemma}

\begin{proof}
The first inequality is equivalent to
\[
0 \leq \mathcal H (\rmD_\eta h ) + \po\frac{\eta^{-2}-1}{2}\pf \mathcal I_{\bfB_v} (\rmD_\eta  h) \leq \mathcal H (h ) < +\infty\,,
\]
which is established in~\cite[Equation~(45)]{MonmarcheIdealized} (it is sufficient to prove the bound for a smooth positive~$h$ with compact support and pass to the limit). 

The second point follows from Jensen's inequality and the convexity of $(t,x)\mapsto \norm{x}^2/t$ on $\rset^*_+ \times \rset^d$ (using \cite[Proposition 2.3 (ii)]{combettes2018perspective}) which imply
\[\int_{\R^{2d}} \frac{| \tilde{\bfA} \na h|^2 }{h}\dd \mu 
=\int_{\R^{2d}} \rmD_\eta\po \frac{| \tilde{\bfA} \na h|^2 }{h}\pf \dd \mu  
\geqslant  \int_{\R^{2d}} \frac{|\rmD_\eta \tilde{\bfA}  \na h|^2 }{\rmD_\eta h}\dd \mu  =  \int_{\R^{2d}} \frac{| \tilde{\bfA} \bfB_{\eta}^{-1} \na \rmD_\eta h|^2 }{\rmD_\eta h}\dd \mu \,,\]
where the last equality comes from~$\na_v (\rmD_\eta h) = \eta \rmD_\eta (\nabla_v h)$ and~$\rmD_\eta (\nabla_x h) = \nabla_x (\rmD_\eta h)$.
\end{proof}

\begin{lemma}[Fisher inequality on Verlet step] 
\label{L:fisher_verlet}
For all matrices~$\tilde{\bfA}$, all smooth~$h$ with~$\mathcal I(h) <+\infty$, and all~$\varepsilon>0$,
\[
\mathcal I_{\tilde{\bfA}}\po (\rmV_\delta^K)^* h \pf \leq (1+\varepsilon) \mathcal I_{\tilde{\bfA}\Psi } (h )+(1+\varepsilon^{-1})\left|\tilde{\bfA}\right|^2 \mathrm{Er}_1^{\delta,K}(h) \, ,
\]
with $\Psi =(\na \Phi_\delta^{-K})\circ \Phi_\delta^{K}$   and the error term
\[\mathrm{Er}_1^{\delta,K}(h  ) = \int_{\R^{2d}} \left|\left(\na \square H\right)\circ\Phi_\delta^{K} \right|^2 h \, \dd \mu \,. 
\]
\end{lemma}

\begin{proof}
By an approximation argument, it can be assumed that~$h$ is bounded below by a positive constant and globally Lipschitz. Recalling the action of~$(\rmV_\delta^K)^*$ from~\eqref{eq:action_adjoint_Verlet},
\[
\begin{aligned}
 \int_{\R^{2d}} \frac{\left|\tilde{\bfA}\na  (\rmV_\delta^K)^*  h\right|^2}{(\rmV_\delta^K)^* h} \, \dd \mu & =  \int_{\R^{2d}} \frac{\left|\tilde{\bfA}\na  \po (h\circ \Phi_\delta^{-K}) \rme^{\square H} \pf \right|^2}{(h\circ \Phi_\delta^{-K})\rme^{\square H}} \, \dd \mu \\
& \qquad = \int_{\R^{2d}} \frac{\left|\tilde{\bfA}\co \na (\Phi_\delta^{-K})(\na  h)\circ\Phi_\delta^{-K}- (h \circ \Phi_\delta^{-K}) \na \square H\cf  \rme^{\square H} \right|^2}{(h\circ \Phi_\delta^{-K})\rme^{\square H}} \, \dd \mu \\
& \qquad = \int_{\R^{2d}} \frac{\left|\tilde{\bfA}\co \na (\Phi_\delta^{-K})(\na  h)\circ \Phi_\delta^{-K} - (h \circ \Phi_\delta^{-K})\na \square H\cf \right|^2}{h\circ \Phi_\delta^{-K} }\rme^{\square H} \, \dd \mu \\
& \qquad = \int_{\R^{2d}} \frac{\left|\tilde{\bfA}\co \po \na (\Phi_\delta^{-K} )\circ \Phi_\delta^{K}\pf \na  h  - h  (\na \square H)\circ \Phi_\delta^{K} \cf \right|^2}{h  } \,  \dd \mu\,,
\end{aligned}
\]
where we used the change of variable $\Phi_\delta^{-K}$ in the last line. Introducing~$ \Psi = \na (\Phi_\delta^{-K} )\circ \Phi_\delta^{K}$, we therefore obtain, with a discrete Cauchy--Schwarz inequality: for any~$\varepsilon>0$,
\[
\int_{\R^{2d}} \frac{|\tilde{\bfA}\na  (\rmV_\delta^K)^*  h|^2}{(\rmV_\delta^K)^* h} \, \dd \mu \leqslant (1+\varepsilon) \int_{\R^{2d}} \frac{|\tilde{\bfA} \Psi \na  h  |^2}{h  } \, \dd \mu + \po 1 + \frac1\varepsilon \pf\left|\tilde{\bfA}\right|^2 \int_{\R^{2d}} \left|\na \square H\circ \Phi_\delta^{K} \right|^2 h \, \dd \mu  \,,
\]
which is indeed the claimed estimate.
\end{proof}

\begin{remark}\label{rem:Fisher_Verlet_MatrixField}
If the constant matrix $\tilde{\bfA}$ is replaced by a matrix field $z\mapsto \tilde{\bfA}(z)$,  a straightforward adapation of the proof yields
\[\mathcal I_{\tilde{\bfA}} \po (\rmV_\delta^K)^* h \pf \leq (1+\varepsilon) \mathcal I_{(\tilde{\bfA}\circ \Phi_\delta^K) \Psi } (h )+(1+\varepsilon^{-1}) \left\|\tilde{\bfA}\right\|_\infty^2 \mathrm{Er}_1^{\delta,K}(h)
\]
with $\displaystyle \left\|\tilde{\bfA}\right\|_\infty = \sup_{z\in\R^{2d}}|\tilde{\bfA}(z)|$.
\end{remark}

\begin{lemma}[Entropy inequality on Verlet step]\label{lem:entropy_verlet}
For all smooth~$h$ with $\mathcal I(h)<+\infty$ and all $\varepsilon>0$,
\[
\mathcal H \po (\rmV_\delta^K)^* h\pf -\mathcal H (h  )\leq
\sum_{j=0}^{K-1} \co \frac{\varepsilon}{2} \int_0^\delta \mathcal I \po  (\rmV_\delta^j \rmV_s^1)^* h \pf \,ds  + \frac1{2\varepsilon} \mathrm{Er}_2\po (\rmV_{\delta}^j)^* h\pf  \cf \,,
\]
with the error term
\[
  \mathrm{Er}_2( h)  = \int_0^\delta\po \int_{\R^{2d}}|\partial_s\Phi_s -F_H\circ \Phi_s|^2\,h\,d\mu\pf \rmd s  \,,
\]
where 
\begin{equation}
    \label{eq:def_F_H}
    F_H(x,v)=\begin{pmatrix}
v\\
-\na U(x)
\end{pmatrix}\,.
\end{equation}
\end{lemma}

\begin{proof}
In view of the decomposition
\[
\mathcal H \po (\rmV_\delta^K)^* h\pf -\mathcal H (h  ) = 
\sum_{j=0}^{K-1} \co \mathcal H \po (\rmV_\delta^{j+1})^* h\pf -\mathcal H \po (\rmV_\delta^j)^* h\pf\cf\,,  
\]
we simply have to establish the result for $K=1$.

 For $z\in\R^{2d}$  and $t\in[0,\delta]$, let $z_t  =   \Phi_{t}       (z)$.  For $f_0$  an initial (smooth positive) distribution, denote by $f_t$ the law of  $z_t$ when $z\sim f_0$ and $h_t=f_t/\mu$. By a change a variable we see that, for any smooth function~$g$ on~$\R^{2d}$,
\[
\int_{\R^{2d}} g(z)   f_t(z) \, \dd z  =  \int_{\R^{2d}} g \po  z_t\pf  f_0(z) \, \dd z  = \int_{\R^{2d}} g(z)   f_0 \po \Phi_t^{-1}(z)\pf \dd z\,,  
\]
\ie~$f_t(z) =  f_0 \po \Phi_t^{-1}(z)\pf$, and
\begin{align}
\int_{\R^{2d}} g(z) \partial_t f_t(z) \, \dd z 
 = \frac{\rmd}{\rmd t}\left( \int_{\R^{2d}} g \po z_t\pf  f_0(z) \, \dd z  \right)  &= \int_{\R^{2d}} \partial_t \Phi_t(z) \cdot \nabla g( \Phi_t(z)) f_0(z) \, \dd z \notag \\
& = \int_{\R^{2d}} j_t(z)\cdot \nabla g(z) f_t(z) \, \dd z \,, \label{ed:d_ds_int_g_f0}
\end{align}
with $j_t = \partial_t \Phi_t\circ \Phi_t^{-1}$. Then,
\[
\int_{\R^{2d}}  f_\delta \ln h_\delta - \int_{\R^{2d}} f_0\ln h_0 = \int_0^\delta \frac{\rmd}{\rmd s}\left( \int_{\R^{2d}} f_s\ln h_s \, \dd s\right),
\]
and
\[
\frac{\rmd}{\rmd s}\left( \int_{\R^{2d}} f_s\ln h_s \right) = \int_{\R^{2d}} \partial_s h_s \, \dd \mu + \int_{\R^{2d}} \left(\ln h_s\right) \partial_s f_s\,.
\]
The first term of the right hand side vanishes as it is the time derivative of the integral of~$h_s$, which is~1 for any~$s \geq 0$. Combining the latter equality with~\eqref{ed:d_ds_int_g_f0} then gives 
\[
\frac{\rmd}{\rmd s}\left( \int_{\R^{2d}} f_s\ln h_s \right) = \int_{\R^{2d}} j_s\cdot  \na (\ln h_s)  f_s   \,.
\]
Since~$F_H \cdot \na \mu = 0$ and~$F_H$ is divergence free, 
\[\int_{\R^{2d}}  F_H \cdot \na (\ln  h_s) f_s = \int_{\R^{2d}}  F_H \cdot \na h_s \, \dd\mu = - \int_{\R^{2d}} h_s \left( F_H \cdot \na \mu + \mu \, \mathrm{div} F_H \right)= 0\,.
\]
Thus, we can add this term to the previous equality to get, for any $\varepsilon>0$,
\begin{align}
\frac{\rmd}{\rmd s}\left( \int_{\R^{2d}} f_s\ln h_s \right)
& = \int_{\R^{2d}} \po j_s(z) - F_H(z)\pf \cdot  \na \ln h_s(z)  f_s(z) \, \dd z  \\
&  \leqslant \frac{1}{2\varepsilon}\int_{\R^{2d}}\left| j_s(z) - F_H(z) \right|^2 f_s(z) \, \dd z + \gabriel{\frac{\varepsilon}{2}} \int_{\R^{2d}}|\na  \ln h_s(z)|^2  f_s(z) \, \dd z \\
& = \frac{1}{2\varepsilon} \int_{\R^{2d}}\left| \partial_s \Phi_s - F_H \circ\Phi_s \right|^2 h_s \, \dd \mu + \frac{\varepsilon}{2} \int_{\R^{2d}}|\na  \ln h_s(z)|^2  f_s(z) \, \dd z \,,
\end{align}
which gives  the announced result once integrated over $s\in[0,\delta]$.
\end{proof}

At this stage, gathering Lemmas~\ref{L:diss_damp}, \ref{L:fisher_verlet} and \ref{lem:entropy_verlet} yields a bound on~$\mathcal L_a(h_{n+1})$ involving~$h_n$, the matrix field $\Psi$ of Lemma~\ref{L:fisher_verlet} and the numerical error terms of  Lemmas~\ref{L:fisher_verlet} and \ref{lem:entropy_verlet}. It remains to understand these parts. We give some estimates on $\Psi$ in Section~\ref{subsec:Psi} and analyse the numerical errors in Section~\ref{sec:numerique}. 

\subsection{Jacobian matrix of the Verlet integrator}\label{subsec:Psi}

As mentioned above, this section focuses on the matrix field $\Psi = (\na\Phi_\delta^{-K})\circ\Phi^K_\delta$ appearing in Lemma~\ref{L:fisher_verlet}. The following result is the analogue for the Verlet scheme of the first part of \cite[Lemma~1]{MonmarcheIdealized} for the  Hamiltonian dynamics.

\begin{lemma}\label{lem:VerletMatrice}
Under \Cref{hyp:target} with $L=1$, for all $z\in\R^{2d}$, there exist $d\times d$ matrices   $\alpha,\beta,\zeta $ with norms less than $1$ such that, recalling the notation $\tint=K\delta$,
\[\left | \Psi (z) - \begin{pmatrix}
\displaystyle \IdM- \frac{\tint^2}2  \alpha    &  \tint  \beta \\ -\tint   & \displaystyle \IdM- \frac{\tint^2}2   \zeta
\end{pmatrix} \right|\leqslant  \frac13 \tint^3\,.   \]
In particular, 
\begin{equation}\label{eq:normPsi}
    |\Psi(z)| \leqslant 1+ \tint + \frac12 \tint^2 + \frac13 \tint^3 \eqsp.
\end{equation}
\end{lemma}

\begin{proof}
Denote by $\mathcal M_1$ the set of $d\times d$ matrices with operator norm bounded by $1$. Note that the product of two elements in~$\mathcal{M}_1$ is still in~$\mathcal{M}_1$. Fix $z\in\R^{2d}$. For $k\geqslant 1$, set $z_{k} = \Phi_\delta^{-1}(z_{k-1})$, with $z_0=\Phi^K_\delta(z)$. For  $k\geqslant 0$, set $Q_k = \na^2 U (x_k )$, where $z_{k} = (x_k,v_k)$. By assumption, $Q_k\in\mathcal M_1$ for all $k\geqslant 0$.

Recall that, for $F:\R^d\rightarrow \R^d$, we use the convention $\na F = (\partial_i F_j)_{1\leqslant i,j \leqslant d}$ where $i$ stands for the row and $j$ for the column. Hence, from \eqref{eq:phi-1}, for all $k\geqslant 0$,
\[\na (\Phi_{\delta}^{-1})(z_k) = 
\begin{pmatrix}
\dps \IdM  - \frac{\delta^2}{2}Q_k  &
\dps \frac\delta2 \co  Q_k + \po \IdM  - \frac{\delta^2}{2} Q_k) \pf Q_{k+1}\cf\\
- \delta  & \dps \IdM  -  \frac{\delta^2}{2}Q_{k+1}
\end{pmatrix}
\,.\]
In particular,
\begin{equation}
\label{eq:recurr}
\left| \na (\Phi_{\delta}^{-1})(z_k) - 
\begin{pmatrix}
\dps \IdM  - \frac{\delta^2}{2}Q_k  &
\dps \frac\delta2 \co  Q_k +  Q_{k+1}\cf\\
- \delta   & \dps \IdM  -  \frac{\delta^2}{2}Q_{k+1}
\end{pmatrix}\right| \leqslant \frac{\delta^3}{4}\,. 
\end{equation}
For all $k\in \cco 1,K\ccf$, let us determine by induction a constant~$C_k\geqslant 0$ and matrices~$\alpha_k,\beta_k,\zeta_k \in \mathcal M_1$ such that
\begin{align}\label{eq:Jinduction}
\left|\na (\Phi_\delta^{-k})(z_0) - 
\begin{pmatrix}
\dps \IdM -  \frac{(\delta k)^2}2  \alpha_k & \delta k \beta_k\\ -\delta k & \dps \IdM -\frac{(\delta k)^2}2  \zeta_k
\end{pmatrix}   \right| \leqslant C_k\,.
\end{align}
 For $k=1$, this is given by \eqref{eq:recurr} (applied at $k=0$) with $C_1=\delta^3 /4$. Suppose that the result is true for some $k\geqslant 1$. In particular, using that $\delta k \leqslant 1/10$, 
 \begin{align}
 |\na (\Phi^{-k}_\delta)(z_0) | & \leqslant
 \left|
\begin{pmatrix}
\dps \IdM -  \frac{(\delta k)^2}2  \alpha_k & 0\\ 0 & \dps \IdM -\frac{(\delta k)^2}2  \zeta_k
\end{pmatrix}   \right| +  \left|
\begin{pmatrix}
0 & \delta k\beta_k\\ -\delta k   & 0
\end{pmatrix}   \right|+C_k \nonumber\\
& \leqslant 1 +  \frac{21}{20}  \delta k    +C_k \,.\label{eq:Jdemo}
 \end{align}
We next use~$\na \Phi^{-k-1}_\delta (z_0) = \na \Phi_\delta^{-1}(z_k) \na \Phi^{-k}_\delta(z_0)$. At dominant order, the product of the terms on the right hand side of the previous equality is
\begin{align}
& \begin{pmatrix}
\dps \IdM  - \frac{\delta^2}{2}Q_k  &
\dps \frac{\delta}{2} \co  Q_k +  Q_{k+1}\cf\\
- \delta  & \dps \IdM  -  \frac{\delta^2}{2}Q_{k+1}
\end{pmatrix}
\begin{pmatrix}
\dps \IdM -  \frac{(\delta k)^2}2  \alpha_k & \delta k \beta_k\\ -\delta k  & \dps \IdM -\frac{(\delta k)^2}2  \zeta_k
\end{pmatrix}\\
& \qquad\qquad\qquad\qquad = \begin{pmatrix}
\dps \IdM - \frac{\delta^2 (k+1)^2}2 \alpha_{k+1} & \delta (k+1) \beta_{k+1}\\ -\delta (k+1)  & \dps \IdM  -\frac{\delta^2 (k+1)^2}2 \zeta_{k+1}
\end{pmatrix} + \mathfrak R \, ,
\end{align}
with
\[
\alpha_{k+1}= \frac1{(k+1)^2}\po Q_k + k^2 \alpha_k + k(Q_k+Q_{k+1})\pf \, , 
\qquad 
\beta_{k+1}= \frac1{2(k+1)}\po  Q_k+Q_{k+1} + 2k\beta_k \pf \, , 
\]
\[
\zeta_{k+1}= \frac1{(k+1)^2}\po 2k\beta_k + Q_{k+1} + k^2 \zeta_k \pf \, ,
\]
which are all three in $\mathcal M_1$, and a remainder term
\[
\mathfrak R=\begin{pmatrix}
\dps \frac{\delta^4 k^2}4 Q_{k}  \alpha_{k}
& \dps -\frac{\delta^3 k}4\po 2 Q_k\beta_{k}+k(Q_k+Q_{k+1}) \zeta_k\pf\\ 
\dps \frac{\delta^3k}2 \po  Q_{k+1}\newOld{+} k \alpha_k\pf    
& \dps \frac{\delta^4 k^2}4 Q_{k+1} \zeta_k
\end{pmatrix}\,. 
\]
Using that $\delta \leqslant 1/10$,
\begin{align}
|\mathfrak R| & = \frac{\delta^4 k^2}4  \left|\begin{pmatrix}
Q_{k}  \alpha_{k}
&  0 \\ 
0
&  Q_{k+1} \zeta_k
\end{pmatrix}
\right|
+ \frac{\delta^3 k}2
 \left|\begin{pmatrix}
0
& - Q_k\beta_{k}-k(Q_k+Q_{k+1})/2 \zeta_k\\ 
  Q_{k+1} + k \alpha_k     
& 0
\end{pmatrix}
\right| \nonumber \\
& \leqslant \frac{\delta^4 k^2}4  + \frac{\delta^3 k}2(k+1) \leq \frac{21}{40} \delta^3 k(k+1)\,. \label{eq:borneR}
\end{align}
Then,
\begin{align}
(\star) & :=\left|\na\Phi_\delta^{-k-1}(z_0) - \begin{pmatrix}
\IdM  - \delta^2 (k+1)^2 \alpha_{k+1}/2 & \delta (k+1) \beta_{k+1}\\ -\delta (k+1) & \IdM  -\delta^2 (k+1)^2 \zeta_{k+1}/2
\end{pmatrix} \right|\\
& \leqslant |\mathfrak R| +  \left|\co \na (\Phi_{\delta}^{-1})(z_k) - 
\begin{pmatrix}
\IdM  - \delta^2 Q_k/2  &
\delta \co  Q_k +  Q_{k+1}\cf/2 \\
- \delta   & \IdM  -  \delta^2 Q_{k+1}/2
\end{pmatrix}\cf   \na\Phi_\delta^{-k}(z_0)\right|  \\
& \ \ + \left|  \begin{pmatrix}
\IdM  - \delta^2 Q_k/2  &
\delta \co  Q_k +  Q_{k+1}\cf/2\\
- \delta   & \IdM - \delta^2 Q_{k+1}/2
\end{pmatrix}
\co \na\Phi_\delta^{-k}(z_0)- \begin{pmatrix}
\IdM - (\delta k)^2  \alpha_k/2 & \delta k \beta_k \\ -\delta k   & \IdM -(\delta k)^2 \zeta_k/2
\end{pmatrix}  \cf\right| \\
& \leqslant \frac{21}{40}  \delta^3 k(k+1)  + \frac{\delta^3}{4} |\na\Phi_\delta^{-k}(z)| + \po 1+ \frac{21}{20} \delta\pf  C_k\, ,
\end{align}
where we used \eqref{eq:borneR} to bound $|\mathfrak  R|$, \eqref{eq:recurr} for the second term and the induction hypothesis \eqref{eq:Jinduction} for  the third one (with the same computation as in \eqref{eq:Jdemo} for the factor~$1+\frac{21}{20}\delta$). In view of~\eqref{eq:Jdemo}, we finally obtain
\begin{align}
(\star) \leqslant C_{k+1} & := \frac{21}{40}  \delta^3 k(k+1)  + \frac{\delta^3}{4} \po 1 +  \frac{21}{20}  \delta k    +C_k\pf  + \po 1+ \frac{21}{20} \delta\pf C_k \\
&\leqslant \frac{21}{40}  \delta^3 (k+1)^2  + \po 1+ \frac{11}{10} \delta\pf C_k  \,.
\end{align}
which allows to obtain the induction hypothesis~\eqref{eq:Jinduction} at the next step. 
From this, since~$C_1 \leq 21 \delta^3/40$,
\[
C_K \leqslant \frac{21}{40}  \delta^3   \sum_{k=0}^{K-1} (k+1)^2 \po 1+ \frac{11}{10} \delta\pf^{K-1-k} \leqslant \frac{21}{40\times 3} (\delta K)^3     \rme^{11 \delta K/10} \leqslant \frac13 (\delta K)^3\,.
\]
which gives the desired estimate.
\end{proof}


In particular, thanks to the estimate on~$|\Psi|$ of Lemma~\ref{lem:VerletMatrice} and to Lemma~\ref{L:fisher_verlet}  we can bound the Fisher term appearing in Lemma~\ref{lem:entropy_verlet} to get the following somewhat more explicit estimate.

\begin{corollary}[Improved entropy inequality on Verlet step] \label{C:improved_ent_verlet}
For any $\varepsilon >0$,
\begin{multline}
\mathcal H \po (\rmV_\delta^K)^* h\pf -\mathcal H (h) \leqslant  4 \varepsilon K \delta  \mathcal I  \po  h \pf \\
+ \sum_{j=0}^{K-1} \co 3\delta \varepsilon  \mathrm{Er}_1^{\delta,j} \po    h \pf  + \varepsilon \int_0^\delta \mathrm{Er}_1^{s,1} \po  (\rmV_\delta^j)^* h \pf  \dd s + \frac1{2\varepsilon} \mathrm{Er}_2\po (\rmV_{\delta}^j)^* h\pf  \cf    \,.
\end{multline}
\end{corollary}

\begin{proof}
For $j\in\cco 0,K-1\ccf$ and $s\in[0,\delta]$, writing $\Psi_s = \na \Phi_s^{-1} \circ \Phi_s$ and $\Psi_j =\na \Phi_\delta^{-j}\circ \Phi_\delta^j$, applying twice Lemma~\ref{L:fisher_verlet}  (with $\tilde{\bfA}=I_{2d}$ and $\varepsilon=1$) and using the bound \eqref{eq:normPsi} (with $s\leqslant K\delta \leqslant 1/10$) on $|\Psi_s|$ and $|\Psi_j|$, which implies that these quantities are smaller than~$3/2$ in sup norm,
\begin{align}
 \mathcal I \po  (\rmV_\delta^j \rmV_s^1)^* h \pf   & \leqslant 2 \mathcal I_{\Psi_s } \po  (\rmV_\delta^j)^* h \pf +2 \mathrm{Er}_1^{s,1} \po  (\rmV_\delta^j )^* h \pf \\
  & \leqslant 2 \|\Psi_s\|_\infty^2\mathcal I \po  (\rmV_\delta^j )^* h \pf +2 \mathrm{Er}_1^{s,1} \po  (\rmV_\delta^j)^* h \pf \\
   & \leqslant 3 \mathcal I \po  (\rmV_\delta^j )^* h \pf +2 \mathrm{Er}_1^{s,1} \po  (\rmV_\delta^j)^* h \pf \\
   & \leqslant 6 \mathcal I_{\Psi_j}(h) +6 \mathrm{Er}_1^{\delta,j}(h)+2 \mathrm{Er}_1^{s,1} \po  (\rmV_\delta^j)^* h \pf \\
    & \leqslant 8 \mathcal I(h) +6 \mathrm{Er}_1^{\delta,j}(h)+2 \mathrm{Er}_1^{s,1} \po  (\rmV_\delta^j)^* h \pf \,.
\end{align}
Plugging this bound in the result of Lemma~\ref{lem:entropy_verlet} gives the claimed estimate.
\end{proof}

\subsection{Numerical error}\label{sec:numerique}

The goal of this section is to bound the numerical errors appearing in Lemmas~\ref{L:fisher_verlet} and \ref{lem:entropy_verlet}.
At various places, we use the following inequality (obtained by convexity): for~$\alpha \geq 1$ and nonnegative sequences~$(a_i)_{1 \leq i \leq I}$ and~$(w_i)_{1 \leq i \leq I}$, it holds 
\begin{equation}
    \label{eq:convexity_ineq_sums}
\left( \sum_{i=1}^I w_i a_i\right)^\alpha \leq \left( \sum_{i=1}^I w_i\right)^{\alpha-1} \sum_{i=1}^I w_i a_i^\alpha.
\end{equation}

\subsubsection{Error in the Fisher term}

\begin{lemma}\label{Lemme:error1}
Under \Cref{hyp:target} with $L=1$ and \Cref{assu:derivativesUbounded}, it holds, for any~$z\in\R^d$,
\begin{multline}
  |\na \square H(z)|^ 2 \leqslant    2 \delta^6 K \sum_{k=0}^{K-1}\Big[  N_4^6(v_k) +     N_4^6 (\na U(x_k)) + N_3^4(v_k)+      N_3^4(\na U(x_k)) \\
  + |v_k|^2 +  |\na U(x_k)|^2 \Big]\,, 
\end{multline}
where $(x_k,v_k) = \Phi_\delta^{- k}(z)$ for $k\in\cco 0,K\ccf$.
\end{lemma}

\begin{proof}
In view of the decomposition
\[
\square H = H - H\circ \Phi_\delta^{-K}= \sum_{k=1}^K \po H \circ \Phi_\delta^{-k+1} -H \circ \Phi_\delta^{-k}\pf\,, 
\]
which leads to the upper bound
\[
|\na   \square H|^2 \leqslant K \sum_{k=1}^K \left|\na  \co  H \circ  \Phi_\delta^{-k+1}-H \circ  \Phi_\delta^{-k}\cf\right|^2 \,. 
\]
it suffices to establish the estimate for~$K=1$. Recall $H(x,v)=U(x)+|v|^2/2$ and recall from~\eqref{eq:phi-1} that
\begin{align}
(x_1,v_1) & = \Phi_\delta^{-1}(x,v)\\
 &= \po x- \delta v - \frac{\delta^2}{2} \na U(x)\ , \ v + \frac{\delta}{2}\co \na U(x) + \na U \po x- \delta v - \frac{\delta^2}{2} \na U(x)\pf  \cf \pf  \,. 
\end{align}
Hence, starting with the gradient with respect to the velocities, 
\begin{align}
  & \na_v \po H  - H\circ \Phi_\delta^{-1}\pf(x,v) \\
  & \qquad = v - \co -\delta \na U(x_1) + \po I- \frac{\delta^2}{2}\na^2 U(x_1)\pf v_1 \cf \\
  & \qquad = \frac{\delta}2\po  \na U(x_1) - \na U(x) \pf + \frac{\delta^2}{2}\na^2 U(x_1)v + \frac{\delta^3}{4}\na^2 U(x_1) \po \na U(x)+ \na U(x_1)\pf \\
  & \qquad =  \frac{\delta}2\po  \na U(x_1) - \na U(x) +\na^2 U(x_1)(x-x_1)\pf     + \frac{\delta^3}{4}\na^2 U(x_1)   \na U(x_1) \,. 
\end{align}
  Notice that \eqref{eqdef:assuU3bounded} implies that
\begin{align}
    |\na U(x+y) - \na U(x) - \na^2 U(x) y| 
&= \left|\int_0^1 \po \na^2 U (x+sy) - \na^2 U (x) \pf y \, \dd s \right| \leqslant  \frac12 \rmN_3^2(y) \label{eqdef:U3bounded_bis}\,.
\end{align}
 Using this bound and the assumption~$\|\na^2 U\|_\infty\leqslant L = 1$, we have thus obtained that
\begin{align}
\left|\na_v \po H  - H\circ \Phi_\delta^{-1}\pf(x,v)\right| 
& \leqslant \frac{\delta}4\rmN_3^2(x-x_1)   + \frac{\delta^3}{4}|  \na U(x_1) |\nonumber\\
& \leqslant \frac{\delta^3}2 \rmN_3^2(v) +  \frac{\delta^5}8 \rmN_3^2(\na U(x))  + \frac{\delta^3}{4}\po |  \na U(x) | + |x_1-x|\pf \nonumber \\
& \leqslant \frac{\delta^3}2 \rmN_3^2(v) +  \frac{\delta^5}8 \rmN_3^2(\na U(x))  + \frac{\delta^4}{4}|v| + \po \frac{\delta^3}{4} + \frac{\delta^5}{8} \pf   |  \na U(x) | \,. 
\end{align}
With~$\delta\leqslant 1/10$ and a Cauchy--Schwarz inequality,
 \begin{align}
|\na_v \po H  - H\circ \Phi_\delta^{-1}\pf(x,v)|^2 
& \leqslant \delta^6 \po \frac{1}2 \rmN_3^2(v) +  \frac{1}{800} \rmN_3^2(\na U(x))  + \frac{1}{40}|v| +  \frac{201}{800}   |  \na U(x) | \pf^2\nonumber \\
& \leqslant \frac{\delta^6}3  \po  \rmN_3^4(v) +   \rmN_3^4(\na U(x)) + |v|^2 + |\na U(x) |^2 \pf\,. \label{eq:erreurnumlemme}
 \end{align}
 
Turning to the gradient with respect to the positions,
  \begin{align}
  \lefteqn{\na_x \po H  - H\circ \Phi_\delta^{-1}\pf(x,v)} \\
  &= \na U(x) - \co \po I - \frac{\delta^2}{2}\na^2 U(x)\pf \na U(x_1) +\frac\delta2 \po \na^2 U(x) + \po I - \frac{\delta^2}{2}\na^2 U(x)\pf \na^2 U(x_1) \pf v_1 \cf \\
  &= (\ast) + (\ast\ast) + (\ast\ast\ast)\,,
  \end{align}
  with, organizing terms depending on their orders in $\delta$, and using~$\delta v_1 = x-x_1+\delta \nabla U(x_1)/2$,
  \begin{align}
  (\ast) &= \na U(x) - \na U(x_1) - \frac12\po \na^2 U(x)+\na^2 U(x_1)\pf (x-x_1)\,, \\
  (\ast \ast) & =\frac{\delta^2}4\po \na^2 U(x)-\na^2 U(x_1)\pf \na U(x_1)\,, \\ 
  (\ast \ast \ast) & = \frac{\delta^3}{4}\na^2 U(x) \na^2 U(x_1) v_1   \,. 
  \end{align}
  Using the bound \eqref{eqdef:assuU4bounded} on the fourth derivative of $U$ and~\eqref{eq:convexity_ineq_sums}, 
  \begin{align}
  |(\ast)|  \leqslant  \newOld{N_4^3}(x-x_1) & \leqslant  \delta^3\po  \newOld{N_4}(v) +  \delta\newOld{N_4} (\na U(x))\pf^3 \\ & \leqslant  \delta^3\left(1+\delta\right)^2   \po \newOld{N_4^3}(v) + \delta \newOld{N_4^3} (\na U(x))\pf \,.     
  \end{align}
 From the bound \eqref{eqdef:assuU3bounded} on the third derivative of $U$  and that $\|\na^2 U\|_\infty \leqslant L = 1$
  \begin{align}
  |(\ast\ast)| & \leqslant \newOld{  \frac{\delta^2}{4}|(\na^2 U(x) - \na^2 U(x_1)) \na U(x) | + \frac{\delta^2}{2} |\na U(x_1)-\na U(x)| } \\
   & \leqslant \frac{\delta^2}4 \newOld{N_3}(x-x_1)\newOld{N_3}(\na U(x)) + \frac{\delta^2}{2}|x_1-x| \\
   & \leqslant \frac{\delta^3}8 \newOld{N_3^2}(v)+    \frac{\delta^3+\delta^4}8 \newOld{N_3^2}(\na U(x)) + \frac{\delta^3}{2}|v| + \frac{\delta^4}{4}|\na U(x)|\,.
  \end{align}
  Finally, 
  \begin{align}
  | (\ast\ast\ast)| \leqslant \frac{\delta^3}{4}|v_1|  & \newOld{\leqslant \frac{\delta^3}{4} |v| + \frac{\delta^4}{4}|\na U(x)| + \frac{\delta^4}{8}|x_1-x|}\\
  &\newOld{ \leqslant \frac{2\delta^3 + \delta^5}{8}   |v| + \frac{4\delta^4 + \delta^6}{16} |\na U(x)|  } \,. 
    \end{align}
As a conclusion, using that $\delta\leqslant 1/10$ and Cauchy--Schwarz inequality,
\begin{align}
   \lefteqn{|\na_x \po H  - H\circ \Phi_\delta^{-1}\pf(x,v)|^2} \\
   & \qquad \leqslant \delta^6 \Big[   (1+\delta)^2   \po \newOld{N_4^3}(v) + \delta \newOld{N_4^3} (\na U(x))\pf + \frac{1}8 \newOld{N_3^2}(v) \\
   & \qquad \qquad \left. +    \frac{1+\delta}{8} \newOld{N_3^2}(\na U(x)) + \left(\frac34+\frac{\delta^2}{8}\right)|v| + \delta\left(\frac12+\frac{\delta^2}{16}\right)|\na U(x)|\right]^2 \\ 
     & \qquad \leqslant \frac{5}{3} \delta^6 \Big[       \newOld{N_4^6}(v) +  \newOld{N_4^6} (\na U(x)) + \newOld{N_3^4}(v) +     \newOld{N_3^4}(\na U(x)) + |v|^2 + |\na U(x)|^2\Big] \,,
\end{align}
which, combined with \eqref{eq:erreurnumlemme}, gives the desired conclusion.
\end{proof}

Integrating the bound of \Cref{Lemme:error1} with respect to~$h\mu$ immediately yields the following (recall the definitions of the error in Lemma~\ref{L:fisher_verlet}  and of the function $\mathbf{M}$ in \eqref{eq:defVmoments}).

\begin{corollary} \label{C:error1}
Under \Cref{hyp:target} with $L=1$ and  \Cref{assu:derivativesUbounded}, for any smooth positive relative density $h$,
    \[ \mathrm{Er}_1^{\delta,K}(h) \leqslant \newOld{2} \delta^6 K \sum_{k=0}^{\newOld{K-1}} \int_{\R^{2d}} \po \mathbf{M} \circ \Phi_\delta^k \pf     h   \dd \mu  \,.\]
       In particular, assuming moreover \Cref{assu:moments}, and writing $h_n = (\rmPo^n)^*h_0$ for $n\in\N$,
       \begin{align}
         \mathrm{Er}_1^{\delta,K}(\rmD_\eta  h_n)& \leqslant \newOld{2} \delta^6 K^2 \po  C_1 \rme^{-\rho n \tint} \nu_0(\Wlyap)  + C_2\pf \eqsp, \\
         \mathrm{Er}_1^{\delta,1 }\po(\rmV_{\delta}^j)^*\newOld{\rmD_\eta}  h_n\pf & \leqslant \newOld{2} \delta^6  \po  C_1 \rme^{-\rho n \tint} \nu_0(\Wlyap)  + C_2\pf \eqsp,
       \end{align}
       \newOld{for all $j\in\cco 0,K\ccf $.}
    \end{corollary}
 
 \begin{remark}\label{rem:order1}
If in \Cref{assu:derivativesUbounded} we only assume  \eqref{eqdef:assuU3bounded} but not \eqref{eqdef:assuU4bounded}, from \eqref{eqdef:assuU3bounded}, using that $\na U(x+y) - \na U(x) = \int_0^1 \na^2 U(x+sy) y\, \dd s $ (and distinguishing the integration on~$[0,1/2]$ and on~$[1/2,1]$), we deduce that for all $x,y\in\R^d$,
\begin{align}
\abs{\na U(x+y) - \na U(x) - \frac12\po \na^2 U(x)+\na^2 U(x+y)\pf y}  &\leqslant   \frac14 \rmN_3^2(y)
\,. \label{eqdef:assuU4bounded_bis}
\end{align}
We can thus use this inequality instead of~\eqref{eqdef:assuU4bounded} in Lemma~\ref{Lemme:error1}  (which is the only place where~\eqref{eqdef:assuU4bounded} intervenes) and replace $\rmN_4^6$ by $\rmN_3^4$ in the definition \eqref{eq:defVmoments} of   $ \mathbf{M}$. Following then the rest of the proof, we see that Theorem~\ref{thm:main_KL} still holds in this case, but with $\delta^4$ replaced by $\delta^2$ as the error terms in \Cref{C:error1} are of order~$\delta^4$ instead of~$\delta^6$.
\end{remark}
  
\subsubsection{Error on the entropy term}

It suffices to consider the case $K=1$, as the general case is deduced from this one by summing up estimates. Recall that~$F_H$ is defined in~\eqref{eq:def_F_H}.

\begin{lemma}
Under \Cref{hyp:target} with $L=1$ and  \Cref{assu:derivativesUbounded}, for all $z\in\R^d$ and $s\in[0,\delta]$, 
\[|\partial_s \Phi_s(z) - F_H \po \Phi_s(z)\pf |^2 \leqslant s^4  \po \frac{1}{3}|v|^2 + \frac{1}{7}|\na U(x)|^2 + \frac{1}7  \newOld{N_3^4}(v) + \frac{1}{140}\newOld{N_3^4}(\na U(x)) \pf\,. \]
\end{lemma}

\begin{proof}
Decomposing $z_s= \Phi_s(x,v)$ as 
\[
z_s = \begin{pmatrix}
x_s\\ v_s
\end{pmatrix}
=
\begin{pmatrix}
\dps x+sv - \frac{s^2}{2} \na U(x) \\
\dps v - \frac{s}{2}\co \na U(x) + \na U \po x+sv - \frac{s^2}{2} \na U(x)\pf  \cf 
\end{pmatrix}\,,\]
we compute
\begin{align}
(\partial_s \Phi_s)(z) - F_H(z_s) & = \begin{pmatrix}
v - s \na U(x)\\
\dps -\frac12   \po \na U(x) + \na U(x_s)\pf -\frac s2  \na^2 U(x_s) \frac{dx_s}{ds}
\end{pmatrix} - \begin{pmatrix}
v_s \\ -\na U(x_s)
\end{pmatrix}\\
& = \frac12\begin{pmatrix}
 s \po \na U(x_s)-\na U(x)\pf \\
   \na U(x_s)-\na U(x) -  s  \na^2 U(x_s) \po v - s \na U(x)\pf
\end{pmatrix} \,.
\end{align}
As a consequence, using \eqref{eqdef:U3bounded_bis} and~$\|\na^2U\|_\infty\leqslant 1$,
\begin{align}
|(\partial_s \Phi_s)(z) - F_H(z_s)|^2 
& \leqslant \frac{s^2}{4}|x_s-x|^2 + \frac12 |\na U(x_s)-\na U(x) - \na^2 U(x_s) (x_s-x)|^2 \\
& \ \ + \frac{1}2 |x_s-x - s(v-s\na U(x))|^2 \\
& \leqslant \frac{s^2}{4}|x_s-x|^2 + \frac{\newOld{1}}8 \newOld{N_3^4}(x_s-x) + \frac{\newOld{s^4}}{8} |\na U(x)|^2 \,.
\end{align}
Using that $s\leqslant \delta\leqslant 1/10$, we bound 
\[
\newOld{N_3^4}(x_s-x) \leqslant \po s \newOld{N_3}(v) + \frac{s}{20} \newOld{N_3}(\na U(x))\pf^4 \leqslant \frac{21^3}{20^3 } s^4  \po  \newOld{N_3^4}(v) + \frac{1}{20} \newOld{N_3^4}(\na U(x))\pf\,,
\]
and similarly $|x_s-x|^2 \leqslant 21/20  s^2\po |v|^2 + |\na U(x)|^2/20\pf$, which allows to conclude.
\end{proof}

Again, integrating the previous bound as
\begin{align}
\mathrm{Er}_2(\delta,h) &=  \int_0^\delta\po \int_{\R^{2d}}|\partial_s \Phi_s - F_H \circ \Phi_s |^2  \,h\,d\mu\pf\,\dd s \leqslant \frac13 \int_0^\delta\po \int_{\R^{2d}}s^4 \mathbf{M}   \,h\,d\mu\pf\,\dd s \,.
\end{align}
immediately yields the following result (recall the definitions of the Error 2 in Lemma~\ref{lem:entropy_verlet} and of the function $\mathbf{M}$ in \eqref{eq:defVmoments}).

\begin{corollary} \label{C:error2}
 Under \Cref{hyp:target} with $L=1$ and  \Cref{assu:derivativesUbounded}, for any smooth positive relative density $h$,
        \[\mathrm{Er}_2(h) \leqslant \frac{\delta^5}{15}  \int_{\R^{2d}}  \mathbf{M}   \,h\,d\mu\,.\]
        In particular, assuming moreover \Cref{assu:moments}, and writing $h_n = (\rmPo^n)^*$ for $n\in\N$ then, for any $j\in\cco 0,K-1\ccf$,
    \[\mathrm{Er}_2((\rmV_\delta^j)^* \rmD_\eta  h_n) \leqslant \frac{\delta^5}{15} \po  C_1 (1-\rho K\delta )^n \nu_0(\Wlyap)  + C_2\pf\,. \]
\end{corollary}

\subsection*{Acknowledgments}

The works of P.M. and G.S. are supported by the European Research Council (ERC) under the European Union’s Horizon 2020 research and innovation program (project EMC2, grant agreement No 810367); and by the Agence Nationale de la Recherche, under grant SWIDIMS (ANR-20-CE40-0022) for P.M., and ANR-19-CE40-0010-01 (QuAMProcs) and ANR-21-CE40-0006 (SINEQ) for G.S. A.O.D. would like
to thank the Isaac Newton Institute for Mathematical Sciences for support and hospitality during the programme "The mathematical and statistical foundation of future data-driven engineering" where work on this paper was undertaken. This work was supported by EPSRC grant number EP/R014604/1.

\bibliographystyle{plain}
\bibliography{bibliography}

\begin{appendix}

\section{Proof of Proposition~\ref{prop:regularity}}\label{sec:regularity}

Note first that, for~$j \neq i$ and defining~$w(x_i,x_j) = W^{(q)}(x_i-x_j)$, it holds
\[
\nabla^2_{x_i,x_j} w(x_i,x_j) = -\frac{1}{\nq} \nabla^2 W^{(q)}(x_i-x_j)\,,
\]
while
\[
\nabla^2_{x_i} w(x_i,x_j) = \frac{1}{\nq} \sum_{j=1}^{\nq} \nabla^2 W^{(q)}(x_i-x_j)\,.
\]
Therefore, 
\[
\begin{aligned}
& \sum_{i=1}^{d_0} \left| \sum_{j=1}^{d_0}  \po  \na^2_{x_i,x_j} w(x_i+y_i,x_j+y_j) - \na^2_{x_i,x_j} w(x_i,x_j)\pf z_j \right|^2 \\
& \qquad = \frac{1}{\nq^2}\sum_{i=1}^{d_0} \left| \sum_{j \neq i} \po  \na^2  W^{(q)}(x_i+y_i-x_j-y_j) - \na^2  W^{(q)}(x_i-x_j)\pf (z_i-z_j)\right|^2 \\
& \qquad = \frac{1}{\nq^2}\sum_{i=1}^{d_0} \left| \sum_{j=1}^{d_0}  \po  \na^2  W^{(q)}(x_i+y_i-x_j-y_j) - \na^2  W^{(q)}(x_i-x_j)\pf (z_i-z_j)\right|^2.
\end{aligned}
\]
We can now compute
\begin{align}
& \left| \po \na^2 U(x+y) - \na^2 U(x) \pf z\right |^2 \\
& \quad = \sum_{i=1}^{d_0} \left| \sum_{j=1}^{d_0}\po  \na^2 _{x_i,x_j} U(x+y) - \na^2_{x_i,x_j} U(x)\pf z_j\right|^2 \\
& \quad \leqslant (1+\epsilon) \sum_{i=1}^{\nq} \abs{\po \na^2 U^{(q)}(x_i+y_i) - \na^2 U^{(q)}(x_i) \pf z_i}^2 \\
& \quad \ \ + \frac{\epsilon^2(1+\epsilon^{-1})}{\nq^2 }  \sum_{i=1}^{d_0} \left| \sum_{j=1}^{d_0}  \po  \na^2  W^{(q)}(x_i+y_i-x_j-y_j) - \na^2  W^{(q)}(x_i-x_j)\pf (z_i-z_j)\right|^2 \\
& \quad \leqslant (1+\epsilon) \left\|\rmD^3 U^{(q)}\right\|_\infty^2 \sum_{i=1}^{\nq} |y_i|^2|z_i|^2 + \frac{\epsilon(1 + \epsilon)}{\nq }  \left\|\rmD^3 W^{(q)}\right\|_\infty^2 \sum_{i=1}^{d_0}  \sum_{j=1}^{d_0} |y_i-y_j|^2 |z_i-z_j|^2  \\
& \quad \leqslant (1+\epsilon) \po \left\|\rmD^3 U^{(q)}\right\|_\infty ^2 + 16 \epsilon \left\|\rmD^3 W^{(q)}\right\|_\infty^2 \pf \sqrt{ \sum_{i=1}^{\nq} |y_i|^4 \sum_{j=1}^{\nq}|z_j|^{\newOld{4}}} \,,
\end{align}
where we used 
\begin{align}
\frac{1}{\nq}\sum_{i=1}^{\nq}  \sum_{j=1}^{\nq} |y_i-y_j|^2 |z_i-z_j|^2 & \leq \frac{4}{\nq}\sum_{i=1}^{\nq}  \sum_{j=1}^{\nq} \left( |y_i|^2 + |y_j|^2\right) \left( |z_i|^2 + |z_j|^2 \right) \\
& \leq 8 \sum_{i=1}^{\nq} |y_i|^2 |z_i|^2 + \frac{8}{\nq}\sum_{i=1}^{\nq} |y_i|^2 \sum_{j=1}^{\nq} |z_j|^2 \\
& \leq 16 \sqrt{ \sum_{i=1}^{\nq} |y_i|^4 \sum_{j=1}^{\nq}|z_j|^{\newOld{4}}} \,,
\end{align}
to conclude the proof of \eqref{eqdef:assuU3bounded}. Following similar computations,
\begin{align}
  & \abs{\na U(x+y) - \na U(x) - \frac12\po \na^2 U(x)+\na^2 U(x+y)\pf y}^2\\
  & \quad = \left| \po \int_0^1 \na^2 U(x+sy) \, \dd s  - \frac12\po \na^2 U(x)+\na^2 U(x+y)\pf \pf y\right|^2  \\
   & \quad = \sum_{i=1}^{\nq} \left| \sum_{j=1}^{\nq}  \po \int_0^1 \na_{x_i,x_j}^2 U(x+sy) \, \dd s  - \frac12\po \na_{x_i,x_j}^2 U(x)+\na^2 _{x_i,x_j} U(x+y)\pf \pf y_j\right|^2  \\
   & \quad \leq (1+\epsilon) \sum_{i=1}^{\nq} \left|   \po \int_0^1 \na^2 U^{(q)}(x_i+sy_i) \, \dd s  - \frac12\po \na^2 U^{(q)}(x_i)+\na^2 U^{(q)}(x_i+y_i)\pf \pf y_i\right|^2 \\
   & \quad \ \  + \frac{\epsilon +\epsilon^2}{\nq }  \sum_{i,j=1}^{d_0} \left| \left( \int_0^1 \na^2 W^{(q)}(x_i-x_j+s(y_i-y_j)) \, \dd s \right. \right. \\
   & \quad \qquad \left. \left. \phantom{\int_0^1} - \frac12\po \na^2 W^{(q)}(x_i-x_j)+\na^2 W^{(q)}(x_i+y_i-x_j-y_j)\pf \right) (y_i-y_j)\right|^2 \,.
\end{align}
Using that, for all $f\in\mathrm{C}^2([0,1],\R)$, integrating twice by parts,
\[
\int_0^1 f(s) \dd s - \frac{f(0)+f(1)}2 
 = \gabriel{\frac12} \int_0^1 s(1-s)  f''(s) \, \dd s\,,
\]
fixing $i \in\cco 1,\nq\ccf$, $x_i,y_i\in\R^{q}$ and denoting by~$\bfe_{\ell}$  the $\ell$-th  canonical vector of $\rset^{q}$,
\begin{align}
   & \left| \bfe_\ell \cdot  \po \int_0^1 \na^2 U^{(q)}(x_i+sy_i) \, \dd s  - \frac12\po \na^2 U^{(q)}(x_i)+\na^2 U^{(q)}(x_i+y_i)\pf \pf y_i\right| \\
   & \qquad = \frac12 \left| \int_0^1 s(1-s) \rmD^4 U^{(q)}(x_i+s y_i) \{\bfe_\ell,y_i,y_i,y_i\} \, \dd s\right|  \\
   &\qquad \leqslant \frac1{12} \| \rmD^4 U^{(q)} \|_\infty |y_i|^3  \,.
\end{align}
The terms involving $W^{(q)}$ are treated in a similar manner and we get
\begin{align}
  & \abs{\na U(x+y) - \na U(x) - \frac12\po \na^2 U(x)+\na^2 U(x+y)\pf y}^2 \\
   & \qquad \leqslant (1+\epsilon) \sum_{i=1}^{\nq}  \frac{q}{144} \left\| \rmD^4 U^{(q)} \right\|_\infty^2 |y_i|^6  + \frac{\epsilon +\epsilon^2}{\nq }  \sum_{i,j=1}^{d_0}  \frac{q}{144} \left\| \rmD^4 W^{(q)} \right\|_\infty^2 |y_i-y_j|^6 \\ 
   & \qquad \leqslant \po \frac{1+\epsilon}{144} \left\| \rmD^4 U^{(q)} \right\|_\infty^2 + \frac{4(\epsilon +\epsilon^2) }{9 } \left\|\rmD^4 W^{(q)} \right\|_\infty^2 \pf q \sum_{i=1}^{\nq}   |y_i|^6 \,,
\end{align}
 which concludes the proof of \eqref{eqdef:assuU4bounded}.
 
 Finally, to bound $\mathbf{M}$, for $\ell \in \{2,4,6\}$, using that $\na U^{(q)}(0)=0=\na W^{(q)}(0)$, we write
 \begin{align}
    |\na_{x_i} U(x) |^\ell & = \left| \na U^{(q)}(x_i) + \frac{2\epsilon}{\nq}\sum_{j=1}^{\nq} \na W^{(q)}(x_i-x_j) \right|^\ell  \\
    & \qquad \leqslant 2^{\ell-1} |\na U^{(q)}(x_i)|^\ell + \frac{2^{2\ell-1}\epsilon^{\ell}}{\nq } \sum_{j=1}^{\nq} \left| \na W^{(q)}(x_i-x_j) \right|^\ell \\
    & \qquad \leqslant 2^{\ell-1} \left\|\na^2 U^{(q)}\right\|_\infty^{\ell}  |x_i|^\ell + \frac{2^{2\ell-1}\epsilon^{\ell}}{\nq } \left\|\na^2 W^{(q)}\right\|_\infty^\ell \sum_{j=1}^{\nq} \left| x_i-x_j \right|^\ell \,.
 \end{align}
 Using  $|x_i-x_j|^\ell \leqslant 2^{\ell-1}(|x_i|^\ell + |x_j|^{\ell})$ in the last term, we end up with
 \[
 \sum_{i=1}^{\nq} |\na_{x_i} U(x) |^\ell  \leqslant \po 2^{\ell-1}\left\|\na^2 U^{(q)}\right\|_\infty^{\ell}  + 2^{3\ell}\epsilon^{\ell} \left\|\na^2 W^{(q)}\right\|_\infty^\ell \pf \sum_{i=1}^{\nq}  |x_i|^\ell\,, 
 \] 
which leads to the claimed result.

\section{Uniform-in-time moment bounds}\label{sec:Lyapunov}

In all this section, which is devoted to the proof of Theorem~\ref{Thm:drift}, assumption~\Cref{ass:for_drift} is enforced. We denote by~$L_1 = \max(\|\na^2 U^{(q)}\|_\infty,\|\na^2 W^{(q)}\|_\infty)$. In fact, using a rescaling as in Section~\ref{subsubsec:scaling}, we assume without loss of generality that $L_1 \leqslant 1$. 
The structure of the proof is the following. After providing some useful preliminary estimates, we describe the class of Lyapunov functions we consider, state the main result in Theorem~\ref{theo:drift}, and discuss that it implies Theorem~\ref{Thm:drift}. The remainder of the section is then devoted to the proof of Theorem~\ref{theo:drift}. We start by making precise in Lemma~\ref{lem:drift_v1} how the Lyapunov function evolves over one Verlet step; and then over several steps in Lemma~\ref{lem:drift_k_step}, relying on Lemma~\ref{lem:rec_for_drift} which quantifies how much prefactors change when performing several steps. We can then consider the addition of the fluctuation-dissipation step, by first incorporating the momentum update in Lemma~\ref{lem:drift_for_fixed_Gaussian}, then averaging over the momentum updates in Lemma~\ref{lem:drift_one_step}, relying on some moment bounds provided by Lemma~\ref{lem:moment_gaussian}. The proof of Theorem~\ref{theo:drift} can then be concluded.

Let us emphasize that our proofs are written both for kinetic Langevin dynamics, for which~$\eta>0$ (more precisely $1-\eta \geq 0$ is of order~$\delta$), and for HMC, for which~$\eta = 0$. This second case requires more precise estimates.

\paragraph{Preliminary estimates.}
We state a few consequences of \Cref{ass:for_drift} which will be useful in the remainder of this section. We additionally assume without loss of generality that~$\Uq(0) = 0$. First, \eqref{eq:ass_drift} and the inequality 
\[
\ps{x_1 }{ \nabla U^{(q)}(x_1)} \leq \frac{1}{2\mtt} \abs{\nabla U^{(q)}(x_1)}^2 + \frac{\mtt}{2} \abs{x_1}^{2}, 
\]
imply that
\begin{equation}
  \label{eq:tail_nabla_U}
  \forall x_1 \in\rset^q, \qquad \norm{\nabla \Uq(x_1)}^2 \geq 2\mtt \left(\frac{\mtt}{2}\norm{x_1}^2 - \Mtt\right).
\end{equation}
In addition, since $\na U^{(q)}$ is $L_1$-Lipschitz and $\na U^{(q)}(0)=0$, the inequality $\ps{x_1 }{ \nabla U^{(q)}(x_1)} \leq L_1 \abs{x_1}^2$,
and the drift condition~\eqref{eq:ass_drift} together imply, by taking the limit~$\abs{x_1} \to +\infty$, that 
\begin{equation}
\label{eq:bound_m}
    \mtt \leq L_1 \leqslant 1 \eqsp .
\end{equation}
Finally, since $\na U^{(q)}(0)=\na W^{(q)}(0)=0$, we obtain that, for any $x_1\in\rset^q$
\begin{equation}
    \label{eq:pot_drift_sub_linear}
    \norm{\nabla \Uq(x_1)} \leq \norm{x_1} \eqsp, \qquad \norm{\nabla \Wq(x_1)} \leq \norm{x_1}\eqsp, \qquad U^{(q)}(x_1) \leqslant \frac{|x_1|^2}{2} \eqsp.
\end{equation}

\paragraph{Construction of Lyapunov functions.}
Our analysis is based on the following Lyapunov functions. Under \Cref{ass:for_drift}, for $\bfomega = (K,\delta,\eta) \in \nsets\times \rset_+^* \times \coint{0,1}$ and $\ell \in\nsets$, we define $\Wol : \rset^{2d} \to \rset$ as: for any $x =(x_1,\ldots,x_{\nq})$ and~$v =(v_1,\ldots,v_{\nq}) \in \rset^d$, 
\begin{equation}
  \label{eq:def_wlyap}
  \Wol(x,v) = \sum_{i=1}^{\nq} \po \Woi(x,v)\pf^\ell  \eqsp, \quad \Woi(x,v) = \frac{\gamma_0^{2}}{2}\norm{x_i}^2 +\norm{v_i}^2 + \eta \gamma_0\ps{x_i}{v_i} + \Uq(x_i) \eqsp,
\end{equation}
where (recall~\eqref{eq:def_tint_gammaint}) $  \gamma_0 =  \frac{1-\eta}{\delta K}$.
Note that since $\eta \in \coint{0,1}$, using Young's inequality $\ps{\mathbf{a}}{\mathbf{b}} \leq (\varepsilon/2)\norm{\mathbf{a}}^2 + \norm{\mathbf{b}}^2/(2\varepsilon)$ with $\varepsilon = 2/3$ for the left inequality and~$\varepsilon = 1$ for the right one, 
we get
\begin{equation}
  \label{eq:23}
  \frac16 \left(\gamma_0^{2} \norm{x}^2 + \norm{v}^2\right) + \sum_{i=1}^{\nq}\Uq(x_i) \leq \WoU(x,v) \leq \frac32 \left(\gamma_0^{2}\norm{x}^2 + \norm{v}^2\right) +  \sum_{i=1}^{\nq}\Uq(x_i) \eqsp.
\end{equation}
By \eqref{eq:23} and the fact $0 \leqslant U^{(q)}(x_1) \leqslant |x_1|^2/2$, the function $\Wlyap$ defined in \eqref{eq:Wlyap_simple} is equivalent to $\Wlyap_{\bfomega}^{(3)}$, more precisely there exists $A>0$ which depends only on $\gamma_0$ such that, for all $(x,v)\in\R^{2d}$,
\begin{equation}
\label{eq:Lyap_for_main_drift_thm}
A^{-1} \left(\nq+\Wlyap_{\bfomega}^{(3)}(x,v)\right) \leqslant \nq+\Wlyap(x,v)  \leqslant A \left(\nq+ \Wlyap_{\bfomega}^{(3)}(x,v)\right).
\end{equation}

\paragraph{Statement of the main result.}
The main result in this section is the following.
\begin{theorem}
  \label{theo:drift}
  Assume that~\Cref{ass:for_drift} holds with $L_1=1$ and let $\bdelta,\bT,\bareta >0$.
  Then, for any $\ell \in\nset$, $\epsilon >0$, $\bfomega = (K,\delta,\eta)$, $\delta >0$, $K\in\nsets$, $\eta \in \coint{0,\bareta}$ satisfying
      \begin{equation}
    \label{eq:20}
    \epsilon \leq \espilon_1 \wedge \epsilon_2 \eqsp,\qquad     \delta \leq \delta_1 \wedge \delta_2 \wedge 1/[7\times 8] \wedge \bdelta \eqsp, \qquad K \delta \leq {T}_1 \wedge T_2 \wedge T_3 \wedge \bT \eqsp,
  \end{equation}
  for any $x,v\in\rset^d$ and $k \in\{0,\ldots,K\}$,
  \begin{multline}
    \label{eq:drift_final}
     \rmP_{\bfomega}^n \rmD_{\eta}\rmV_{\delta}^k  \Wol(x,v) \leq C_{1,\ell}^{\ell} (1-\rho_{\bfomega}  K\delta)^{\ell  n}    \Wol(x,v) \\
     + \nq (C_{2,\ell} q^{\ell} + C_{3,\ell}) \sup_{s \in [0,\bT]} \frac{s}{1-(1-\rho_{\bfomega}s/8)^{\ell}}  \eqsp, 
     \end{multline}
   where $\epsilon_1,\epsilon_2,T_1,T_2,T_3,\delta_1,\delta_2$ are defined in  
\eqref{eq:10_applied_final}-\eqref{eq:condition_epsilon_e_0_final}-\eqref{eq:condition_epsilon_e_0_1_final}-\eqref{eq:condi_k_delta_T_3}-\eqref{eq:def_epsilon_2},
   \begin{equation}
  \label{eq:def_rho_omega_final}
  \rho_{\bfomega} = 2^{5}[\gamma_0\wedge \{ \eta \gamma_0 +\gamma_0^2K \delta/2\}\mtt(2^6 3)^{-1}] \eqsp,
\end{equation}
and  $C_{1,\ell},C_{2,\ell},C_{3,\ell}$ are given in \eqref{eq:def_C_ell_final}. These constants do not depend on $\nq,q$ and only depend on $\delta,K,\eta$ through $\gamma_0$ and $\rho_{\bfomega}$.
   \end{theorem}
   
   The remainder of Section~\ref{sec:Lyapunov} is devoted to the proof of this result. Before proceeding with this proof, let us notice that Theorem~\ref{theo:drift} implies Theorem~\ref{Thm:drift} in view of~\eqref{eq:Lyap_for_main_drift_thm}
and, as can be checked on their expressions, the constants $\epsilon_1,\epsilon_2,T_1,T_2,T_3,\delta_1,\delta_2,C_{1,\ell},C_{2,\ell},C_{3,\ell}$ appearing in Theorem~\ref{theo:drift} depend only on $\mtt,\Mtt,\overline{T},\overline{\delta},\gamma_0 $ and not on $q,\nq,\epsilon$.
   
\paragraph{Evolution over one Verlet step.} In order to prove \Cref{theo:drift}, we will work with a Lyapunov function slightly more general than $\Woi$.  Consider, for any $\bfo = (\a,\b,\c,\b_0,\e,\f) \in \rset_+^6$, $i \in \{1,\ldots,\nq\}$, and $x = (x_1,\ldots,x_q), v = (v_1,\ldots,v_q) \in \rsetd$, the function
\[
\begin{aligned}
\Voi(x,v) =  \a \norm[2]{x_i} & + \b \norm[2]{v_i} + \c \ps{x_i}{v_i} + 2\b_0 \Uq(x_i) \\
 & + \e \frac{\epsilon}{\nq}\sum_{j=1}^{\nq}\parentheseDeux{\norm{x_i-x_j}^2+\norm{v_i-v_j}^2} +\f\eqsp.
\end{aligned}
\]
Notice that the structure of $\Voi$ is similar to the one of $\Woi$ except that we have added the parameters $\bfo$ to be more flexible. That way, we will bound the value of $\Voi$ after some Verlet steps in terms of  $\Vlyapc[\tilde{\bfo}][i]$ with different parameters $\tilde \bfo$, so that the general forms of the function is preserved and we only have to keep track of the evolution of the coefficients $\bfo$.

Specifically, the following result makes precise how the Lyapunov function evolves over one  Verlet step, by bounding it with a similar Lyapunov function with modified coefficients. Note that the coefficient~$b_0$ in front of~$U^{(q)}$ in the expression of~$\Voi$ is fixed, so that this parameter of the Lyapunov function does not change over the iterations. 

\begin{lemma}
\label{lem:drift_v1}
Assume that \Cref{ass:for_drift} holds and that $\epsilon \leq \mtt/4$. 
Then, for any $x,v \in \rsetdd$, $\delta\in \ocintLigne{0,\mtt/40}$, $i\in\iint{1}{\nq}$ and $\bfo = (\a,\b,\c,b_0,\e,\f)\in \rset_+^6$ with $b_0 \leq b$, it holds
  \begin{equation}
    \label{eq:7}
      \rmV_{\delta} \Voi(x,v) \leq        \Vlyapc[\tilde{\bfo}][i](x,v)  \eqsp,
    \end{equation}
    where $\tilde{\bfo} = (\tilde{\a},\tilde{\b},\tilde{\c},\b_0,\tilde{\e},\tilde{\f})$ with
    \begin{align}
    \tilde{\a} & = a\left(1-\frac{\mtt\delta^2}{2}\right) + 8 b \delta^2 {- b_0 \frac{ \delta^2\mtt^2}{2}} - \frac{\delta \mtt c}{2} + \delta (b-b_0)\eqsp, \\
    \tilde{\b} &= 2 \delta^2 \a + (1+\epsilon\delta + 31\delta^2)\b {+ 2 b_0\delta^2} +\delta(1+10\delta) \c  + \delta (b-b_0)\eqsp,\\
    \tilde{\c} &= 2 \delta \a + \c \eqsp, \\
    \tilde{\e} & = \e + \delta(2\a\delta + 2\b+2\delta b_0+2\c+{7}\e) \eqsp,\\
    \tilde{\f} & =  \f +\a\delta^2 \Mtt+ 2\b_0 \delta^2 \mtt \Mtt  + \delta \Mtt c \eqsp.
    \end{align}
\end{lemma}

\begin{proof}
Note first that the assumption $\epsilon \leq \mtt/4$ and the bound $\mtt \leq 1$ from~\eqref{eq:bound_m} imply that $\epsilon \leq 1/4$.
  Let $x,v \in \rsetd$, $i \in \iint{1}{\nq}$, $\delta \leq \mtt/40$ and $\bfo = (a,b,c,b_0,e,f)$.
To alleviate the notation, we introduce, for~$j \in\iint{1}{\nq}$, the partial derivative of~$U$ with respect to the $j$-th variable, namely the function $\psi_j: \rset^d \to \rset^d$ defined for any $\tilde{x} = \newOld{(\tilde x_1,\dots,\tilde x_{d_0})}\in \mathbb R^{d}$ 
 as
  \begin{equation}
    \label{eq:5}
    \psi_j(\tilde{x}) = \nabla \Uq(\tilde{x}_j) + \frac{\epsilon}{\nq}\sum_{j'=1}^{\nq}  \nabla \Wq(\tilde{x}_j-\tilde{x}_{j'}) \eqsp.
  \end{equation}
With this notation, it holds by definition of $\rmV_{\delta}$ that 
  \begin{equation}
    \label{eq:drift_eq_1}
\rmV_{\delta} \Voi(x,v) = \Voi(y,p)\eqsp,
\end{equation}
where $y = \newOld{(y_1,\dots,y_{d_0})}$ and $p=\newOld{(p_1,\dots,p_{d_0})}$ 
 with, for $j \in \iint{1}{\nq}$,
\begin{equation}
  \label{eq:def_y_j_v_j}
  y_j = x_j + \delta v_j -\frac{\delta^2}{2}\psi_j(x) \eqsp, \qquad  p_j = v_j -\frac{\delta}{2} \defEns{\psi_j(x) + \psi_j(y)} \eqsp.
\end{equation}

We start by giving some useful preliminary estimates. First, by a discrete Cauchy--Schwarz inequality and~\eqref{eq:pot_drift_sub_linear},
\begin{align}
\norm{\psi_i(x)-\psi_j(x)}^2 & \leq 2 \norm{\nabla \Uq(x_i)-\nabla \Uq(x_j)}^2 \\
& \qquad + \frac{2\epsilon^2}{\nq} \sum_{k=1}^{\nq} \norm{\nabla \Wq(x_k-x_{i})-\nabla \Wq(x_k-x_{j})}^2 \\
& \leq 2\left(1+\epsilon^2\right)\norm{x_i-x_j}^2.
\label{eq:bound_psi_i_minus_psi_j}
\end{align}
Next, by two successive discrete Cauchy--Schwarz inequalities and~\eqref{eq:pot_drift_sub_linear}, we have, for any $j \in \iint{1}{\nq}$,
\begin{align}
  \norm{y_{j}-x_{j}}^2 & \leq \frac{\delta^4}{2} \norm{\psi_j(x)}^2 + 2\delta^2 \norm{v_j}^2\\
  & \leq 2 \parenthese{ \delta^4 \norm{ \nabla \Uq(x_{j})}^2 + \frac{\delta^4 \epsilon^2}{\nq} \sum_{j'=1}^{\nq} \norm{\nabla \Wq(x_{j}-x_{j'})}^2 + \delta^2 \norm{v_j}^2 } \\
    \label{eq:bound_drift_lem_1_y_i_x_i}
  & \leq 2 \parenthese{ \delta^4  \norm{ x_{j} }^2 + \frac{\delta^4\epsilon^2}{\nq} \sum_{j'=1}^{\nq} \norm{x_{j}-x_{j'}}^2+ \delta^2 \norm{v_{j}}^2 } \eqsp .
\end{align}
In view of~\eqref{eq:pot_drift_sub_linear} and discrete Cauchy--Schwarz inequalities, it holds, for any $\tilde{x} = \newOld{(\tilde x_1,\dots,\tilde x_{d_0})} , \tilde{y} = \newOld{(y_1,\dots,y_{d_0})} \in \mathbb R^{d_0}$ 
and $i \in \iint{1}{\nq}$,
\begin{align}
\norm{  \psi_i(\tilde{x}) - \psi_i(\tilde{y})}^2
& \leqslant 2 \norm{ \tilde{x}_i-\tilde{y}_i}^2 + \frac{2\epsilon^2}{\nq} \sum_{j=1}^{\nq}\norm{\tilde{x}_i-\tilde{x}_{j}-\left(\tilde{y}_i-\tilde{y}_{j}\right)}^2\\
\label{eq:bound_phi_j_x}
& \leqslant 2\left(1 +2  \epsilon^2\right)\norm{\tilde{x}_i-\tilde{y}_i}^2 + \frac{4  \epsilon^2}{\nq} \sum_{j=1}^{\nq}\norm{\tilde{x}_{j}-\tilde{y}_{j}}^2 \eqsp, 
\end{align}
as well as
\begin{equation}
\label{eq:bound_psi_j_square}
\norm{\psi_i(\tilde{x})}^2 \leqslant 2 \norm{\tilde{x}_i}^2 +\frac{2 \epsilon^2}{\nq} \sum_{j=1}^{\nq} \norm{\tilde{x}_i-\tilde{x}_{j}}^2  \eqsp.
\end{equation}
As a result, using \eqref{eq:bound_drift_lem_1_y_i_x_i} to bound all the terms on right hand side of~\eqref{eq:bound_phi_j_x}, we obtain 
\begin{align}
    \norm{  \psi_i(x) - \psi_i(y)}^2 &\leq 4 \left(1 +2  \epsilon^2\right) 
    \parenthese{\delta^4  \norm{ x_i }^2 +   \delta^4  \frac{\epsilon^2}{\nq} \sum_{j=1}^{\nq} \norm{x_i-x_j}^2+ \delta^2 \norm{v_i}^2} \\
    &\ \ + \frac{8  \epsilon^2}{\nq} \sum_{j=1}^{\nq} \left( \delta^4  \norm{ x_{j} }^2 + \frac{\delta^4\epsilon^2}{\nq} \sum_{j'=1}^{\nq} \norm{x_{j}-x_{j'}}^2+ \delta^2 \norm{v_{j}}^2 \right)\\
    &\leq 4 \left(1 +2  \epsilon^2\right) 
    \parenthese{\delta^4  \norm{ x_i }^2 +   \delta^4  \frac{\epsilon^2}{\nq} \sum_{j=1}^{\nq} \norm{x_i-x_j}^2+ \delta^2 \norm{v_i}^2} \\
    &\ \ + \frac{16  \epsilon^2}{\nq} \sum_{j=1}^{\nq} \left( \delta^4  \norm{ x_{i} }^2 +   (1+2\epsilon^2)\delta^4 \norm{ x_{i}-x_j }^2 + \delta^2 \norm{v_{i}}^2 + \delta^2 \norm{v_{j}-v_i}^2 \right)\\
    & \leq 30 \parenthese{\delta^4  \norm{ x_i }^2 + \frac{\delta^4\epsilon^2}{\nq} \sum_{j=1}^{\nq} \norm{x_i-x_j}^2+ \delta^2 \norm{v_i}^2 + \frac{  \epsilon^2 \delta^2}{\nq} \sum_{j=1}^{\nq}  \norm{v_i-v_j}^2} \eqsp ,
      \label{eq:delta_psi_x_y}
\end{align}
where we used~$20+40\epsilon^2 \leq 30$ in the last inequality, and the second inequality has been obtained by using
\begin{equation}
\label{eq:astuce_somme_xi_2}    
\frac{1}{\nq} \sum_{j=1}^{\nq} \abs{x_j}^2 \leq 2\abs{x_i}^2 + \frac{2}{\nq} \sum_{j=1}^{\nq} \abs{x_j-x_i}^2,  
\end{equation}
and a similar inequality for~$v$, as well as the inequality $\norm{x_{j}-x_i + x_i-x_{j'}}^2 \leq 2 \norm{x_{j}-x_i}\newOld{^2} + 2\norm{x_i-x_{j'}}^2$ to write
\[
\begin{aligned}
\frac{1}{\nq^2}\sum_{j=1}^{\nq} \sum_{j'=1}^{\nq} \norm{x_{j}-x_{j'}}^2 & = \frac{1}{\nq^2}\sum_{j=1}^{\nq} \sum_{j'=1}^{\nq} \norm{x_{j}-x_i + x_i-x_{j'}}^2 \leq \frac{4}{\nq}\sum_{j=1}^{\nq} \norm{x_{j}-x_i}^2 \eqsp.
\end{aligned}
\]

We can now turn to estimating the various terms in~$\Vlyapc[\tilde{\bfo}][i](x,v)$. In view of~\eqref{eq:bound_drift_lem_1_y_i_x_i},
\begin{align}
\norm{y_i}^2 & = \norm{x_i}^2 + 2 \ps{x_i}{\delta v_i-\frac{\delta^2}{2}\psi_i(x)} +  \norm{y_i-x_i}^2 \\
& \leq  \norm{x_i}^2- \delta^2\ps{x_i}{\nabla \Uq(x_i)} + \delta^2\epsilon \norm{x_i}^2 +   \frac{\delta^2\epsilon}{\nq} \sum_{j=1}^{\nq} \norm{x_i-x_j}^2 \\
& \qquad +2 \delta\ps{x_i}{v_i} +       2 \parenthese{ \delta^4  \norm{ x_i }^2 + \frac{\delta^4\epsilon^2}{\nq} \sum_{j=1}^{\nq} \norm{x_i-x_j}^2 + \delta^2 \norm{v_i}^2 }\eqsp,
\end{align}
where $\ps{x_i}{\nabla W^{(q)}(x_i-x_j)}$ is bounded with a discrete Cauchy--Schwarz inequality and~\eqref{eq:pot_drift_sub_linear}.
Therefore, using first the drift condition~\eqref{eq:ass_drift} and the inequality $2\delta^2 \epsilon \leq 1$, and then $\max(\delta,\epsilon) \leq \mtt/4 \leq 1/4$,
\begin{align}
\label{eq:11}
\norm{y_i}^2 & \leq \left(1+\delta^2 \epsilon  + 2 \delta^4 \right)\norm{x_i}^2+2 \delta\ps{x_i}{v_i} +2\delta^2  \norm{v_i}^2 \\
&  \qquad - \delta^2\ps{x_i}{\nabla \Uq(x_i)} +  \frac{\delta^2  \epsilon(1 + 2 \delta^2 \epsilon)}{\nq}\sum_{j=1}^{\nq} \norm{x_i-x_j}^2\\
& \leq \left(1-\delta^2 \mtt +\delta^2 \epsilon  + 2 \delta^4 \right)\norm{x_i}^2+2 \delta\ps{x_i}{v_i} +2\delta^2  \norm{v_i}^2 + \delta^2 \Mtt + \frac{2 \delta^2  \epsilon} {\nq}\sum_{j=1}^{\nq} \norm{x_i-x_j}^2 \\
& \leq \left(1-\frac{\delta^2 \mtt}{2} \right)\norm{x_i}^2+2 \delta\ps{x_i}{v_i} +2\delta^2  \norm{v_i}^2  + \delta^2 \Mtt +  \frac{2 \delta^2  \epsilon}{\nq}\sum_{j=1}^{\nq} \norm{x_i-x_j}^2 \eqsp.
\label{eq:bound_y_i}
\end{align}
Similarly, using a discrete Cauchy--Schwarz inequality to bound~$\delta \langle v_i,\psi_i(x)-\psi_i(y)\rangle$,
\begin{align}
\norm{p_i}^2 & = \norm{v_i}^2 - \delta \ps{v_i}{ \psi_i(x) + \psi_i(y) } + \frac{\delta^2}{4} \norm{ \psi_i(x) + \psi_i(y) }^2   \\
&  \leqslant\norm{v_i}^2-2 \delta\left\langle v_i, \psi_i(x)\right\rangle  +\frac{\delta^2 }{2}\norm{v_i}^2  +\frac{1}{2}\norm{\psi_i(x)-\psi_i(y)}^2 \\
& \qquad \qquad + \frac{\delta^2}{2}\left[4\norm{\psi_i(x)}^2+\norm{\psi_i(x)-\psi_i(y)}^2\right].
\end{align}
Using 
\[
2\left\langle v_i, \frac{1}{d_0}\sum_{j=1}^{d_0} \nabla W^{(q)}(x_i-x_j)\right\rangle \leq \norm{v_i}^2 + \frac{1}{d_0}\sum_{j=1}^{d_0} \norm{x_i-x_j}^2,
\]
as well as~\eqref{eq:bound_psi_j_square}, it follows that
\begin{align}
\norm{p_i}^2 & \leqslant\norm{v_i}^2-2 \delta\left\langle v_i, \nabla \Uq(x_i)\right\rangle  +\left(\delta\epsilon+ \frac{\delta^2}{2}\right)\norm{v_i}^2  + \frac{1+\delta^2}{2} \norm{\psi_i(x)-\psi_i(y)}^2\\
& \qquad \qquad + 4\delta^2 \left[ \norm{x_i}^2 +\frac{\epsilon^2}{\nq} \sum_{j=1}^{\nq} \norm{x_i-x_j}^2 \right] + \frac{\delta \epsilon}{\nq} \sum_{j=1}^{\nq}\norm{x_i-x_j}^2 \\
& \leqslant \left(1+\delta \epsilon + \frac{\delta^2}{2}\right) \norm{v_i}^2 + 4\delta^2 \norm{x_i}^2 -2 \delta\left\langle v_i, \nabla \Uq(x_i)\right\rangle + \frac{\delta \epsilon(1+4\delta \epsilon)}{\nq} \sum_{j=1}^{\nq} \norm{x_i-x_j}^2 \\
& \qquad \qquad +  \norm{\psi_i(x)-\psi_i(y)}^2\\
& \leq \left[1+\delta\epsilon+\frac{\delta^2}{2}+30\delta^2 \right]\norm{v_i}^2 -2\delta\ps{v_i}{\nabla \Uq(x_i)} + 4\delta^2 \left[1+8\delta^2 \right]\norm{x_i}^2  \\
\label{eq:bound_pi}
& \quad + \frac{\delta\epsilon(1+4\delta\epsilon+30\delta^3\epsilon)}{\nq} \sum_{j=1}^{\nq} \norm{x_i-x_j}^2+ \frac{30 \epsilon^2 \delta^2}{\nq} \sum_{j=1}^{\nq} \norm{v_i-v_j}^2 \\
& \leq \left(1+\delta\epsilon +31 \delta^2\right)\norm{v_i}^2 -2\delta\ps{v_i}{\nabla \Uq(x_i)} + 8\delta^2\norm{x_i}^2  \\
\label{eq:bound_pi}
& \qquad + \frac{2\delta\epsilon}{\nq} \sum_{j=1}^{\nq} \norm{x_i-x_j}^2+ \frac{30 \epsilon^2 \delta^2}{\nq} \sum_{j=1}^{\nq} \norm{v_i-v_j}^2 
\eqsp,
\end{align}
where we used~\eqref{eq:delta_psi_x_y} to obtain the third inequality and that $\delta \leq 1/40$ and $\epsilon\leq 1/4$ to simplify the bounds in the last inequality. Next, 
\begin{align}
  \left\langle y_i, p_i\right\rangle & =\left\langle x_i+\delta v_i-\frac{\delta^2}{2} \psi_i(x), v_i-\frac{\delta}{2}\left(\psi_i(x)+\psi_i(y)\right)\right\rangle \\
  &=\left\langle x_i, v_i\right\rangle-\frac{\delta}{2}\left\langle x_i, \psi_i(x)+\psi_i(y)\right\rangle +\delta\norm{v_i}^2 -\frac{\delta^2}{2}\left\langle v_i, \psi_i(x)\right\rangle \\
  & \qquad -\frac{\delta^2}{2}\left\langle v_i, \psi_i(x)+\psi_i(y)\right\rangle +\frac{\delta^3}{4}\left\langle\psi_i(x), \psi_i(x)+\psi_i(y)\right\rangle \\
  & \leqslant\left\langle x_i, v_i\right\rangle-\delta\left\langle x_i, \psi_i(x)\right\rangle  +\frac{\delta}{2}\norm{x_i} \norm{\psi_i(x)-\psi_i(y)} +\delta\norm{v_i}^2 -\frac{3\delta^2}{2}\left\langle v_i, \psi_i(x)\right\rangle\\
   & \qquad +\frac{\delta^2}{2}\norm{v_i} \norm{\psi_i(x)-\psi_i(y)} + \frac{\delta^3}{2}  \norm{\psi_i(x)}^2+\frac{\delta^3}{4}\norm{\psi_{i}(x)}\norm{\psi_i(x)-\psi_i(y)}   \\
  & \leq \left\langle x_i, v_i\right\rangle-\delta\left\langle x_i, \psi_i(x)\right\rangle  +\frac{1}{4}\left( \delta^2 \norm{x_i}^2+\norm{\psi_i(x)-\psi_{i}(y)}^2\right)-\frac{3\delta^2}{2}\left\langle v_i, \psi_i(x)\right\rangle \\
  & \qquad +\delta\left(1+\frac{\delta}{4} \right) \norm{v_i}^2  +\frac{5\delta^3}{8} \norm{\psi_i(x)}^2 +\frac{\delta^2(2+\delta)}{8}\norm{\psi_i(x)-\psi_i(y)}^2.
  \end{align}
Let us emphasize that the term~$-\delta\left\langle x_i, \psi_i(x)\right\rangle$ is crucial to obtained a dissipation of order~$\delta$ in the position variables. This motivates the inclusion of a cross term~$\left\langle x_i, v_i\right\rangle$ in the Lyapunov function. We next use~\eqref{eq:pot_drift_sub_linear}, \eqref{eq:bound_psi_j_square} and~\eqref{eq:delta_psi_x_y} and the drift condition~\eqref{eq:ass_drift} to write
  \begin{align}
  \left\langle y_i, p_i\right\rangle & 
  \leq \left\langle x_i, v_i\right\rangle-\delta\left\langle x_i, \nabla \Uq(x_i)\right\rangle - \frac{3\delta^2}{2}\left\langle v_i, \nabla \Uq(x_i)\right\rangle \\
  & \qquad + \delta \left(\norm{x_i} + \frac{3\delta}{2}\norm{v_i}\right)\frac{\epsilon}{\nq}\sum_{j=1}^{\nq} \norm{\nabla W^{(q)}(x_i-x_j)} +\delta\left(1+\frac{\delta}{4} \right) \norm{v_i}^2 \\
  & \qquad + \frac{\delta^2}{4}\norm{x_i}^2 + \frac{7\delta^3}{8} \norm{\psi_i(x)}^2 + \frac{2+\delta^2(2+\delta)}{8}\norm{\psi_i(x)-\psi_i(y)}^2 \\
  & \leq \left\langle x_i, v_i\right\rangle-\delta \mtt \norm{x_i}^2 + \delta \Mtt + \frac{3\delta^2}{4}\left( \norm{v_i}^2 + \norm{\nabla \Uq(x_i)}^2\right) +\delta\left(1+\frac{\delta}{4} \right) \norm{v_i}^2\\
  & \qquad + \delta \epsilon \left(\frac{1}{2} \norm{x_i}^2 + \frac{9\delta^2}{8}\norm{v_i}^2 + \frac{1}{\nq}\sum_{j=1}^{\nq} \norm{\nabla W^{(q)}(x_i-x_j)}^2 \right) + \frac{\delta^2}{4}\norm{x_i}^2 \\
  & \qquad + \frac{7\delta^3}{4} \left( \norm{x_i}^2 +\frac{\epsilon^2}{\nq} \sum_{j=1}^{\nq} \norm{x_i-x_j}^2 \right)\\
  & \qquad + \frac{30\delta^2 [2+\delta^2(2+\delta)]}{8}  
        \parenthese{\delta^2  \norm{ x_i }^2 + \frac{\delta^2\epsilon^2}{\nq} \sum_{j=1}^{\nq} \norm{x_i-x_j}^2+ \norm{v_i}^2 + \frac{\epsilon^2 }{\nq} \sum_{j=1}^{\nq} \norm{v_i-v_j}^2}\\
  &   \leq \left\langle x_i, v_i\right\rangle- \frac{\delta \mtt}{2} \norm{x_i}^2 + \delta \Mtt + \delta(1+10\delta) \norm{v_i}^2
    \\
    \label{eq:bound_ps_drift}
  & \qquad + \frac{2 \delta \epsilon}{\nq} \sum_{j=1}^{\nq} \norm{x_i-x_j}^2  + \frac{8 \delta^2\epsilon^2}{\nq} \sum_{j=1}^{\nq} \norm{v_i-v_j}^2   \eqsp,
\end{align}
where we used in the last inequality that $\epsilon \leq \mtt/4$ and $\delta \leq \mtt / 40 \leq 1/40$.

Using first a Taylor expansion and (recalling~\eqref{eq:pot_drift_sub_linear} and~\eqref{eq:bound_psi_j_square})
\begin{equation}
    \ps{\nabla \Uq(x_i) }{\psi_i(x)} \leq \frac12 \left(\norm{\nabla \Uq(x_i)}^2 + \norm{\psi_i(x)}\right) \leq \frac12 \parenthese{3 \norm{x_i}^2 + \frac{2 \epsilon^2}{d_0} \sum_{j=1}^{d_0} \norm{x_i-x_j} }\,,
\end{equation}
then~\eqref{eq:bound_drift_lem_1_y_i_x_i} and~\eqref{eq:tail_nabla_U}, it follows
\begin{align}
  \Uq(y_i)& \leq \Uq(x_i) + \ps{\nabla \Uq(x_i)}{y_i-x_i} + \frac12 \norm{y_i-x_i}^2 \\
  & \leq  \Uq(x_i) + \ps{\nabla \Uq(x_i)}{\delta v_i-\frac{\delta^2}{2} \psi_i(x)} + \frac12 \norm{y_i-x_i}^2 \\
          & \leq  \Uq(x_i) -\frac{\delta^2}{2} \norm{\nabla \Uq(x_i)}^2 + \delta \ps{\nabla \Uq(x_i)}{v_i} \\
          & \qquad + \frac{\delta^2\epsilon}{2}\parenthese{2  \norm{x_i}^2 + \frac{2}{\nq}\sum_{j=1}^{\nq} \norm{x_i-x_j}^2} + {\frac12} \norm{y_i-x_i}^2  \\
          & \leq \Uq(x_i) -\frac{\delta^2 \mtt^2}{2} \norm{x_i}^2 + \delta^2 \mtt \Mtt + \delta \ps{\nabla \Uq(x_i)}{v_i} \\
          & \qquad+\delta^2 \left(\epsilon + \delta^2 \right) \norm{x_i}^2 + \delta^2 \norm{v_i}^2  +  \frac{\epsilon \delta^2(1+\delta^2 \epsilon)}{\nq} \sum_{j=1}^{\nq} \norm{x_i-x_j}^2 \\
          & \leq \Uq(x_i) -\frac{\delta^2 \mtt^2}{4} \norm{x_i}^2 + \delta^2 \mtt \Mtt + \delta \ps{\nabla \Uq(x_i)}{v_i} + \delta^2 \norm{v_i}^2 \\
          \label{eq:bound_uq}
          & \qquad  +  \frac{2\epsilon \delta^2}{\nq} \sum_{j=1}^{\nq} \norm{x_i-x_j}^2 \eqsp.
\end{align}
,Using first
$\norm{\mathbf{a} + \mathbf{b}}^2 \leq (1+\delta) \norm{\mathbf{a}}^2 + (1+\delta^{-1})\norm{\mathbf{b}}^2$, then~\eqref{eq:bound_psi_i_minus_psi_j} as well as $\delta \leq \mtt/40\leq 1/40 $ and $\epsilon \leq \mtt/4\leq 1/4$, we obtain 
\begin{align}
\norm{y_i-y_j}^2 & = \norm{x_i-x_j-\frac{\delta^2}{2}\left(\psi_{i}(x)-\psi_j(x)\right)+\delta(v_i-v_j)}^2 \\
& \leq (1+\delta)\norm{x_j-x_j}^2  + \delta^2\left(1+\delta^{-1}\right)\norm{v_i-v_j-\frac{\delta}{2}\left(\psi_i(x)-\psi_j(x)\right)}^2 \\
&\leq (1+\delta)\norm{x_i-x_j}^2 + 2 \delta(1+\delta) \norm{v_i-v_j}^2 + \frac{\delta^3(1+\delta)}{2} \norm{\psi_i(x)-\psi_j(x)}^2 \\
  & \leq \left[1+\delta+\delta^3(1+\delta)\left(1+\epsilon^2\right) \right]\norm{x_i-x_j}^2 
    + 2\delta(1+\delta) \norm{v_i-v_j}^2\\
  \label{eq:bound_y_iy_j}
    & \leq (1+2\delta)\norm{x_i-x_j}^2 + 4\delta \norm{v_i-v_j}^2 \eqsp.
\end{align}
Finally, using twice~\eqref{eq:bound_psi_i_minus_psi_j},
\begin{align}
\norm{p_i-p_j}^2 
& \leqslant(1+\delta)\norm{v_i-v_j}^2 + \frac{\delta^2\left(1+\delta^{-1}\right)}{4} \norm{ \psi_i(y)+\psi_i(x)-\psi_j(y)-\psi_j(x)}^2  \\
& \leqslant(1+\delta)\norm{v_i-v_j}^2 + \frac{\delta(1+\delta)}{2} \left( \norm{ \psi_i(x)-\psi_j(x)}^2 + \norm{ \psi_i(y)-\psi_j(y)}^2 \right)\\
& \leqslant(1+\delta)\norm{v_i-v_j}^2 + \delta(1+\delta)\left(1+\epsilon^2\right) \left( \norm{ x_i-x_j}^2 + \norm{y_i-y_j}^2 \right)\\
& \leqslant \delta(1+\delta)\left(1+\epsilon^2\right) \left( 4\norm{ x_i-x_j}^2 + 3\delta^2\norm{v_i-v_j}^2 + {\frac{3\delta^4}{4}}\norm{\psi_i(x)-\psi_j(x)}^2\right)\\
& \qquad + (1+\delta)\norm{v_i-v_j}^2 \\
& \leq \delta(1+\delta)\left(1+\epsilon^2\right) \left(4 + \frac{3\delta^2}{2}\left(1+\epsilon^2\right)\right) \norm{ x_i-x_j}^2 + (1+{2\delta})\norm{v_i-v_j}^2 \\
& \leq {5\delta}\norm{ x_i-x_j}^2 + (1+{2\delta})\norm{v_i-v_j}^2 \eqsp .
\end{align}
Combining this last inequality with \eqref{eq:bound_y_i}-\eqref{eq:bound_pi}-\eqref{eq:bound_ps_drift}-\eqref{eq:bound_uq}-\eqref{eq:bound_y_iy_j} in
\eqref{eq:drift_eq_1}, and (recalling $b \geq b_0$)
\[
{-2}\delta (b-b_0) \psLigne{v_i}{\nabla \Uq(x_i)} \leq \delta (b-b_0)\left( \norm{v_i}^2 + \norm{x_i}^2 \right),
\]
completes the proof.
\end{proof}

\paragraph{Quantifying how prefactors change over several steps.}\  In view of Lemma~\ref{lem:drift_v1}, controlling the evolution of the Lyapunov function along several Verlet steps amounts to a discrete Grönwall-type result for the sequence of parameters  defined recursively by the relation $\bfo \mapsto \tilde{\bfo}$ of Lemma~\ref{lem:drift_v1}. Notice that, among these parameters, $b_0$ is fixed and $f$ does not influence the other parameters, so that we focus first on the relations intertwining the parameters $a,b,c,e$. We state an abstract result, independent from the specific constants appearing in Lemma~\ref{lem:drift_v1}, except that it is crucial to keep track  of the order of the various parts in terms of the step size $\delta$. More specifically, Lemma~\ref{lem:drift_v1} reads $\tilde{\bfo} = (I+\delta A + \delta^2 B)\bfo$ for some matrices $A,B$ and thus, iterating this $k$ times, as long as $k\delta$ is small enough, it is sufficient to keep the lowest order terms in $k\delta$ and $\delta$ in $(I+\delta A + \delta^2 B)^k$. This is essentially the content of the next lemma.

\begin{lemma}
  \label{lem:rec_for_drift}
 Consider the four real sequences
  $(a_k)_{k\in\nset},(b_k)_{k\in\nset},(c_k)_{k\in\nset},(e_k)_{k\in\nset}$
such as $a_0,b_0,c_0,e_0 >0$ and defined by
  \begin{align}
    a_{k+1} &= (1-\delta^2 C_{a,1}) a_k - \delta c_k C_{a,2} + \delta^2 b_k C_{a,3}+ \delta e_k C_{a,4} + C_{a,5}\delta(b_{k}-b_0)\eqsp, \\
    b_{k+1} & = (1+C_{b,1}\delta + C_{b,2}\delta^2) b_k + \delta^2 C_{b,3} a_k + \delta c_k C_{b,4} + \delta e_k C_{b,5} + \delta(b_{k}-b_0) \eqsp, \\
    c_{k+1} &= c_k + C_{c,1} \delta a_k \eqsp, \quad e_{k+1} = e_k + C_{e,1} \delta (a_k\delta+b_k+c_k+e_k) \eqsp,
  \end{align}
  for some non-negative real constants $\{C_{\square,\triangle}\,: \, \square \in \{a,b,c,e\} \, ,\triangle \in \{1,2,3,4,5\}\}$ and $\delta >0$.
  Then, for any $\delta >0$ and $k \in \nset$ satisfying
  \begin{equation}
    \label{eq:10}
\delta \leq \tilde{\delta}_1 \wedge \tdelta \eqsp, \quad k\delta \leq \tilde{T}_1 \wedge \tilde{T}_2\wedge\tilde{T} \eqsp,
  \end{equation}
 for $\tilde{T},\tdelta >0$, with 
\begin{equation}
    \label{eq:10_v2}
    \begin{aligned}
      \tilde{\delta}_1 &=  [ C_{b,2}]^{-1} \wedge [a_0 C_{a,2} C_{c,1}/\{8(\tC_{e,1} C_{a,4} +  \tC_{b,3} ( \tdelta C_{a,3} + C_{a,5} ))\}] \wedge [4 C_{a,1}]^{-1} \wedge [4 C_{c,1}C_{a,2} \tilde{C}_{a,3}]^{-1} 
      \\
      \tilde{T}_1 & = 4^{-1}\{\tilde{C}_{b,1}^{-1} \wedge \tilde{C}_{b,2}^{-1/2} \wedge C_{e,1}^{-1} \wedge 1\} \eqsp, \\
      \tilde{T}_2 & = 4^{-1}\{[C_{a,2} /(4C_{b,4} C_{a,5})] \wedge  [a_0 C_{a,2} C_{c,1}(2 C_{a,5}\tC_{b,2} b_0)^{-1}]\wedge  [a_0(2 C_{a,2}c_0)^{-1}]\wedge [2 C_{c,1}C_{a,2}]^{-1/2}\} \eqsp, \\
    \end{aligned}
  \end{equation}
  assuming moreover that
  \begin{equation}
  \begin{aligned}
    & 2 e_0 C_{a,4} \leq c_0 C_{a,2}/4   \eqsp, \qquad  C_{a,5}(C_{b,1} b_0+ C_{b,5} e_0) +\tC_{e,1} C_{a,4} \leq C_{a,2} C_{c,1} a_0 / 8 \eqsp, 
  \end{aligned}
      \label{eq:condition_c0}
\end{equation}
 it holds
  \begin{equation}
  \begin{aligned}
    &a_0/2 \leq a_k \leq a_0 + \tilde{C}_{a,3}k\delta^2  \eqsp, \quad b_0 \leq  b_k \leq 2 b_0 + \tilde{C}_{b,3}  k \delta^2    \eqsp, \quad e_k \leq 2 e_0 + \tilde{C}_{e,1} k \delta + \tilde{C}_{e,2} k \delta^2  \eqsp, \\
    &c_0 + k\delta a_0 C_{c,1}/2 \leq c_k \leq     c_0 + k\delta a_0 C_{c,1} + (\delta k)^2 \delta C_{c,1} \tilde{C}_{a,3}\\
    & b_k \leq b_0(1+k\delta \tilde{C}_{b,1} + (k\delta)^2 \tilde{C}_{b,2}) + \delta^2 k \tilde{C}_{b,3}\\
    & e_k \leq e_0(1+k\delta {C}_{e,1}) + \tilde{C}_{e,1} k \delta + \tilde{C}_{e,2} k \delta^2 \\
        & a_k\leq a_0 (1-k\delta \tilde{C}_{a,1} -(k\delta)^2 \tilde{C}_{a,2}) + \delta^2 k \tilde{C}_{a,3} \\
        & a_k\geq a_0 (1-k\delta (C_{a,2}c_0/a_0 + \delta C_{a,1}) -  (k\delta)^2 C_{a,2} C_{c,1} (1+ \delta \tilde{T} \tilde{C}_{a,3}) )\\
    & c_k \geq c_0 + C_{c,1} k\delta a_0 -(k\delta)^2 C_{c,1} (C_{a,2}c_0 + \delta C_{a,1} a_0) -(k\delta)^3 C_{c,1}^2 C_{a,2}a_0(1+\delta \tilde{T} \tilde{C}_{a,3})   \eqsp,
  \end{aligned}\label{eq:14}
\end{equation}
  where
  \begin{align}
    \label{eq:C_a_2_tilde}
    \tilde{C}_{a,1} &= c_0 C_{a,2}/(2a_0) \eqsp, \quad \tilde{C}_{a,2} = a_0 [ C_{a,1} \wedge C_{a,2} C_{c,1} / 16] \eqsp, \\
     \tilde{C}_{a,3} &= 2 b_0 C_{a,3} + a_0C_{a,1} \tT (\tC_{a,1} + \tT \tC_{a,2})  \eqsp, \\
    \tilde{C}_{b,1} & = C_{b,1} + c_0 C_{b,4}/b_0 + e_0 C_{b,5}/b_0 \eqsp, \\
    \tC_{b,2} &= 4[ C_{b,1} \tC_{b,1} + \tC_{b,1} + C_{b,4}(C_{c,1} a_0 + \tC_{e,1} C_{b,5}))/b_0] \eqsp, \\
    \tilde{C}_{b,3} &= 3\{b_0 C_{b,2} + a_0 C_{b,3} + \tT(C_{b,3} \tC_{a,3} + b_0 C_{b,2} \tC_{b,1} + C_{b,5} \tC_{e,2} + \tT \tC_{a,3} C_{c,1} C_{b,4})\} \\
    \tC_{e,1} & =2 C_{e,1} (2b_0 + c_0 + \tT(C_{e,1} e_0 + C_{c,1} a_0 +\delta \tC_{b,3}))\\
    \tC_{e,2} & = 2 C_{e,1} (a_0 + \tT \delta \tC_{a,3} + \tT^2 C_{c,1} \tC_{a,3}))\eqsp.                        
  \end{align}
\end{lemma}

\begin{proof}
  The proof is by induction. Equation \eqref{eq:14} is trivially true
  for $k=0$. Suppose that it is true for $k \in\nset$ and suppose $\delta,k+1$ satisfy \eqref{eq:10}.
  Then by definition, the induction hypothesis implies that $c_0 + (k+1)\delta a_0 C_{c,1}/2 \leq c_{k+1} \leq     c_0 + (k+1)\delta a_0 C_{c,1} + (k+1)^2\delta^3 C_{c,1}  \tilde{C}_{a,3}$.

  Similarly, it follows that
  \begin{align}
      & b_{k+1}
      \leq [(1+k\delta \tilde{C}_{b,1} + (k\delta)^2 \tilde{C}_{b,2})b_0+\tilde{C}_{b,3} \delta^2 k](1+C_{b,1}\delta + C_{b,2} \delta^2) \\
    &\quad + \delta^2 C_{b,3}(a_0 + \tC_{a,3} \delta^2 k) + \delta C_{b,4}(c_0+\delta k C_{c,1} a_0 + (\delta k )^2 \delta C_{c,1}  \tC_{a,3}) \\
     &\quad  +\delta C_{b,5}(2e_0 + \delta k \tC_{e,1} + \delta^2 k \tC_{e,2}) \\
      & \leq b_0 + \delta b_0(k \tilde{C}_{b,1} +  C_{b,1}  + c_0 C_{b,4}/b_0+ 2 C_{b,5}e_0/b_0) \\
      & \quad + \delta^2 k b_0 \{\tC_{b,2} k  + k \delta \tC_{b,2}C_{b,1} + C_{b,1}\tC_{b,1} + \tC_{b,1} \\
    & \quad \qquad \qquad + (C_{b,4}/b_0) (C_{c,1}a_0 +  \tC_{e,1}C_{b,5})+ k \delta (\delta \tC_{b,2} C_{b,2} + \tC_{b,2}) \} \\
      & \quad + \delta^2 (k \tC_{b,3} + \tC_{b,3}(k\delta C_{b,1} + k\delta^2 C_{b,2}) + b_0 C_{b,2} + a_0 C_{b,3} \\
    &\quad \qquad \qquad  + \tT^2 \tC_{a,3} C_{c,1} C_{b,4}+ \tT C_{b,3} \tC_{a,3} +\tT b_0 C_{b,2} \tC_{b,1} + \tT C_{b,5}\tC_{e,2}) \\
    &\leq b_0(1+(k+1)\delta\tilde{C}_{b,1}+k(k+2)\delta^2 \tilde{C}_{b,2}) + \tC_{b,3} (k+1)\delta^2   \eqsp,
  \end{align}
  where we have used that $ k\delta \leq 4^{-1} [C_{b,1}^{-1} \wedge 1]$, $\delta \leq C_{b,2}^{-1}$, $\tilde{C}_{b,1} \geq C_{b,1}$ and $\tilde{C}_{b,2} \geq C_{b,2}$ by definition.
  In addition, $b_0 \leq b_{k+1} \leq 2 b_0 + \delta^2 k \tC_{b,3}$ since $\delta,k+1$ satisfy \eqref{eq:10}.

  Similarly, using that $k\delta \leq (4 C_{e,1})^{-1}$, we have
  \begin{align}
    &    e_{k+1} \leq [(1+k\delta C_{e,1})e_0 + \delta k \tC_{e,1} + \delta^2 k \tC_{e,2}](1+C_{e,1}\delta) \\
    & \quad + C_{e,1}\delta^2(a_0 + \tC_{a,3} k \delta^2) + C_{e,1} \delta(2b_0 + \delta^2 k \tC_{b,3}) + C_{e,1} \delta(c_0 + k \delta C_{c,1}a_0 + (k\delta)^2 \delta C_{c,1}\tC_{a,3}) \\
    & \leq (1+(k+1)\delta C_{e,1}) e_0 + \delta(k \tC_{e,1} + C_{e,1}(2b_0 + c_0 + \tT C_{e,1} e_0 + \tT \tC_{e,1} + \tT  C_{c,1} a_0 + \delta \tT \tC_{b,3}))\\
    & \quad + \delta^2 (k \tC_{e,2} + C_{e,1}(k \delta \tC_{e,2} + a_0 + k \delta^2 \tC_{a,3} + (k\delta)^2 C_{c,1} \tC_{a,3})) \eqsp,
  \end{align}
which concludes the proof of the upper bound for $e_{k+1}$ using $k\delta \leq 1/(4 C_{e,1})$.
  
Using  the induction hypothesis, in particular the lower bounds on $c_k$ and the definition of $\tC_{b,1}$,  we get 
  \begin{align}
    a_{k+1} &\leq a_0- \delta(a_0k \tilde{C}_{a,1} + C_{a,2} c_0 - 2 C_{a,4} e_0 ) - \delta^2 a_0 C_{a,1} \\
    & \qquad+  \delta^2 (k \tC_{a,3} +2b_0 C_{a,3} + a_0 C_{a,1}\tT(\tC_{a,1} + \tT \tC_{a,2})) \\
            &\qquad  -  \delta^2 k [k \tilde{C}_{a,2} a_0 + a_0C_{c,1}C_{a,2}/2 -C_{a,5} \tC_{b,1} b_0 - \tC_{e,1} C_{a,4} \\
            & \qquad \qquad -\delta k C_{a,5} \tC_{b,2} b_0 - \delta^2 C_{a,3} \tC_{b,3} -\delta C_{a,4} \tC_{e,1} -\delta C_{a,5} \tC_{b,3}]\\ 
    &\leq a_0- \delta(a_0k \tilde{C}_{a,1} + C_{a,2} c_0 - 2 C_{a,4} e_0 - c_0 C_{b,4} C_{a,5} \delta k ) - \delta^2 a_0 C_{a,1} \\
    & \qquad + \delta^2 (k \tC_{a,3} +2b_0 C_{a,3} + a_0 C_{a,1}\tT(\tC_{a,1} + \tT \tC_{a,2})) \\
            &\qquad  -  \delta^2 k [k \tilde{C}_{a,2} a_0 + a_0C_{c,1}C_{a,2}/2 -C_{a,5} (C_{b,1} b_0+ C_{b,5} e_0) - \tC_{e,1} C_{a,4} \\
    & \qquad \qquad -\delta k C_{a,5} \tC_{b,2} b_0 - \delta^2 C_{a,3} \tC_{b,3} -\delta C_{a,4} \tC_{e,1} -\delta C_{a,5} \tC_{b,3}]
               \eqsp.
  \end{align}
  Using the condition \eqref{eq:condition_c0}, $\delta, k+1$ satisfy \eqref{eq:10}, we get that $a_{k+1} \leq a_0 - a_0 (k+1)\delta \tilde{C}_{a,1} - a_0 (k+2)k \delta^2 \tilde{C}_{a,2} - a_0 C_{a,1} \delta^2 +  \delta^2 k \tilde{C}_{a,3}  \leq  a_0 - a_0 (k+1)\delta \tilde{C}_{a,1} - a_0 (k+1)^2\delta^2 \tilde{C}_{a,2} +  \tilde{C}_{a,3}\delta^2 (k+1)$. The lower bound for $a_{k+1}$ proceeds similarly  using the upper bound on $c_k$ given by the induction hypothesis. Finally, the conditions $(k+1)\delta \leq 8^{-1}\{[a_0(C_{a,2}c_0)^{-1}]\wedge [C_{c,1}C_{a,2}]^{-1/2}\}$, $\delta \leq  [4 C_{a,1}]^{-1} \wedge [4 C_{c,1}C_{a,2} \tilde{C}_{a,3}]^{-1}$ and $(k+1)\delta \leq 1/2$ imply that $a_0/2 \leq a_{k+1} \leq a_0 + \delta^2 k \tC_{a,3}$.
  The last inequality for $c_{k+1}$ in \eqref{eq:14} uses the induction hypothesis  $ a_k\geq a_0 (1-k\delta (C_{a,2}c_0/a_0 + \delta C_{a,1}) -  (k\delta)^2 C_{a,2} C_{c,1} (1+ \delta \tilde{T} \tilde{C}_{a,3}) )$ and $c_{k+1}=c_k + C_{c,1} \delta a_k$. 
\end{proof}

\paragraph{Evolution over several Verlet steps.} Combining Lemmas~\ref{lem:drift_v1} and \ref{lem:rec_for_drift} leads to the following, which gives a control of the expected value of the Lyapunov function after a full trajectory of Verlet integrator.

\begin{lemma}
  \label{lem:drift_k_step}
Assume that \Cref{ass:for_drift} holds.
Let $\bfo = (\a_0,\b_0,\c_0,2b_0,\e_0,\f_0)\in \rset_+^6$,  $\bT,\bdelta >0$ and denote
    \begin{align}
      \bar{C}_{a,1} & = \mtt / (3\times 4) \eqsp, \quad \bar{C}_{a,2} = \mtt / (3 \times 2^3) \eqsp, \bC_{a,3} = 16 b_0 + a_0\bT(\mtt/2)^2(1+\bT) \eqsp, \\
      \bar{C}_{b,1} & = \epsilon + (c_0/b_0)(1+10\bdelta)  \eqsp, \quad \bC_{b,2,0} = 4[(1+\epsilon) \bC_{b,1} +(1+10\bdelta)(2a_0/b_0)  ] \eqsp, \\
      \bC_{b,3,0} &= 3\{31b_0 +2a_0 +\bT(2\bC_{a,3} + 31b_0 \bC_{b,1} +  2 \bar{C}_{a,3}\bT(1+10\bdelta)) \} \eqsp,\\
      \bC_{b,2} & = \bC_{b,2,0} + c_0 \mtt + a_0 \mtt \bT   \eqsp, \quad  \bC_{b,3} = \bC_{b,3,0} + \bar{C}_{a,3} \eqsp,\\
      \bC_{e,1} &= 14[2b_0 + c_0 + \bT(14 e_0 + 2 a_0 + \bdelta \bC_{b,3})] \eqsp, \quad \bC_{e,2} = 14(a_0+\bT \bC_{a,3}(\bdelta  + 2\bT)) \eqsp, \\
 \bar{C}_{f,1} &= a_0 \Mtt \bdelta + \bC_{a,3} \Mtt \bdelta \bT + 4 b_0 \mtt \Mtt \bdelta + \Mtt(c_0+\bC_{b,3,0}  \bar{T}^2 \bar{\delta}  + 2 \bar{T} a_0 + 2 \bC_{a,3} \bar{T}^2 \bar{\delta}) \eqsp.
  \end{align}

Let $\epsilon \leq \mtt/4$ and suppose that 
\begin{equation}
  \label{eq:condition_epsilon_e_0}
 \epsilon b_0 \leq a_0 \mtt/2^3  \eqsp.
\end{equation}
Then, for any $\delta>0$, $\eta \in \coint{0,1}$, and $K \in \nset$ satisfying
$\delta \leq \bdelta_1 \wedge \bdelta$ and $K\delta \leq \bT_1 \wedge \bT_2 \wedge \bar{T}$, and 
\begin{equation}
  \label{eq:condition_epsilon_e_0_1}
\delta \eta \bC_{a,3} \leq \frac{c_0 \mtt}{ 3\times 2^2} \eqsp, \quad   2\delta \frac{(1-\eta)}{\delta K}\bC_{a,3} \leq \frac{a_0  \mtt}{3 \times 2^3} \eqsp,
\end{equation}
where
\begin{equation}
    \label{eq:10_applied}
    \begin{aligned}
      \bdelta_1 &= 31^{-1} \wedge [\mtt/40] \wedge [ a_0 2^{-3}/( \bC_{b,3}(1+8\bdelta))] \wedge [1/(2\mtt)] \wedge [1/(8\mtt \bC_{a,3})] \eqsp, \\
    \bT_1 & = 4^{-1}\parentheseDeux{31^{-1} \wedge \bar{C}_{b,1}^{-1} \wedge \bar{C}_{b,2}^{-1/2}}\eqsp, \\
  \quad 
    \bT_2&  =4^{-1} \defEnsLigne{[2^{-3}\mtt/(1+10\bdelta)]\wedge[ 8^{-1} a_0 /(\bC_{b,2} b_0)]\wedge a_0\mtt(2 c_0)^{-1} \wedge 1/(7\mtt)^{1/2 }} \eqsp.
  \end{aligned}
  \end{equation}
 it holds
\begin{enumerate}[label=(\alph*),noitemsep,topsep=0pt,parsep=0pt,partopsep=0pt,leftmargin=*,wide]
\item   \label{lem:drift_k_step_a} For any $i \in \iint{1}{\nq}$, $x,v \in \rsetd$,  $      \rmV_{\delta}^K \Voi(x,v) \leq        \Vlyapc[\bfo_{K}][i](x,v)$,
    where ${\bfo}_K = ({\a}_K,{\b}_K,{\c}_K,2b_0,{\e}_K,{\f}_K)$
    \begin{equation}
    \begin{aligned}
      {\a}_K & = a_0 -K\delta \bar{C}_{a,1}c_0  - (K\delta)^2 \bar{C}_{a,2} a_0  \eqsp, \qquad
      {\b}_K = b_0 (1+K\delta\bar{C}_{b,1} + (K\delta)^2 \bar{C}_{b,2})  + \delta^2 K \bC_{b,3} \eqsp, \\
      {\c}_K &= c_0 + 2 K \delta a_0   \eqsp, \quad       {\e}_K  = \e_0(1 +17 K \delta) + \bC_{e,1} K\delta + \bC_{e,2} K\delta^2  \eqsp, \\
                   {\f}_K & = f_0 + K \delta \bar{C}_{f,1} \eqsp.
    \end{aligned}\label{eq:coeff_drift_after_rec_and_drift}
  \end{equation}
\item   \label{lem:drift_k_step_b}  In addition for any $k \in\{0,\ldots,K\}$, $i \in \iint{1}{\nq}$,
  $x,v \in \rsetd$,  $      \rmV_{\delta}^k \Voi(x,v) \leq        \Vlyapc[\bfo_{k}][i](x,v)$,
    where ${\bfo}_k = ({\a}_k,{\b}_k,{\c}_k,2b_0,{\e}_k,{\f}_k)$
    \begin{equation}
    \begin{aligned}
      {\a}_k & = a_0+2\bC_{a,3}\delta^2 k   \eqsp,   \quad 
      {\b}_k = 2b_0  + \delta^2 k \bC_{b,3} \eqsp, \\
      {\c}_k& = c_0 + 2 k \delta a_0   \eqsp, \quad       {\e}_k  = 2 \e_0 + \bC_{e,1} k\delta + \bC_{e,2} k\delta^2  \eqsp, \quad
                   {\f}_k  = f_0 + k \delta \bar{C}_{f,1} \eqsp.
    \end{aligned}\label{eq:coeff_drift_after_rec_and_drift_v2}
  \end{equation}
\end{enumerate}
\end{lemma}

\begin{proof}
  \begin{enumerate}[label=(\alph*),noitemsep,topsep=0pt,parsep=0pt,partopsep=0pt,leftmargin=*,wide]
  \item Let $i\in\{1,\ldots,\nq\}$,  $\delta \leq \bdelta_1$, $K \in\nset$, $K\delta \leq \bT \wedge \bT_1\wedge \bT_2$ and $x,v \in \rset^d$. Let $(a_0,b_0,c_0,e_0,f_0) \in \rset_+^6$.
  By \Cref{lem:drift_v1}, we have starting from $\tilde{\bfo}_0 = \bfo_0 = (a_0,b_0,c_0,2b_0,e_0,f_0)$, for any $k\in\nset$,
  \begin{equation}
    \label{eq:7_rec_drift_proof}
      \rmV_{\delta}^{k+1} \Vlyapc[\tbfo_{0}][i](x,v) \leq  \rmV_{\delta} \Vlyapc[\tbfo_{k}][i](x,v)     \leq  \Vlyapc[\tbfo_{k+1}][i](x,v)  \eqsp,
    \end{equation}
    where ${\tbfo}_{k+1} = (\tilde{\a}_{k+1},\tilde{\b}_{k+1},\tilde{\c}_{k+1},b_0,\tilde{\e}_{k+1},\tilde{\f}_{k+1})$ are defined by the recursion
        \begin{align}
      \tilde{\a}_{k+1} & = \tilde{a}_k(1-\mtt\delta^2/2) + 8 \tilde{b}_k \delta^2 - \delta \mtt \tilde{c}_k/2  + \delta \absLigne{\tilde{b}_k-b_0} \\
      \tilde{\b}_{k+1} &= 2 \delta^2 \tilde{\a}_k + (1+\epsilon\delta + 31 \delta^2)\tilde{\b}_k + \delta(1+10 \delta) \tilde{\c}_k  + \delta \absLigne{\tilde{b}_k-b_0}\\
      \tilde{\c}_{k+1} &= 2 \delta \tilde{\a}_k + \tilde{\c}_k \eqsp, \quad       \tilde{\e}_{k+1}  = \tilde{\e}_k + 7 \delta(\tilde{\a}_k\delta + \delta b_0 +\tilde{\b}_k+\tilde{\c}_k+\tilde{\e}_k)\\
                   \tilde{\f}_{k+1} & =  \tilde{\f}_k +\tilde{\a}_k \delta^2 \Mtt+ 2 \tilde{\b}_{k} \mtt \Mtt \delta^2  + \delta \Mtt \tilde{c}_k \eqsp.
    \end{align}

Applying \Cref{lem:rec_for_drift}, we get that for any $k \in \iint{0}{K}$,
        \begin{align}
      \tilde{\a}_{k} & \leq  a_0 (1-\mtt  k\delta c_0/(4 a_0) -\mtt (k\delta)^2/2^3) + (16 b_0 +  a_0(\mtt/2)^2 \bT(1+\bT)) \delta^2 k   \\
      \tilde{\b}_{k} & \leq  b_0 (1+k\delta\bar{C}_{b,1} + (k\delta)^2 \bar{C}_{b,2,0}) +\bC_{b,3,0} k\delta^2 \\
          \abs{\tilde{\c}_{k} - c_0 - 2k\delta a_0 } &\leq (k\delta)^2c_0 \mtt + a_0 \mtt (k \delta)^3 + 2 \bar{C}_{a,3} \delta^2 k\\
                 \tilde{\e}_{k}  &\leq \e_0(1 + 7 k \delta) + k\delta\bC_{e,1} + k\delta^2 \bC_{e,2}  \eqsp, \quad
                   \tilde{\f}_{k} \leq  f_0 + \delta k \bar{C}_{f,1}  \eqsp.
        \end{align}

        Using that $\abs{[c_k-c_0 - 2 \delta k a_0]\ps{v}{x}} \leq 2^{-1}\abs{c_k-c_0 - 2 \delta k a_0}[\norm{x}^2 + \norm{v}^2]$, \eqref{eq:condition_epsilon_e_0_1} and $K \delta \leq 3 \times 2^{-3} $ completes the proof.
        
      \item The proof of the second statement follows the same lines and is omitted. 
      \end{enumerate}
\end{proof}

\paragraph{Introducing momentum updates.} In the next lemma, we state pointwise estimates that control the effect on the Lyapunov function of a momentum update for a given   realization $g$ of a Gaussian variable. After raising these estimates to some powers, the expectation with respect to the Gaussian distribution will be taken afterwards (Lemma~\ref{lem:drift_one_step} below).

\begin{lemma}
  \label{lem:drift_for_fixed_Gaussian}
Assume that \Cref{ass:for_drift} holds and that $\epsilon \leq \mtt/4$ and consider the notations introduced in \Cref{lem:drift_k_step}.
 Then, for any $i \in \iint{1}{\nq}$, $\bfomega = (\delta,K,\eta)$, $\delta >0$, $K \in \nset$, $\eta \in \ooint{0,1}$, $\bfo = (\a_0,\b_0,\c_0,2b_0,\e_0,\f_0)\in \rset_+^6$,
  satisfying
$\delta \leq \bdelta_1$,  $K\delta \leq \bT \wedge \bar{T}_1 \wedge \bT_2 $ with $\bT >0$, \eqref{eq:condition_epsilon_e_0}-\eqref{eq:condition_epsilon_e_0_1},
\begin{equation}
     \label{eq:condition_c_0}
  \begin{aligned}
    &    c_0(\eta-1)+2 \eta a_0 \delta K = 0 \eqsp, \quad c_0 \leq  \sqrt{2a_0b_0} \eqsp, \quad c_0/b_0 \leq \eta(1-\eta)/(\delta K) \eqsp, \\
    & \bC_{b,1} -c_0/b_0 \leq (1-\eta)/[4\delta K ] \eqsp, \quad \delta K  \bC_{b,2} + \delta \bC_{b,3}/b_0 \leq (1-\eta)/[4\delta K ] \eqsp,
  \end{aligned}
\end{equation}
it holds
\begin{enumerate}[label=(\alph*),noitemsep,topsep=0pt,parsep=0pt,partopsep=0pt,leftmargin=*,wide]
\item   \label{lem:drift_for_fixed_Gaussian_a} For any $x=(x_1,\ldots,x_{\nq}),v=(v_1,\ldots,v_{\nq}),g=(g_1,\ldots,g_{\nq}) \in \rsetd$,
  \begin{multline}
    \label{eq:7}
    \rmV_{\delta}^K \Voi(x,\eta v + \teta g ) \leq  (1- \bar{\rho}_{\bfomega}K\delta)[a_0 \norm{x_i}^2 + b_0 \norm{v_i}^2 +c_0\ps{x_i}{v_i} + 2 b_0 \Uq(x_i)] +f_K \\
    +2 \eta \teta b_K\ps{v_i}{g_i} + c_K\teta\ps{x}{g}  + \teta^2 \b_K \norm{g_i}^2  + e_K \frac{\epsilon}{\nq}\sum_{j=1}^{\nq}[\norm{x_i-x_j}^2 + \norm{v_i-v_j+g_i-g_j}^2]  \eqsp,
    \end{multline}
    where  $b_K \leq 2b_0 + \delta^2K \bC_{b,3} $, $e_K \leq 2 e_0 + (K\delta) \bC_{e,1} + K\delta^2  \bC_{e,2}$,
    \begin{align}
      \label{eq:18}
      \teta^2 &= 1- \eta^2 \\
      \bar{\rho}_{\bfomega} &= [(c_0 + K\delta a_0) \bar{C}_{a,2} /(4a_0)] \wedge [(1-\eta)/(4K\delta)] \wedge [ (c_0 + K\delta a_0) \bar{C}_{a,2}/(4b_0a_0)]  \eqsp.
    \end{align}
  \item     \label{lem:drift_for_fixed_Gaussian_b}  In addition for any $k \in \{1,\ldots,K\}$,  $x=(x_1,\ldots,x_{\nq}),v=(v_1,\ldots,v_{\nq}),g=(g_1,\ldots,g_{\nq}) \in \rsetd$,
  \begin{multline}
    \label{eq:7}
    \rmV_{\delta}^k \Voi(x,\eta v + \teta g ) \leq  \bC_1 [a_0 \norm{x_i}^2 + b_0 \norm{v_i}^2 +c_0\ps{x_i}{v_i} + 2 b_0 \Uq(x_i)] +f_K \\
    +2 \eta \teta b_K\ps{v_i}{g_i} + c_K\teta\ps{x}{g}  + \teta^2 \b_K \norm{g_i}^2  + e_K \frac{\epsilon}{\nq}\sum_{j=1}^{\nq}[\norm{x_i-x_j}^2 + \norm{v_i-v_j+g_i-g_j}^2]  \eqsp,
  \end{multline}
  where
  \begin{equation}
    \label{eq:12}
    \bC_1 =    [\{a_0 + 2 c_0 + \bT a_0 +  \bT^2c_0 \mtt/8 + a_0 \mtt \bT^3/4 +2\bC_{a,3}\bdelta \bT\}/a_0] \vee [\{2b_0 + 2c_0 +\bT a_0 + \bdelta \bT\bC_{b,3}\}/b_0]    \eqsp.
  \end{equation}
  \end{enumerate}
\end{lemma}

\begin{remark}
  \label{rem:choice_a_b_c}
  Note that the first line in  \eqref{eq:condition_c_0} is in particular satisfied choosing $b_0=1$, $a_0 = (1-\eta)^2/[2(\delta K)^2]$  and $c_0=  \eta  (1-\eta)/(\delta K)$. 
\end{remark}
\begin{proof}
  \begin{enumerate}[label=(\alph*),noitemsep,topsep=0pt,parsep=0pt,partopsep=0pt,leftmargin=*,wide]
  \item Let $\delta >0$, $K \in \nset$, $\eta \in \ooint{0,1}$, $\bfo = (\a_0,\b_0,\c_0,2b_0,\e_0,\f_0)\in \rset_+^6$,
  satisfying
$\delta \leq \bdelta_1 \wedge \bdelta$,  $K\delta \leq \bT \wedge \bar{T}_1 \wedge \bT_2 $, \eqref{eq:condition_epsilon_e_0}-\eqref{eq:condition_epsilon_e_0_1}-\eqref{eq:condition_c_0}. For any $x=(x_1,\ldots,x_{\nq}),v=(v_1,\ldots,v_{\nq}),g=(g_1,\ldots,g_{\nq}) \in \rsetd$,  using \Cref{lem:drift_k_step}-\ref{lem:drift_k_step_a}, we get 
  \begin{align}
    \label{eq:7_0}
&    \rmV_{\delta}^K \Voi(x,\eta v + \teta g ) \leq a_K \norm{x_i}^2 + b_K \eta^2 \norm{v_i}^2 + c_K\eta \ps{x_i}{v_i} + 2 b_0  \Uq(x_i) +f_K\\
    & +2 \eta\teta b_K\ps{v_i}{g_i} + \teta^2 b_K\norm{g_i}^2 + c_K\teta\ps{x_i}{g_i}     + \frac{\epsilon e_K}{\nq} \sum_{j=1}^{\nq} [\norm{x_i-x_j}^2+ \norm{v_i-v_j+g_i-g_j}^2]\eqsp.
  \end{align}

  Note that by the first line in \eqref{eq:condition_c_0} and using $c_K = c_0 +2 K\delta a_0$,

  \begin{equation}
    \label{eq:08_2024_1}
    \eta c_K \ps{x}{v} = c_0\ps{x}{v}\eqsp.
  \end{equation}
  Second using the second line in \eqref{eq:condition_c_0}, we have
  \begin{align}
    \label{eq:1}
    b_K\eta^2 &= (1+\bar{C}_{b,1} K\delta + \bar{C}_{b,2}(K\delta)^2 + \bC_{b,3} K\delta^2/b_0) b_0 (1-(1+\eta)(1-\eta)) \\
              & \leq  b_0[1- (1-\eta) -\eta(1-\eta) + C_1 \delta K + \delta K (\delta K\bC_{2,b} + \delta \bC_{3,b}/b_0)] \\
              & \leq  b_0[1- (1-\eta) -\eta(1-\eta) + c_0 \delta K/b_0 + (1-\eta)/4 + \delta K (\bC_{2,b} + \delta \bC_{3,b}/b_0)]\\
        \label{eq:08_2024_2}
    & \leq   b_0 (1- \delta K(1-\eta)/[2\delta K]) \eqsp.
  \end{align}

  Then using \eqref{eq:08_2024_1}-\eqref{eq:08_2024_2}, $U(x) \leq \norm{x}^2$ under \Cref{ass:for_drift} and since $U(0)=0$, we get setting $\bC_{a,4} = K\delta \bC_{a,2}(c_0 + \delta K a_0)/a_0$,
  \begin{align}
    \label{eq:7_0}
    &\rmV_{\delta}^K \Voi(x,\eta v + \teta g ) \leq (1-\bar{C}_{a,4}/2)a_0 \norm{x_i}^2 + \delta K(1-\eta)\norm{v_i}^2/[2 K \delta] + c_0 \ps{x_i}{v_i}\\
    &\qquad + (1-\bar{C}_{a,4}/(4  b_0))2 b_0  \Uq(x_i) +f_K  +  2 \eta\teta b_K\ps{v_i}{g_i} + \teta^2 b_K\norm{g_i}^2 + c_K\teta\ps{x_i}{g_i}  \\
    &\qquad \txts  +  \frac{\epsilon e_K}{\nq} \sum_{j=1}^{\nq} [\norm{x_i-x_j}^2+\norm{v_i-v_j+g_i-g_j}^2]\\
&    \leq (1-\bar{C}_{a,4}/2)a_0 \norm{x_i}^2 + \delta K(1-\eta)\norm{v_i}^2/[2 K \delta] + c_0(1- \bar{\rho}_{\bfomega}K\delta) \ps{x_i}{v_i} +  \bar{\rho}_{\bfomega}K\delta c_0\ps{x}{v}\\
    & \qquad + (1-\bar{C}_{a,4}/(4  b_0))2 b_0  \Uq(x_i) +f_K  +  2 \eta\teta b_K\ps{v_i}{g_i} + \teta^2 b_K\norm{g_i}^2 + c_K\teta\ps{x_i}{g_i}  \\
    &  \qquad \txts +  \frac{\epsilon e_K}{\nq} \sum_{j=1}^{\nq} [\norm{x_i-x_j}^2+ \norm{v_i-v_j+g_i-g_j}^2]\eqsp,
  \end{align}
  which completes the proof using the Cauchy-Schwarz inequality and the condition $c_0 \leq 2 \sqrt{a_0b_0}$.
\item The proof of the second statement follows the same lines using \Cref{lem:drift_k_step}-\ref{lem:drift_k_step_b} instead of \Cref{lem:drift_k_step}-\ref{lem:drift_k_step_a}.
\end{enumerate}
\end{proof}

\paragraph{Averaging over momentum updates.}

We first give a technical result on certain Gaussian moments.
\begin{lemma}
  \label{lem:moment_gaussian}
  Let $a\in\rset_+$, $\ell \in\nsets$ and $G$ be a $d$-dimensional zero-mean Gaussian random variable with covariance matrix identity. Then, for any $x \in\rset^d$,
  \begin{equation}
    \label{eq:def_tbfm}
    \PE[(a+\ps{G}{x})^\ell] \leq a^{\ell-2\floor{\ell/2}}  \parenthese{a^2+\tilde{\bfm}_\ell \norm{x}^2}^{\floor{\ell/2}} \eqsp,
\text{ where }
\tbfm_\ell =     \rme^{2}(\ell/2)^2 \floor{\ell/2}^{-1}  \bfm_{2\floor{\ell/2}}^{1/\floor{\ell/2}} \eqsp, 
  \end{equation}
and $\bfm_{\ell'}$ is the $\ell'$-th moment of the zero-mean one-dimensional Gaussian distribution with variance $1$. 
\end{lemma}

\begin{proof}
  Expanding $(a+\ps{G}{x})^\ell$ and using that $\ps{G}{x}$ is a zero-mean Gaussian random variable with variance $\norm{x}^2$, we get
  \begin{align}
    \PE[(a+\ps{G}{x})^\ell]  & = \sum_{k=0}^{\ell} \binom{\ell}{k} a^{\ell-k} \PE[\ps{G}{x}^k] = \sum_{k=0}^{\floor{\ell/2}} \binom{\ell}{2k} \norm{x}^{2k} a^{\ell-2k} \bfm_{2k}^{2k/2k}\\
                          & \leq a^{\ell-2\floor{\ell/2}}  \sum_{k=0}^{\floor{\ell/2}} \binom{\ell}{2k} \norm{x}^{2k} a^{2(\floor{\ell/2}-k)} \bfm_{2\floor{\ell/2}}^{2k/2\floor{\ell/2}} \eqsp,
  \end{align}
  where we apply Jensen inequality for the last step. 
  Using that $\binom{\ell}{2k} \leq (\rme \ell/(2k))^{2k} \leq (\rme^2 \ell/2)^k(\ell/(2k))^k \leq (\rme^2 (\ell/2)^2 \floor{\ell/2}^{-1})^{k}\binom{\floor{\ell/2}}{k}$, we get 
  \begin{align}
    \PE[(a+\ps{G}{x})^\ell]  & \leq a^{\ell-2\floor{\ell/2}}  \sum_{k=0}^{\floor{\ell/2}} (\rme^2 (\ell/2)^2 \floor{\ell/2}^{-1})^{k}\binom{\floor{\ell/2}}{k} \norm{x}^{2k} a^{2(\floor{\ell/2}-k)} \bfm_{2\floor{\ell/2}}^{2k/2\floor{\ell/2}}\\
    & \leq a^{\ell-2\floor{\ell/2}}  \parenthese{a^2+\rme^{2}(\ell/2)^2 \floor{\ell/2}^{-1} \norm{x}^2 \bfm_{2\floor{\ell/2}}^{1/\floor{\ell/2}}}^{\floor{\ell/2}} \eqsp,
  \end{align}
  which completes the proof. 
\end{proof}

Define for any $x=(x_1,\ldots,x_{\nq}),v=(v_1,\ldots,v_{\nq})\in \rsetd$ and $\ell \in \nsets$,
 \begin{equation}
    \label{eq:def_Vol}
    \Vlyol(x,v) = \sum_{i=1}^{\nq}  \Voi^{\ell}(x,v)\eqsp.
 \end{equation}
We are now ready to control the evolution of the Lyapunov function $\Voi^\ell$ for any $\ell \in \nsets$  over one full transition of the gHMC chain (notice that we are essentially interested in $\ell=3$, since this is the dominant order in \eqref{eq:Lyap_for_main_drift_thm}).

\begin{lemma}
  \label{lem:drift_one_step}
Assume that \Cref{ass:for_drift} holds and that $\epsilon \leq \mtt/4$.
 Then, for any $\bfomega = (K,\delta,\eta)$, $K \in\nset$, $\delta >0$,  $\eta \in \ooint{0,1}$, $\bfo = (\a_0,\b_0,\c_0,2b_0,\e_0,\f_0)\in \rset_+^6$,
  satisfying
$\delta \leq \bdelta_1\wedge \bdelta$,  $K\delta \leq \bT \wedge \bar{T}_1 \wedge \bT_2 $ with $\bT,\bdelta >0$, \eqref{eq:condition_epsilon_e_0}-\eqref{eq:condition_epsilon_e_0_1}-\eqref{eq:condition_c_0},
it holds:
\begin{enumerate}[label=(\alph*),noitemsep,topsep=0pt,parsep=0pt,partopsep=0pt,leftmargin=*,wide]
\item  \label{lem:drift_one_step_a}
For any $x=(x_1,\ldots,x_{\nq}),v=(v_1,\ldots,v_{\nq})\in \rsetd$ and $\ell \in \nsets$,
\begin{align}
  \label{eq:16}
  &\rmP_{\bfomega} \Voi^\ell(x,v) \leq (1-\bar{\rho}_{\bfomega}K\delta/4)^\ell[a_0 \norm{x_i}^2 + b_0\norm{v_i}^2 + c_0 \ps{x_i}{v_i} + 2b_0 \Uq(x_i)]^{\ell}\\
  & \qquad + 2^\ell \floor{\ell/2} \bE_{\ell}[(a_0 \bD_2^x + b_0 (\bD_2^{v}+2\bD_2^x) + c_0 (\bD_2^x)^{1/2}(\bD_2^v)^{1/2})^2 + \bE_{\ell}]^{\floor{\ell/2}-1}  \\
  & \qquad+ 12^{\ell-1}(1+4/(\bar{\rho}_{\bfomega}K\delta))^{\ell-1}[\teta^{2\ell} b_K^\ell q^\ell  \bfm_{2\ell} + 8^{\ell-1} \epsilon^\ell e_K^{\ell}q^\ell \bfm_{2\ell} +  f_K^\ell]\\
    &\qquad \txts  +12^{\ell-1}(1+4/(\bar{\rho}_{\bfomega}K\delta))^{\ell-1} \epsilon^\ell e_K^\ell\parentheseDeux{\nq^{-1} \sum_{j=1}^{\nq}\norm{x_i-x_j}^2 + \norm{v_i-v_j}^2}^\ell \eqsp,
\end{align}
where
\begin{align}
  \label{eq:22}
  \bD_1^x &=  \frac{2^{3/2} (c_K \teta)^2 \tbfm_{\ell}}{a_0^2\bar{\rho}_{\bfomega}K\delta(\sqrt{2}-1)} \eqsp, \quad  \bD_1^{v}  = \frac{8\sqrt{2}  (\teta\eta b_K )^2 \tbfm_{\ell} }{\bar{\rho}_{\bfomega}K\delta(\sqrt{2}-1)}\eqsp, \\
 \bD_2^x &= [2 \bD_1^x] \vee 4 \sqrt{\frac{\bD_1^v \tbfm_{\ell}(\eta\teta b_K)^2 \sqrt{2}}{\bar{\rho}_{\bfomega} K \delta a_0^2 (\sqrt{2}-1)}} \eqsp , \quad   \bD_2^v = [2 \bD_1^v] \vee 2 \sqrt{\frac{\bD_1^x \tbfm_{\ell}(\teta c_K)^2 \sqrt{2}}{\bar{\rho}_{\bfomega} K \delta b_0^2 (\sqrt{2}-1)}}\\
    \bE_{\ell} & = [4   (\teta\eta b_K)^2 \bD_2^{v} + 2 (c_K \teta)^2\bD_2^x]\tbfm_{\ell}\eqsp,
\end{align}
where $\tbfm_{\ell}$ is defined in \eqref{eq:def_tbfm} and $\teta^2= 1- \eta^2$.
\item In addition, if for $\ell \in \nsets$
  \begin{equation}
    \label{eq:verif_epsi_momentn}
    \epsilon \leq \ell^{1/(\ell-1)}( 1-\bar{\rho}_{\bfomega}K\delta/4) \eqsp, \quad 
24 e_K    \epsilon (1+4/(\bar{\rho}_{\bfomega}K\delta))^{(\ell-1)/\ell} \leq \bar{\rho}_{\bfomega}K\delta/8  \eqsp,
  \end{equation}
for   any $x=(x_1,\ldots,x_{\nq}),v=(v_1,\ldots,v_{\nq})\in \rsetd$,
\begin{align}
  \label{eq:16}
  \rmP_{\bfomega} \Vlyol(x,v) &\txts \leq (1-\bar{\rho}_{\bfomega}K\delta/8)^\ell \sum_{i=1}^{\nq}[a_0 \norm{x_i}^2 + b_0\norm{v_i}^2 + c_0 \ps{x_i}{v_i} + 2b_0 \Uq(x_i)]^{\ell}\\
  &\qquad + 4^\ell \floor{\ell/2} \bE_{\ell}[(a_0 \bD_1^x + b_0 (\bD_1^{v}+2\bD_1^x) + c_0 (\bD_1^{x})^{1/2}(\bD_1^v)^{1/2})^2 + \bE_{\ell}]^{\floor{\ell/2}-1}  \\
  &\qquad + 12^{\ell-1}\nq (1+4/(\bar{\rho}_{\bfomega}K\delta))^{\ell-1}[\teta^{2\ell} b_K^\ell q^\ell  \bfm_{2\ell} + 8^{\ell-1} \epsilon^\ell e_K^{\ell}q^\ell \bfm_{2\ell} +  f_K^\ell] \eqsp.
\end{align}
\item Finally, for any $k \in\{0,\ldots,K\}$,
  $x=(x_1,\ldots,x_{\nq}),v=(v_1,\ldots,v_{\nq})\in \rsetd$ and $\ell \in \nsets$,
\begin{align}
  \label{eq:16}
    \rmD_{\eta} \rmV_{\delta}^k \Vlyol(x,v) & \txts \leq \bC_{2}^{\ell} \sum_{i=1}^{\nq} [a_0 \norm{x_i}^2 + b_0\norm{v_i}^2 + c_0 \ps{x_i}{v_i} + 2b_0 \Uq(x_i)]^{\ell} \\
  & \qquad+ 12^{\ell-1}\nq(1+4/(\bar{\rho}_{\bfomega}K\delta))^{\ell-1}[\teta^{2\ell} b_K^\ell q^\ell  \bfm_{2\ell} + 8^{\ell-1} \epsilon^\ell e_K^{\ell}q^\ell \bfm_{2\ell} +  f_K^\ell] \eqsp,
\end{align}
where
\begin{equation}
  \label{eq:bC2}
  \bC_{2} = 8 (1+\bar{\rho}_{\bfomega}K\delta/4) [\bC_{1}+ ( [4 (\eta\teta b_K)^2 b_0^{-1} ] \vee [(c_K \teta)^2 a_0^{-1}])] + 24(1+4/(\bar{\rho}_{\bfomega}K\delta))^{(\ell-1)/\ell} \epsilon e_K \eqsp. 
\end{equation}
\end{enumerate}
\end{lemma}

\begin{proof}
  Let $\delta \leq \bdelta_1\wedge \bdelta$,  $K\delta \leq \bT \wedge \bar{T}_1 \wedge \bT_2 $ with $\bT,\bdelta >0$, satisfying  \eqref{eq:condition_epsilon_e_0}-\eqref{eq:condition_epsilon_e_0_1}-\eqref{eq:condition_c_0}.
\begin{enumerate}[label=(\alph*),noitemsep,topsep=0pt,parsep=0pt,partopsep=0pt,leftmargin=*,wide]
\item Consider  $x,v \in \rsetd$ and $\ell \in \nsets$.
  Using \Cref{lem:drift_for_fixed_Gaussian}-\ref{lem:drift_for_fixed_Gaussian_a} and Jensen inequality, we have,  for any $g=(g_1,\ldots,g_{\nq}) \in \rset^d$,
  \begin{align}
    \label{eq:3}
    & \rmV_{\delta}^K \Voi^{\ell}(x,\eta v + \teta g )\\
    & \leq (1+\bar{\rho}_{\bfomega}K\delta/4)^{\ell-1}[A +2\eta\teta b_K\ps{v_i}{g_i} + c_K \teta\ps{x_i}{g_i}]^{\ell}\\
  & \txts  \qquad + 12^{\ell-1}(1+4/\bar{\rho}_{\bfomega}K\delta)^{\ell-1}[\teta^{2\ell} b_K^{\ell} \norm{g}^{2\ell} +  \epsilon^{\ell} e_K^{\ell}\nq^{-1}\sum_{j=1}^{\nq}  \norm{g_i-g_j}^{2\ell} +f_K^{\ell}]\\
    & \txts  \qquad \txts+12^{\ell-1}(1+4/\bar{\rho}_{\bfomega}K\delta)^{\ell-1} \epsilon^{\ell} e_K^{\ell}\parentheseDeux{\nq^{-1} \sum_{j=1}^{\nq}[\norm{x_i-x_j}^2 + \norm{v_i-v_j}^2}^{\ell} \eqsp,
  \end{align}
  where we have set $A=(1-\bar{\rho}_{\bfomega}K\delta)[a_0 \norm{x_i}^2 + b_0 \norm{v_i}^2 + c_0\ps{x_i}{v_i} + 2b_0 \Uq(x_i)]$. Then, taking expectation and \Cref{lem:moment_gaussian} imply that
\begin{align}
  \label{eq:3_1}
    & \rmP_{\bfomega} \Voi^{\ell}(x,v) 
     \leq (1+\bar{\rho}_{\bfomega}K\delta/4)^{\ell-1}A^{\ell-2\floor{\ell/2}}B^{\floor{\ell/2}}\\
  & \txts  \qquad + 12^{\ell-1}(1+4/\bar{\rho}_{\bfomega}K\delta)^{\ell-1}[\teta^{2\ell} b_K^{\ell} q^{\ell} \bfm_{2\ell} + 8^{\ell-1} \epsilon^{\ell} e_K^{\ell}q^{\ell} \bfm_{2\ell}+f_K^{\ell}]\\
    & \txts  \qquad +12^{\ell-1}(1+4/\bar{\rho}_{\bfomega}K\delta)^{\ell-1} \epsilon^{\ell} e_K^{\ell}\nq^{-1} \sum_{j=1}^{\nq}[\norm{x_i-x_j}^2 + \norm{v_i-v_j}^2]^{\ell} \eqsp.
\end{align}
with $B=  A^2 +[4(\eta\teta b_K)^2 \norm{v_i}^2  + (c_K\teta)^2 \norm{x_i}^2]\tbfm_{\ell}$.
For $\norm{x_i}^2  \geq \bD_1^x$  and $\norm{v_i}^2 \geq \bD_1^{v}$, we get using  $c_0\ps{x_i}{v_i} \leq (a_0\norm{x_i}^2 + b_0\norm{v_i}^2)/\sqrt{2}$ since $c_0 \leq \sqrt{2a_0b_0}$ by the first line in \eqref{eq:condition_c_0},
\begin{align}
  \label{eq:6}
  B & \leq (1-\bar{\rho}_{\bfomega}K\delta/2)[a_0 \norm{x_i}^2 + b_0 \norm{v_i}^2 + c_0\ps{x_i}{v_i} + 2b_0 \Uq(x_i)]^2 \\
&   \qquad \qquad-\bar{\rho}_{\bfomega}K\delta(\sqrt{2}-1)2^{-1/2}[a_0 \norm{x_i}^2 + b_0 \norm{v_i}^2 ]^2 + [4(\eta\teta b_K)^2 \norm{v_i}^2  + (c_K\teta)^2 \norm{x_i}^2]\tbfm_{\ell} \\
  & \leq (1-\bar{\rho}_{\bfomega}K\delta/2)[a_0 \norm{x_i}^2 + b_0 \norm{v_i}^2 + c_0\ps{x_i}{v_i} + 2b_0 \Uq(x_i)]^2\eqsp.
\end{align}

Similarly, for $\norm{x_i}^2  \geq \bD_2^x$  and $\norm{v_i}^2 \leq \bD_1^{v}$ or $\norm{x_i}^2  \leq \bD_1^x$  and $\norm{v_i}^2 \geq \bD_2^{v}$, we get
\begin{align}
  \label{eq:6_0}
  B 
  & \leq (1-\bar{\rho}_{\bfomega}K\delta/2)[a_0 \norm{x_i}^2 + b_0 \norm{v_i}^2 + c_0\ps{x_i}{v_i} + 2b_0 \Uq(x_i)]^2\eqsp.
\end{align}

It yields for $\norm{x_1}^2 \geq \bD_2^x$ or $\norm{v_i}^2 \geq \bD_2^v$,
\begin{align}
  \label{eq:3_2_final}
 & (1+\bar{\rho}_{\bfomega}K\delta/4)^{\ell-1}A^{\ell-2\floor{\ell/2}}B^{\floor{\ell/2}}\\
 & \leq (1+\bar{\rho}_{\bfomega}K\delta/4)^{\ell}A^{\ell-2\floor{\ell/2}}((1-\bar{\rho}_{\bfomega}K\delta/2)[a_0 \norm{x_i}^2 + b_0 \norm{v_i}^2 + c_0\ps{x_i}{v_i} + 2b_0 \Uq(x_i)]^2)^{\floor{\ell/2}}\\
  & \leq (1-\bar{\rho}_{\bfomega}K\delta/4)^{\ell}[a_0 \norm{x_i}^2 + b_0 \norm{v_i}^2 + c_0\ps{x_i}{v_i} + 2b_0 \Uq(x_i)]^{\ell} \eqsp.
\end{align}
Therefore, the proof is completed for $\norm{x_1}^2 \geq \bD_2^x$ or $\norm{v_i}^2 \geq \bD_2^v$. It remains to consider the case where $\norm{x_1}^2 \leq \bD_2^x$ and  $\norm{v_i}^2 \leq \bD_2^v$.

Using that $(t+s)^{\ell'} -t^{\ell'}\leq \ell' s(t+s)^{\ell'-1}$, for $s,t \geq 0$ and $\ell' \geq 1$, we get for $\norm{x_1}^2 \leq \bD_2^x$ and  $\norm{v_i}^2 \leq \bD_2^v$,
\begin{align}
  \label{eq:15}
  B^{\floor{\ell/2}} - A^{2\floor{\ell/2}}  & \leq (A^2 + \bE_{\ell})^{\floor{\ell/2}} - A^{2\floor{\ell/2}} \leq \floor{\ell/2}\bE_{\ell}B^{\floor{\ell/2}-1} \eqsp.
\end{align}
Combining this inequality and \eqref{eq:3_1}, we obtain
\begin{align}
  \label{eq:3_1}
    & \rmP_{\bfomega} \Voi^{\ell}(x,v) 
     \leq (1+\bar{\rho}_{\bfomega}K\delta/4)^{\ell-1}A^{\ell} + (1+\bar{\rho}_{\bfomega}K\delta/4)^{\ell-1} \floor{\ell/2}\bE_{\ell}B^{\floor{\ell/2}} \\
  & \txts  \qquad + 12^{\ell-1}(1+4/\bar{\rho}_{\bfomega}K\delta)^{\ell-1}[\teta^{2\ell} b_K^{\ell} q^{\ell} \bfm_{2\ell} + 8^{\ell-1} \epsilon^{\ell} e_K^{\ell}q^{\ell} \bfm_{2\ell}+f_K^{\ell}]\\
    & \txts  \qquad +12^{\ell-1}(1+4/\bar{\rho}_{\bfomega}K\delta)^{\ell-1} \epsilon^{\ell} e_K^{\ell}\nq^{-1} \sum_{j=1}^{\nq}[\norm{x_i-x_j}^2 + \norm{v_i-v_j}^2]^{\ell} \eqsp,
\end{align}
and the proof is completed using the definition of $A$, $(1+\bar{\rho}_{\bfomega}K\delta/4)^{\ell-1}(1-\bar{\rho}_{\bfomega}K\delta)^{\ell} \leq (1-\bar{\rho}_{\bfomega}K\delta/2)^{\ell}$, $B \leq (a_0 \bD_2^x + b_0 (\bD_2^{v}+2\bD_2^x) + c_0 (\bD_2^x)^{1/2}(\bD_2^v)^{1/2})^2 + \bE_{\ell}$ since $U(x) \leq \norm{x}^2$ under \Cref{ass:for_drift} and $(1+\bar{\rho}_{\bfomega}K\delta/4)^{\ell-1} \leq 2^{\ell}$.

\item By Young's inequality, we have
\begin{align}
  \rmP_{\bfomega} \Vlyol(x,v)& \txts \leq (1-\bar{\rho}_{\bfomega}K\delta/4)^\ell \sum_{i=1}^{\nq}[a_0 \norm{x_i}^2 + b_0\norm{v_i}^2 + c_0 \ps{x_i}{v_i} + 2b_0 \Uq(x_i)]^{\ell}\\
  & \qquad + 2^{\ell} \floor{\ell/2} \nq \bE_{\ell}[(a_0 \bD_2^x + b_0 (\bD_2^{v}+2\bD_2^x) + c_0 (\bD_2^x)^{1/2}(\bD_2^v)^{1/2})^\ell + \bE_{\ell}^\ell]  \\
                             & \qquad+ 12^{\ell-1}(1+4/(\bar{\rho}_{\bfomega}K\delta))^{\ell-1} \nq [\teta^{2\ell} b_K^\ell q^\ell  \bfm_{2\ell} + 8^{\ell-1} \epsilon^\ell e_K^{\ell}q^\ell \bfm_{2\ell} +  f_K^\ell]\\
    \label{eq:16}
    &\txts \qquad   +2 \times 24^{\ell-1}(1+4/(\bar{\rho}_{\bfomega}K\delta))^{\ell-1} \epsilon^\ell e_K^\ell \sum_{i=1}^{\nq}[\norm{x_i}^2 + \norm{v_i}^2 ]^{\ell} \eqsp.
\end{align}
Then the proof is completed using
\begin{equation}
  \label{eq:19}
a_0  \norm{x_i}^2 + b_0\norm{v_i}^2 \leq \sqrt{2}/(\sqrt{2}-1)\{a_0 \norm{x_i}^2 + b_0 \norm{v_i}^2+ c_0 \ps{x_i}{v_i}\} \leq 4 \{a_0 \norm{x_i}^2 + b_0 \norm{v_i}^2 + c_0 \ps{x_i}{v_i}\}  \eqsp,
\end{equation}
the fact that  by \eqref{eq:verif_epsi_momentn}
\begin{align}
&  (1-\bar{\rho}_{\bfomega}K\delta/4)^\ell  + 2 \times 4 \times 24^{\ell-1}[a_0^{-1}\vee b_0^{-1}]^{\ell}(1+4/(\bar{\rho}_{\bfomega}K\delta))^{\ell-1} \epsilon^\ell e_K^\ell \\
  & \qquad  \leq (1-\bar{\rho}_{\bfomega}K\delta/4)^\ell + \ell(1-\bar{\rho}_{\bfomega}K\delta/4)^{\ell-1}[24e_K \epsilon [a_0^{-1}\vee b_0^{-1}] (1+4/(\bar{\rho}_{\bfomega}K\delta))^{(\ell-1)/\ell}]\\
  &\qquad  \leq (1-\bar{\rho}_{\bfomega}K\delta/4+24e_K \epsilon [a_0^{-1}\vee b_0^{-1}] (1+4/(\bar{\rho}_{\bfomega}K\delta))^{(\ell-1)/\ell})^\ell \eqsp,
\end{align}
and \eqref{eq:verif_epsi_momentn} again.
\item Following the same lines as in the proof \ref{lem:drift_one_step_a}  using \Cref{lem:drift_for_fixed_Gaussian}-\ref{lem:drift_for_fixed_Gaussian_b} instead of \Cref{lem:drift_for_fixed_Gaussian}-\ref{lem:drift_for_fixed_Gaussian_a} we have using \eqref{eq:19}
  \begin{align}
  \label{eq:16_0}
   \rmD_{\eta} \rmV_{\delta}^k \Voi^\ell(x,v) & \leq \bC_{2,0}^{\ell} [a_0 \norm{x_i}^2 + b_0\norm{v_i}^2 + c_0 \ps{x_i}{v_i} + 2b_0 \Uq(x_i)]^{\ell} \\
  & \qquad + 2^{\ell} \floor{\ell/2} \bE_{\ell}[(a_0 \bD_2^x + b_0 (\bD_2^{v}+2\bD_2^x) + c_0 (\bD_2^x)^{1/2}(\bD_1^v)^{1/2})^\ell + \bE_{\ell}^\ell]  \\
  & \qquad+ 12^{\ell-1}(1+4/(\bar{\rho}_{\bfomega}K\delta))^{\ell-1}[\teta^{2\ell} b_K^\ell q^\ell  \bfm_{2\ell} + 8^{\ell-1} \epsilon^\ell e_K^{\ell}q^\ell \bfm_{2\ell} +  f_K^\ell]\\
    & \qquad \txts  + 12^{\ell-1}(1+4/(\bar{\rho}_{\bfomega}K\delta))^{\ell-1} \epsilon^\ell e_K^\ell\nq^{-1} \sum_{j=1}^{\nq}\parentheseDeux{\norm{x_i-x_j}^2 +  \norm{v_i-v_j}^2}^\ell \eqsp,
  \end{align}
  with
\begin{equation}
  \label{eq:4_0}
  \bC_{2,0} = 8 (1+\bar{\rho}_{\bfomega}K\delta/4) [\bC_{1}+ ( [4 (\eta\teta b_K)^2 \tbfm_{\ell}b_0^{-1} ] \vee [(c_K \teta)^2 \tbfm_{\ell}a_0^{-1}]) ] \eqsp. 
\end{equation}
The proof is then easily completed.
\end{enumerate}
\end{proof}

\paragraph{Conclusion.}
We can finally conclude this section by providing the proof of \Cref{theo:drift}. Indeed, it only remains to chose carefully the initial parameters $\bfo$ so that Lemma~\ref{lem:drift_one_step} reads $ \rmP_{\bfomega} \Vlyol \leqslant (1-\rho_{\bfomega}  K\delta)^{\ell}\Vlyol +C $ for some constant $C$ (with $\rho_{\bfomega}$  given by \eqref{eq:def_rho_omega_final}), and then \eqref{eq:drift_final} will follow by induction on $n$ and the equivalence between $\Woi$ and $\Vlyol$.

\begin{proof}[Proof of \Cref{theo:drift}]
Taking
  \begin{equation}
    \label{eq:8}
    a_0 = \gamma_0^{2}/2 \eqsp, \quad b_0 = 1 \eqsp, \quad c_0 = \eta \gamma_0\eqsp, \quad \gamma_0 = (1-\eta)/(\delta K) \eqsp, \quad  f_0 = 0 \eqsp,  \quad e_0 = 0 \eqsp.
  \end{equation}
  As mentioned in \Cref{rem:choice_a_b_c}, the first condition in \eqref{eq:condition_c_0} is satisfied for this case.
  In addition, for this choice of parameter, the constants appearing in \Cref{lem:drift_k_step} become:
    \begin{align}
      {C}_{a,1} & = \mtt / (3\times 2^5) \eqsp, \quad {C}_{a,2} = \mtt / (3 \times 2^6) \eqsp, \nC_{a,3} = 20  + \gamma^2_0 \bT(\mtt / (3\times 2^5) + \bT\mtt / (3\times 2^6))/2 \eqsp, \\
      {C}_{b,1} & = \epsilon + \eta \gamma_0(1+10\bdelta) \eqsp, \quad \nC_{b,2,0} = 4[(1+\epsilon) \nC_{b,1} +(1+10\bdelta)(\gamma_0^2 + \nC_{e,1}\epsilon)  ] \eqsp, \\
      \nC_{b,3,0} &= 17 +\gamma^2_0 +\bT(2\nC_{a,3} + 17  \nC_{b,1} + \epsilon \nC_{e,2} + 20\bT(1+10\bdelta)) \eqsp,\\
      \nC_{b,2} & = \nC_{b,2,0} + \eta \gamma_0 \mtt/8 + a_0 \mtt \bT/4   \eqsp, \quad  \nC_{b,3} = \nC_{b,3,0} + 20  + 2 \nC_{a,3} \eqsp,\\
      \nC_{e,1} &= 37[2 + \eta \gamma_0 + \bT(17 \eta \gamma_0 + \gamma_0^2 + \bdelta \nC_{b,3})] \eqsp, \quad \nC_{e,2} = 37(\gamma_0^2+\bT \nC_{a,3}(\bdelta \nC_{a,3} + 2\bT)) \eqsp, \\
 {C}_{f,1} &= \gamma_0^2 \Mtt \bdelta/2 +\bC_{a,3} \Mtt \bdelta \bT+  \mtt \Mtt \bdelta /2 + \Mtt(\eta \gamma_0  + \bar{T} \gamma_0^2/2) \eqsp.
  \end{align}
  Then, with these notations, $\delta_1,\bT_1,\bT_2$ in \eqref{eq:10_applied} can be written as
  \begin{equation}
    \label{eq:10_applied_final}
    \begin{aligned}
      \delta_1 &= 17^{-1} \wedge [\mtt/40] \wedge [ \gamma_0^2 2^{-6}/(48 \epsilon \nC_{e,1} + \nC_{b,3}(1+10\bdelta))] \eqsp, \\
    T_1 & = 4^{-1}\parentheseDeux{17^{-1} \wedge \nC_{b,1}^{-1} \wedge \nC_{b,2,0}^{-1/2}}\eqsp, \\
  \quad 
    T_2&  =4^{-1} \defEnsLigne{[2^{-5}/(1+10\bdelta)]\wedge[ 2^4 \gamma^2_0 /\nC_{b,2}]\wedge \gamma_0/2^4 \wedge 2} \eqsp.
  \end{aligned}
\end{equation}
and the condition \eqref{eq:condition_epsilon_e_0}-\eqref{eq:condition_epsilon_e_0_1} are equivalent to 
\begin{equation}
  \label{eq:condition_epsilon_e_0_final}
 \epsilon \leq \epsilon_1 \eqsp, \qquad \epsilon_1 =  \gamma^2_0 \mtt 2^{-6}/ [1 + 48 \nC_{e,1}]  \eqsp, 
\end{equation}
and 
\begin{equation}
  \label{eq:condition_epsilon_e_0_1_final}
\delta \leq \delta_2 \eqsp, \qquad \delta_2  = \gamma_0 \mtt (3\times 2^8)^{-1}/[  20 +\gamma_0^2 \bT(\nC_{a,1} + \bT\nC_{a,2})/2 ] \eqsp.
\end{equation}
In addition,
\begin{equation}
  \label{eq:bound_b_k_eK}
  b_K \leq \bar{b} \eqsp, \quad \bar{b} = 2 + \bT \bdelta C_{b,3} \eqsp, \quad e_K \leq K\delta \bar{e} \eqsp, \quad \bar{e}  = ( \nC_{e,1} + \bdelta \nC_{e,2}) \eqsp.
\end{equation}

Furthermore, the second line of \eqref{eq:condition_c_0} is equivalent to $\delta \leq 1/(7\times 8)$ and
\begin{equation}
  \label{eq:condi_k_delta_T_3}
  K\delta  \leq T_3 \eqsp, \qquad T_3= 4^{-1} \gamma_0 /[(1-\eta^2)^{1/2}\nC_{b,1} + \nC_{b,2}] \eqsp.
\end{equation}

Then, the proof is  an easy consequence of \Cref{lem:drift_one_step} with $\rho_{\bfomega}$  given by \eqref{eq:def_rho_omega_final}. Condition \eqref{eq:verif_epsi_momentn} is equivalent using \eqref{eq:bound_b_k_eK} to 
\begin{equation}
  \label{eq:def_epsilon_2}
     \epsilon \leq \epsilon_2 \eqsp, \quad  \epsilon_2 = [\ell^{1/(\ell-1)}( 1-\bar{\rho}_{\bfomega}K\delta/4)] \wedge [  (8 \times 24)^{-1}\bar{\rho}_{\bfomega}/\{ \bar{e}(1+4/({\rho}_{\bfomega}\bT))^{(\ell-1)/\ell}\}]\eqsp.
   \end{equation}
   and \eqref{eq:22}-\eqref{eq:bC2} can be written using \eqref{eq:bound_b_k_eK} and $1+\eta \leq 2$, as 
   \begin{align}
  \label{eq:22_final}
     D_1 &=  16 \gamma_0 \rho_{\bfomega}^{-1} \{\bareta^2 + \bT^2 \gamma_0^2 \}\eqsp, \quad  D_2  = 32 \times 2 \bareta^2  \gamma_0/\rho_{\bfomega}\eqsp,\\
       E_{\ell}  &= \delta K F_{\ell}   \eqsp, \quad F_{\ell} = 8 \gamma_0 [\bareta^2 (2+ C_{b,3} \bdelta \bT) D_2 + (\gamma_0^2 \bareta^2 + \gamma_0^4 \bT^2)D_1] \tbfm_{\ell} \eqsp, \\
  \label{eq:bC2_dinfal}
     C_{2} &= (1+{\rho}_{\bfomega}\bT/4) [C_{1}+ ( [8 \gamma_0 \bT\bareta^2 (2+ C_{b,3} \bdelta \bT) ] \vee [16 \gamma_0\bT ( \bareta^2 + \gamma_0^2 \bT^2) ])] \\
     & \qquad \qquad + 24 \bT (C_{e,1} + \bdelta C_{e,2})(4+ \rho_{\bfomega})^{(\ell-1)/\ell} \eqsp, \\
     C_1 & = [2 \{(1+\bT + \bT^3\mtt/4)\gamma_0^2/2 + 2\bareta(1+\bT^2\mtt/8) \gamma_0+2C_{a,3} \bdelta\bT \}/\gamma_0^2] \\
     & \qquad \qquad \vee [2+2\gamma_0 \bareta + \bT \gamma_0^2/2+\bdelta \bT C_{b,3}] \eqsp,
   \end{align}
   where $\tbfm_{\ell}$ is defined in     \eqref{eq:def_tbfm}.
   It completes the proof of \eqref{eq:drift_final} with
   \begin{equation}
   \begin{aligned}
     C_{1,\ell} & = C_2 \eqsp,\\
     C_{2,\ell} & =  2^{\ell} \floor{\ell/2} F_{\ell}[(a_0 D_1 + b_0 (D_2+2D_1) + c_0 D_1^{1/2}D_2^{1/2})^\ell + \bT^{\ell}F_{\ell}^\ell]  \eqsp, \\
C_{3,\ell}  &=  2\times 12^{\ell-1}  \rho_{\bfomega}^{1-\ell}(4+\rho_{\bfomega}\bT)^{\ell-1}[4^{\ell}\gamma_0^{\ell}(2+ \bdelta \bT C_{b,3})\bfm_{2\ell} + 8^{\ell-1}\epsilon^{\ell}(C_{e,1}+\bdelta C_{e,2})^{\ell}\bfm_{2\ell}+ C_{f,1}] \eqsp,
\end{aligned}
\label{eq:def_C_ell_final}
 \end{equation}
   where  $\bfm_{\ell'}$ is the $\ell'$-th moment of the zero-mean one-dimensional Gaussian distribution with variance $1$. 

\end{proof}

\end{appendix}

\end{document}